\documentclass[a4paper,11pt,reqno]{amsart}

\usepackage[margin=1in,includehead,includefoot]{geometry}
\usepackage[utf8]{inputenc}
\usepackage[T1]{fontenc}
\usepackage{lmodern,microtype}
\usepackage{upref}
\usepackage{mathtools,amssymb,amsthm,amsmath}
\usepackage[foot]{amsaddr}
\usepackage{enumitem}
\usepackage[b]{esvect} 
\usepackage[textsize=footnotesize,disable]{todonotes}
\usepackage{graphicx}
\usepackage{hyperref}
\usepackage{wrapfig}
\usepackage{xcolor}
\usepackage{colortbl}
\usepackage{pgfplots}
\usepackage{array}
\usepackage[margin=1cm]{caption}
\usepackage{bookmark}
\usepackage{mycommands}

\graphicspath{{./graphics/}}

\makeatletter
\let\old@setaddresses\@setaddresses
\def\@setaddresses{\bigskip{\parindent 0pt\let\scshape\relax\let\ttfamily\relax\old@setaddresses}}
\makeatother

\newtheorem{theorem}{Theorem}[section]

\newtheorem{proposition}[theorem]{Proposition}
\newtheorem{lemma}[theorem]{Lemma}
\theoremstyle{remark}

\newtheorem{remark}{Remark}


\hypersetup{
  pdftitle={The Hamilton compression of highly symmetric graphs},
  pdfauthor={Petr Gregor, Arturo Merino, Torsten M\"utze}
}

\newcommand{\torsten}[1]{\todo[inline,color=red!40]{\textbf{Torsten:} #1}}

\title{The Hamilton compression of highly symmetric graphs}

\author{Petr Gregor}
\address[Petr Gregor]{Department of Theoretical Computer Science and Mathematical Logic, Charles University, Prague, Czech Republic}
\email{gregor@ktiml.mff.cuni.cz}

\author{Arturo Merino}
\address[Arturo Merino]{Department of Computer Science, University of Saarland, Germany \& Max Planck Institute for Informatics, Germany}
\email{merino@cs.uni-saarland.de}

\author{Torsten M\"utze}
\address[Torsten M\"utze]{Department of Computer Science, University of Warwick, United Kingdom \& Department of Theoretical Computer Science and Mathematical Logic, Charles University, Prague, Czech Republic}
\email{torsten.mutze@warwick.ac.uk}

\thanks{An extended abstract of this work appeared in the Proceedings of the 47th International Symposium on Mathematical Foundations of Computer Science (MFCS~2022).}
\thanks{This work was supported by Czech Science Foundation grant GA~22-15272S, and by German Science Foundation grant~413902284.
Furthermore, Petr Gregor was partially supported by ARRS project P1-0383, and Arturo Merino was supported by ANID Becas Chile 2019-72200522.}

\begin{document}
\begin{abstract}
We say that a Hamilton cycle~$C=(x_1,\ldots,x_n)$ in a graph~$G$ is $k$-symmetric, if the mapping $x_i\mapsto x_{i+n/k}$ for all $i=1,\ldots,n$, where indices are considered modulo~$n$, is an automorphism of~$G$.
In other words, if we lay out the vertices $x_1,\ldots,x_n$ equidistantly on a circle and draw the edges of~$G$ as straight lines, then the drawing of~$G$ has $k$-fold rotational symmetry, i.e., all information about the graph is compressed into a $360^\circ/k$ wedge of the drawing.
The maximum~$k$ for which there exists a $k$-symmetric Hamilton cycle in~$G$ is referred to as the \emph{Hamilton compression of~$G$}.
We investigate the Hamilton compression of four different families of vertex-transitive graphs, namely hypercubes, Johnson graphs, permutahedra and Cayley graphs of abelian groups.
In several cases we determine their Hamilton compression exactly, and in other cases we provide close lower and upper bounds.
The constructed cycles have a much higher compression than several classical Gray codes known from the literature.
Our constructions also yield Gray codes for bitstrings, combinations and permutations that have few tracks and/or that are balanced.
\end{abstract}


\maketitle

\vspace{-5mm}

\section{Introduction}
\label{sec:intro}

\begin{wrapfigure}{r}{0.35\textwidth}
\flushright
\vspace{-10mm}
\includegraphics[scale=0.8]{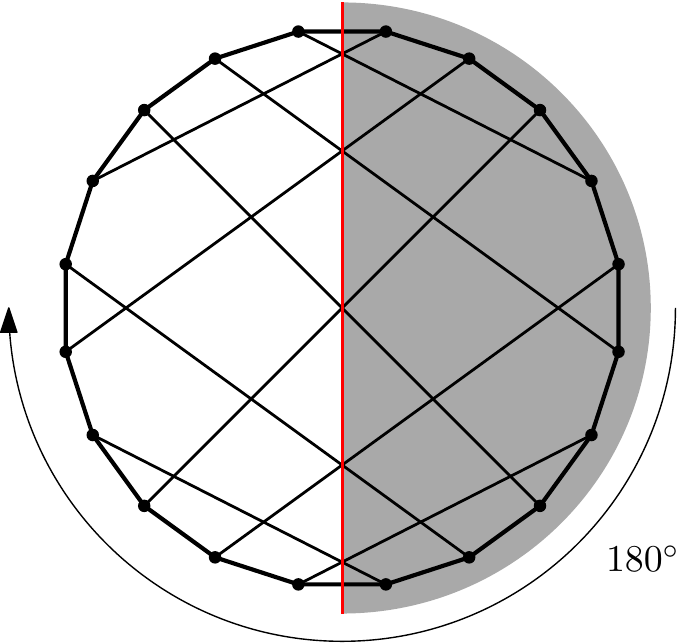}
\end{wrapfigure}
A \emph{Hamilton cycle} in a graph is a cycle that visits every vertex of the graph exactly once.
This concept is named after the Irish mathematician and astronomer Sir William Rowan Hamilton (1805--1865), who invented the Icosian game, in which the objective is to find a Hamilton cycle along the edges of the dodecahedron.
The figure on the right shows the dodecahedron with a Hamilton cycle on the circumference.
Hamilton cycles have been studied intensively from various different angles, such as graph theory (necessary/sufficient conditions, packing and covering etc.~\cite{MR1106528,MR1974368,MR3143857,MR2889513}), optimization (shortest tours, approximation~\cite{tsp_book}), algorithms (complexity~\cite{MR519066}, exhaustive generation~\cite{MR1491049,muetze:22}) and algebra (Cayley graphs~\cite{MR762322,MR1405010,MR2548568,MR2548567}). 
In this work we introduce a new graph parameter that quantifies how symmetric a Hamilton cycle in a graph can be.
For example, the cycle in the dodecahedron shown on the right is 2-symmetric, as the drawing has 2-fold (i.e., $360^\circ/2=180^\circ$) rotational symmetry.

\subsection{Hamilton cycles with rotational symmetry}
\label{sec:introsym}

Formally, let $G=(V,E)$ be a graph with $n$ vertices.
We say that a Hamilton cycle $C=(x_1,\ldots,x_n)$ is \emph{$k$-symmetric} if the mapping $f:V\rightarrow V$ defined by $x_i\mapsto x_{i+n/k}$ for all $i=1,\ldots,n$, where indices are considered modulo~$n$, is an automorphism of~$G$.
In this case we have
\begin{equation}
\label{eq:CP}
C=P,f(P),f^2(P),\ldots,f^{k-1}(P) \quad \text{for the path} \quad P:=(x_1,\ldots,x_{n/k}).
\end{equation}
The idea is that the entire cycle~$C$ can be reconstructed from the path~$P$, which contains only a $1/k$-fraction of all vertices, by repeatedly applying the automorphism~$f$ to it.
In other words, if we lay out the vertices $x_1,\ldots,x_n$ equidistantly on a circle, and draw edges of~$G$ as straight lines, then we obtain a drawing of~$G$ with $k$-fold rotational symmetry, i.e., $f$ is a rotation by $360^\circ/k$; see Figure~\ref{fig:p4} for more examples.
We refer to the maximum~$k$ for which the Hamilton cycle~$C$ of~$G$ is $k$-symmetric as the \emph{compression factor of~$C$}, and we denote it by $\kappa(G,C)$.

\begin{figure}
\makebox[0cm]{ 
\begin{tabular}{cc}
\includegraphics[page=1]{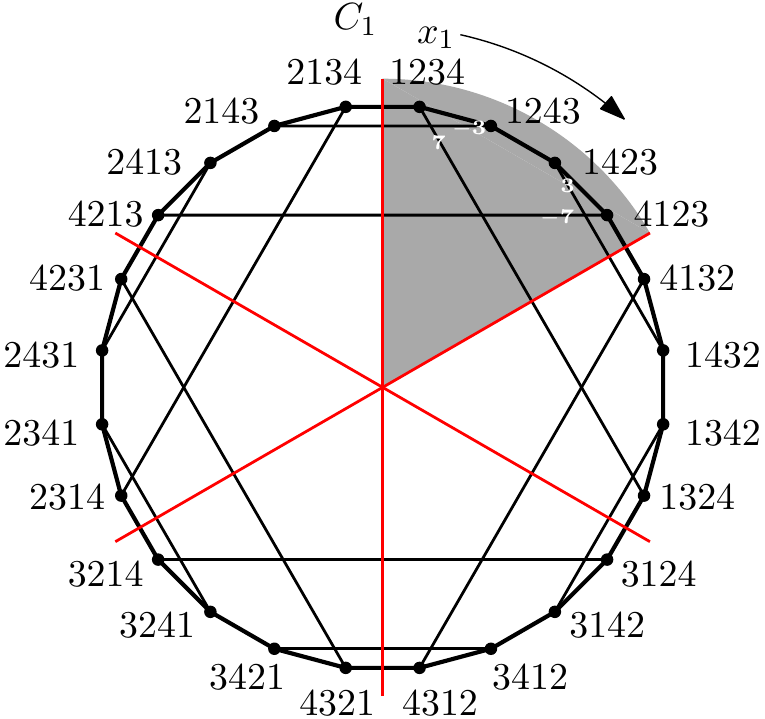} &
\includegraphics[page=2]{p4} \\
(a) $\kappa(\Pi_4,C_1)=6$, $d=(7,-3,3,-7)^6$ & (b) $\kappa(\Pi_4,C_2)=6$, $d=(-7,7,5,-5)^6$ \\
& \\
\includegraphics[page=3]{p4} &
\includegraphics[page=4]{p4} \\
\begin{minipage}{8.5cm}(c) $\kappa(\Pi_4,C_3)=2$, \\ \hspace*{5mm} $d=(-7,-11,11,-5,-7,7,5,-11,11,7,5,-5)^2$\end{minipage} &
\begin{minipage}{8cm}(d) $\kappa(\Pi_4,C_4)=1$, \\ \hspace*{5mm} $d=(-9,9,7,-5,-7,3,-11,5,-3,-7,-9,3,$ \\ \hspace*{14mm} $-5,5,-3,9,7,3,-5,11,-3,7,5,-7)$\end{minipage}
\end{tabular}
}
\caption{Hamilton cycles $C_1,\ldots,C_4$ in the 4-permutahedron~$\Pi_4$ with different LCF sequences and compression factors.}
\label{fig:p4}
\end{figure}

\subsection{Connection to LCF notation}

There is yet another interesting interpretation of the compression factor in terms of the LCF notation of a graph, which was introduced by Lederberg as a concise method to describe 3-regular Hamiltonian graphs (such as the dodecahedron).
It was later improved by Coxeter and Frucht (see~\cite{MR463029}), and dubbed LCF notation, using the initials of the three inventors.
The idea is to describe a 3-regular Hamiltonian graph by considering one of its Hamilton cycles~$C=(x_1,\ldots,x_n)$.
Each vertex~$x_i$ has the neighbors $x_{i-1}$ and~$x_{i+1}$ (modulo~$n$) in the graph, plus a third neighbor~$x_j$, which is $d_i:=j-i$ (modulo~$n$) steps away from~$x_i$ along the cycle.
The \emph{LCF sequence} of~$G$ is the sequence $d=(d_1,\ldots,d_n)$, where each $d_i$ is chosen so that $-n/2<d_i\le n/2$.
Clearly, we also have $d_i\notin\{-1,0,+1\}$.
Note that if $C$ is $k$-symmetric, then the LCF sequence~$d$ of~$G$ is $k$-periodic, i.e., it has the form $d=(d_1,\ldots,d_{n/k})^k$, where the $k$ in the exponent denotes $k$-fold repetition; see Figure~\ref{fig:p4}.
While LCF notation is only defined for 3-regular graphs, we can easily extend it to arbitrary graphs with a Hamilton cycle~$C=(x_1,\ldots,x_n)$, by considering a sequence of sets $D=(D_1,\ldots,D_n)$, where $D_i$ is the set of distances to all neighbors of~$x_i$ along the cycle except $x_{i-1}$ and~$x_{i+1}$; see Figure~\ref{fig:comp4}~(a)+(d).
As before, if $C$ is $k$-symmetric, then the corresponding sequence~$D$ is $k$-periodic, i.e., it has the form $D=(D_1,\ldots,D_{n/k})^k$.
Frucht~\cite{MR463029} writes:
\setlength{\leftmargini}{6mm}
\begin{quote}
`What happens with the LCF notation if we replace one hamiltonian circuit by another one? The answer is: nearly everything can happen! Indeed the LCF notation for a graph can remain unaltered or it can change completely [...] In such cases we should choose of course the shortest of the existing LCF notations.'
\end{quote}
This observation is illustrated in Figure~\ref{fig:p4}, which shows four different Hamilton cycles of the same graph~$G$ that have different LCF sequences and compression factors.

\subsection{Hamilton compression}

Frucht's suggestion is to search for a Hamilton cycle~$C$ in~$G$ whose compression factor~$\kappa(G,C)$ is as large as possible.
Formally, for any graph~$G$ we define
\begin{equation}
\label{eq:kappa}
\kappa(G):=\max\big\{\kappa(G,C)\mid \text{$C$ is a Hamilton cycle in~$G$}\big\},
\end{equation}
and we refer to this quantity as the \emph{Hamilton compression of~$G$}.
If $G$ has no Hamilton cycle, then we define $\kappa(G):=0$.
While the maximization in~\eqref{eq:kappa} is simply over all Hamilton cycles in~$G$, and the automorphisms arise as possible rotations of those cycles, this definition is somewhat impractical to work with.
In our arguments, we rather consider all automorphisms of~$G$, and then search for a Hamilton cycle that is $k$-symmetric under the chosen automorphism.
Specifically, proving a lower bound of~$\kappa(G)\ge k$ amounts to finding an automorphism~$f$ of~$G$ and a $k$-symmetric Hamilton cycle under this~$f$.
To prove an upper bound of~$\kappa(G)<k$, we need to argue that there is no $k$-symmetric Hamilton cycle in~$G$, for any choice of automorphism~$f$.

By what we said in the beginning, the quantity $\kappa(G)$ can be seen as a measure for the nicest (i.e., most symmetric) way of drawing the graph~$G$ on a circle.
Thus, our paper contains many illustrations that convey the aesthetic appeal of this problem.

\begin{figure}[b]
\makebox[0cm]{ 
\begin{tabular}{cc}
\includegraphics{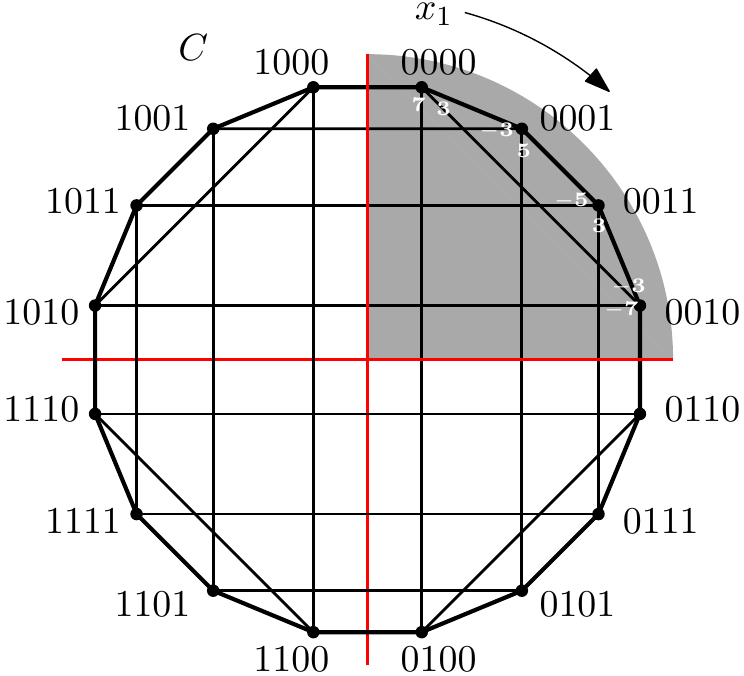} &
\includegraphics{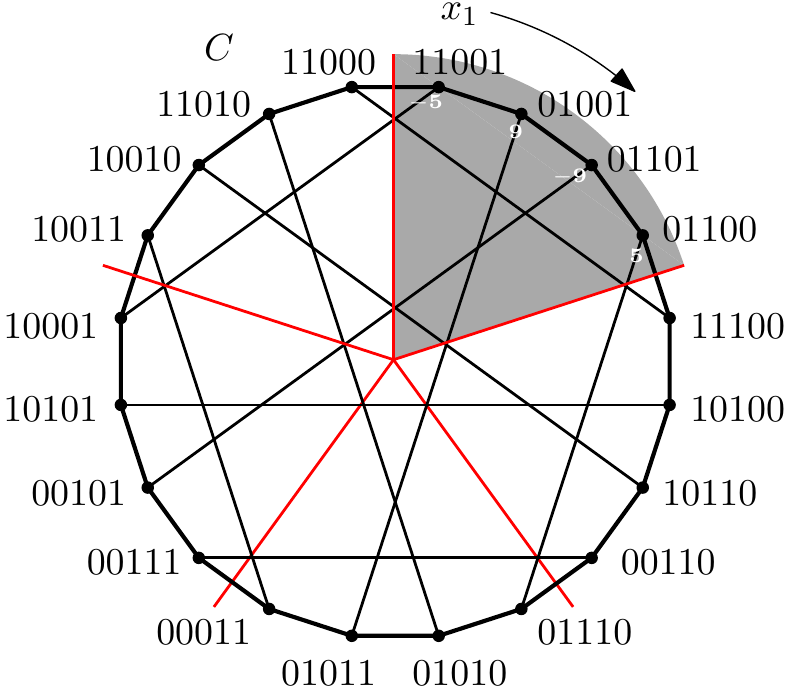} \\
\begin{minipage}{8cm}(a) $\kappa(Q_4,C)=4$, \\ \hspace*{5mm} $D=(\{3,7\},\{-3,5\},\{-5,3\},\{-3,-7\})^4$ \end{minipage}
 & (b) $\kappa(M_5,C)=5$, $d=(-5,9,-9,5)^5$ \\
 & \\
\includegraphics{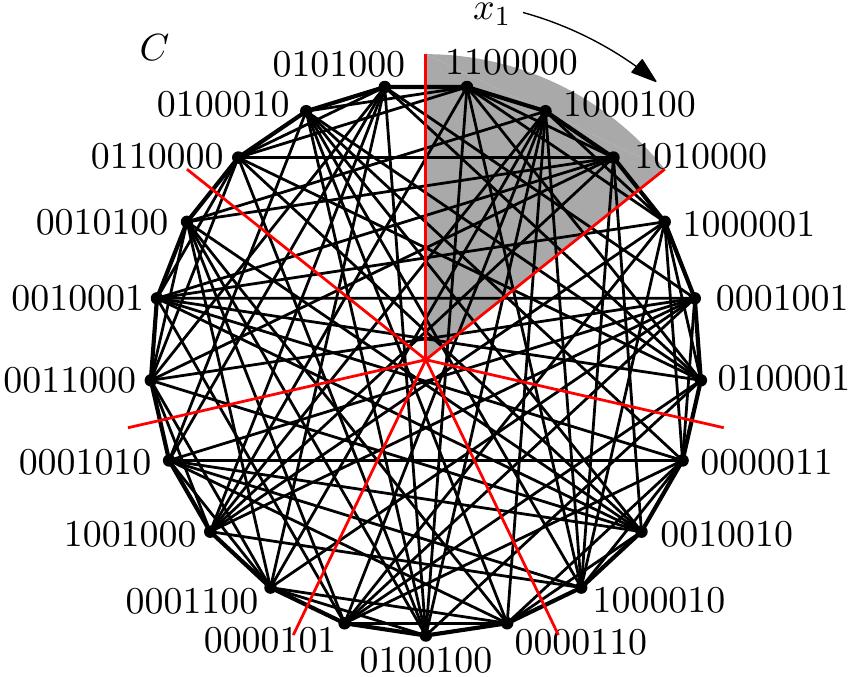} &
\includegraphics{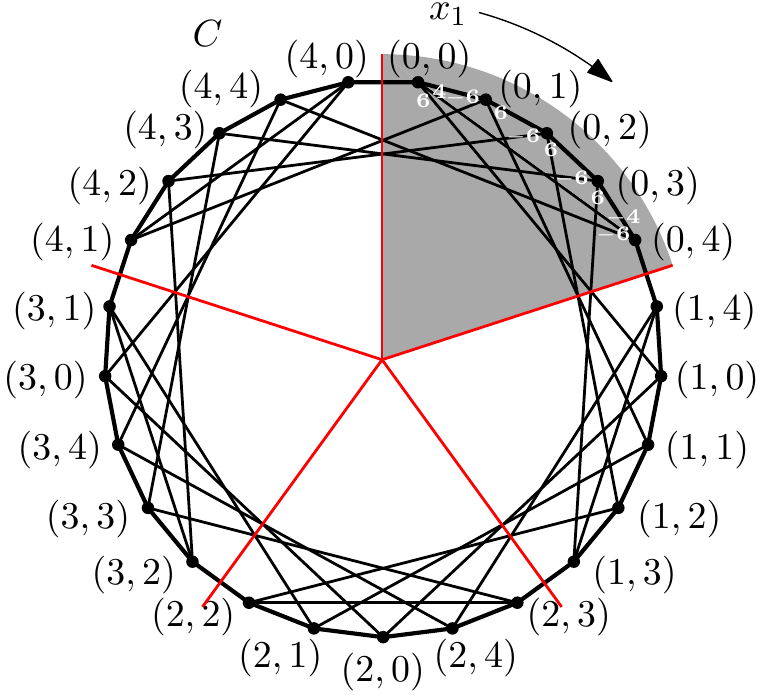} \\
(c) $\kappa(J_{7,2},C)=7$ &
\begin{minipage}{8cm}(d) $\kappa(\mathbb{Z}_5^2,C)=5$, \\ \hspace*{5mm} $D=(\{4,6\},\{-6,6\},\{-6,6\},\{-6,6\},\{-4,-6\})^5$ \end{minipage}
\end{tabular}
}
\caption{Symmetric Hamilton cycles in the (a)~4-cube; (b)~middle levels of the 5-cube; (c)~Johnson graph~$J_{7,2}$; (d)~abelian Cayley graph~$(\mathbb{Z}_5^2,\{(0,1),(1,0)\})$.}
\label{fig:comp4}
\end{figure}

\subsection{Easy observations and bounds}
\label{sec:easy}

We start to collect a few basic observations about the quantity~$\kappa(G)$.
Trivially, we have $0\le \kappa(G)\le n$, where $n$ is the number of vertices of~$G$.
The upper bound~$n$ can be improved to
\begin{equation}
\label{eq:kappa-UB}
\kappa(G)\le \max_{f\in \Aut(G)}\ord(f),
\end{equation}
where $\Aut(G)$ is the automorphism group of~$G$, and $\ord(f)$ is the order of~$f$.
An immediate consequence of~\eqref{eq:CP} is that all orbits of the automorphism~$f$ must have the same size~$n/k$, and the path $P=(x_1,\ldots,x_{n/k})$ visits every orbit exactly once.
This can be used to improve~\eqref{eq:kappa-UB} further by restricting the maximization to automorphisms from~$\Aut(G)$ whose orbits all have the same size.
Furthermore, as $k$ must divide~$n$, we obtain that $\kappa(G)\in\{0,1,n\}$ for prime~$n$.

Clearly, every Hamilton cycle of a graph~$G$ is 1-symmetric, by taking the identity mapping $f=\ide$ as automorphism.
Consequently, we have $\kappa(G)\ge 1$ for any Hamiltonian graph.
On the other hand, if~$G$ is Hamiltonian and highly symmetric, i.e., if it has a rich automorphism group, then intuitively $G$ should have a large value of~$\kappa(G)$, i.e., it should admit highly symmetric Hamilton cycles.
For example, for the cycle~$C_n$ on $n$ vertices and the complete graph~$K_n$ on $n$ vertices we have $\kappa(C_n)=\kappa(K_n)=n$.
More generally, note that $\kappa(G)=n$ if and only if~$G$ is a special circulant graph, namely a Cayley graph of a cyclic group $\mathbb{Z}_n$ for which the generating set contains at least one element coprime with $n$.

\subsection{Our results}
\label{sec:results}

Vertex-transitive graphs are a prime example of highly symmetric graphs.
A graph is \emph{vertex-transitive} if for any two vertices there is an automorphism that maps the first vertex to the second one.
In other words, the automorphism group of the graph acts transitively on the vertices.
In this paper we investigate the Hamilton compression~$\kappa(G)$ of four families of vertex-transitive graphs~$G$, namely hypercubes, Johnson graphs, permutahedra, and Cayley graphs of abelian groups.
Note that in the following definitions and the rest of the paper we use the letter $n$ to denote a graph parameter instead of the number of vertices used in Sections~\ref{sec:introsym}-\ref{sec:easy}.
The \emph{$n$-dimensional hypercube~$Q_n$}, or \emph{$n$-cube} for short, has as vertices all bitstrings of length~$n$, and an edge between any two strings that differ in a single bit; see Figure~\ref{fig:comp4}~(a).
The \emph{Johnson graph~$J_{n,k}$} has as vertices all bitstrings of length~$n$ with fixed Hamming weight~$k$, and an edge between any two strings that differ in a transposition of a 0 and~1; see Figure~\ref{fig:comp4}~(c).
The \emph{$n$-permutahedron~$\Pi_n$}, has as vertices all permutations of~$[n]:=\{1,\ldots,n\}$, and an edge between any two permutations that differ in an adjacent transposition, i.e., a swap of two neighboring entries of the permutations in one-line notation; see Figure~\ref{fig:p4}.
Cayley graphs of abelian groups will be introduced formally in Section~\ref{sec:prelim-cayley}; see Figure~\ref{fig:comp4}~(d).
Note that the hypercube is isomorphic to a Cayley graph of the abelian group~$\mathbb{Z}_2^n$.

Hamilton cycles with various additional properties in the aforementioned families of graphs have been the subject of a long line of previous research under the name of \emph{combinatorial Gray codes}~\cite{MR1491049,muetze:22}.
We will see that some classical constructions of such cycles have a non-trivial small compression factor, and we construct cycles with much higher compression factor that we show to be optimal or near-optimal.
Along the way, many interesting number-theoretic and algebraic phenomena arise.

\subsubsection{Hypercubes}
\label{sec:results-cubes}

One of the classical constructions of a Hamilton cycle in~$Q_n$ is the well-known \emph{binary reflected Gray code (BRGC)}~\cite{gray:patent}.
This cycle in~$Q_n$ is defined inductively by $\Gamma_0:=\varepsilon$ and $\Gamma_n:=0\Gamma_{n-1},1\lvec{\Gamma_{n-1}}$ for all~$n\ge 1$, where $\varepsilon$ is the empty sequence and $\lvec{\Gamma_{n-1}}$ denotes the reversal of the sequence~$\Gamma_{n-1}$; see Figure~\ref{fig:comp4}~(a) and Figure~\ref{fig:q8}~(a).
In words, the cycle $\Gamma_n$ is obtained by concatenating the vertices of~$\Gamma_{n-1}$ prefixed by~0 with the vertices of~$\Gamma_{n-1}$ in reverse order prefixed by~1.
It turns out that the BRGC~$\Gamma_n$ has only compression~$\kappa(Q_n,\Gamma_n)=4$ for $n\ge 2$, which is not optimal (Proposition~\ref{prop:brgc}).
We construct new Hamilton cycles in~$Q_n$ with compression~$\kappa(Q_n)=2^{\lceil \log_2 n\rceil}$ for $n\ge 3$, which is the optimal value (Theorem~\ref{thm:kappa-Qn}); see Figure~\ref{fig:q8}~(b).
Note that $n\le \kappa(Q_n)<2n$, in particular $\kappa(Q_n)=\Theta(n)$, i.e., the optimal compression grows linearly with~$n$.

\subsubsection{Johnson graphs and relatives}
\label{sec:results-johnson}

Our definition of Hamilton compression is inspired by a variant of the well-known middle levels problem raised by Knuth in Problem~56 in Section~7.2.1.3 of his book~\cite{MR3444818}.
Let $M_{2n+1}$ denote the subgraph of~$Q_{2n+1}$ induced by all bitstrings with Hamming weight~$n$ or~$n+1$.
In other words, $M_{2n+1}$ is the subgraph of the cover graph of the Boolean lattice of dimension~$2n+1$ induced by the middle two levels.
There is a natural automorphism of~$M_{2n+1}$ all of whose orbits have the same size, namely cyclic left-shift of the bitstrings by one position.
Knuth asked whether $M_{2n+1}$ admits a $(2n+1)$-symmetric Hamilton cycle under this automorphism, and he rated this the hardest open problem in his book, with a difficulty rating of 49/50.
Such cycles are shown in Figure~\ref{fig:comp4}~(b) and Figure~\ref{fig:mlc7}~(a) for the graphs $M_5$ and $M_7$, respectively.
Knuth's problem was answered affirmatively in full generality in~\cite{MR4262479}, which establishes the lower bound $\kappa(M_{2n+1})\ge 2n+1$.
We show that this is at most a factor of~2 away from optimality, by proving the upper bound $\kappa(M_{2n+1})\le 2(2n+1)$ (Theorem~\ref{thm:kappa-middle}).
Interestingly, it seems that both bounds can be improved.

For the Johnson graph~$J_{n,k}$, we show that $\kappa(J_{n,k})=n$ if $n$ and $k$ are coprime (Theorem~\ref{thm:kappa-Jnk}).
In the other cases we establish bounds for~$\kappa(J_{n,k})$ that are at most by a factor of~4 apart, and for large~$n$ we obtain $\kappa(J_{n,k})=(1-o(1))n$ for $n\neq 2k$.

\subsubsection{Permutahedra}
\label{sec:results-perm}

Another classical Gray code is produced by the \emph{Steinhaus-Johnson-Trotter (SJT) algorithm}, which generates permutations by adjacent transpositions.
This algorithm computes a Hamilton cycle in~$\Pi_n$, which can be described inductively as follows: $\Lambda_1:=1$ and for all $n\ge 2$ the cycle $\Lambda_n$ is obtained from~$\Lambda_{n-1}$ by replacing each permutation of length~$n-1$ by the $n$ permutations given by inserting $n$ in every possible position, alternatingly from right to left or vice versa; see Figure~\ref{fig:p4}~(a) and Figure~\ref{fig:p5}~(a).
It turns out that the SJT cycle~$\Lambda_n$ has only compression~$\kappa(\Pi_n,\Lambda_n)=3$ for $n\ge 5$, which is not optimal (Proposition~\ref{prop:perm}).
We construct new Hamilton cycles in~$\Pi_n$ whose compression is at most a factor of~2 away from the optimum compression~$\kappa(\Pi_n)=e^{(1+o(1))\sqrt{n\ln n}}$ (Theorem~\ref{thm:kappa-Pin}); see Figure~\ref{fig:p5}~(b)+(c).
The growth of the optimum compression is determined by \emph{Landau's function}, and it is mildly exponential.
Moreover, we achieve the optimal compression in infinitely many cases, in particular for the following values of~$n\le 100$: $n=3,4,5,15,22,46,49,51,52,53,55,68,69,72,73,74,75,80,82,85,87,88,89,91,92,93,96,97,99,100$.

\subsubsection{Abelian Cayley graphs}
\label{sec:results-cayley}

A classical folklore result asserts that every Cayley graph of an abelian group has a Hamilton cycle.
The Chen-Quimpo theorem~\cite{MR641233} asserts that in fact much stronger Hamiltonicity properties hold.
It is thus natural to ask whether Cayley graphs of abelian groups have highly symmetric Hamilton cycles.
It turns out that not all abelian Cayley graphs admit a Hamilton cycle with non-trivial compression.
In particular, we show that toroidal grids~$\mathbb{Z}_p\times\mathbb{Z}_q$ for two distinct odd primes~$p,q$ have only compression~1 (Theorem~\ref{thm:odd-order}~(i)).
In contrast to that, we prove that if the order of the abelian group is even or divisible by a square greater than~1, then the Cayley graph admits a Hamilton cycle with compression at least~2 (Theorems~\ref{thm:odd-order}~(ii) and~\ref{thm:comp2}).

\subsection{Related problems}

We proceed to discuss some applications of our results to closely related problems.

\subsubsection{Lov\'asz' conjecture}

A well-known question of Lov\'asz'~\cite{MR0263646} asks whether there are infinitely many vertex-transitive graphs that do not admit a Hamilton cycle.
So far only five such graphs are known, namely $K_2$, the Petersen graph, the Coxeter graph, and the graphs obtained from the latter two by replacing every vertex by a triangle.
Vertex-transitive graphs have a lot of automorphisms, and we may take the quantity $\kappa(G)$ as a measure of how strongly $G$ is Hamiltonian.
In particular, Lov\'asz' question may be rephrased as `Are there infinitely many vertex-transitive graphs~$G$ with~$\kappa(G)=0$?'
More generally, we may ask: `Are there infinitely many vertex-transitive graphs~$G$ with~$\kappa(G)=k$, for each fixed integer~$k$?'
We may ask the same question more restrictively for Cayley graphs or non-Cayley graphs.
From our results mentioned in Section~\ref{sec:results-cayley} we obtain an infinite family of Cayley graphs~$G$ with $\kappa(G)=1$.
In a follow-up work to this paper, Kutnar, Maru{\v{s}}i{\v{c}}, and Razafimahatratra~\cite{kutnar-marusic-raza:23} answered this question affirmatively for Cayley graphs and any fixed integer~$k\geq 2$, and also for non-Cayley graphs and~$k=1$.

\begin{wrapfigure}{r}{0.4\textwidth}
\centering
\vspace{-5mm}
\includegraphics[scale=0.8]{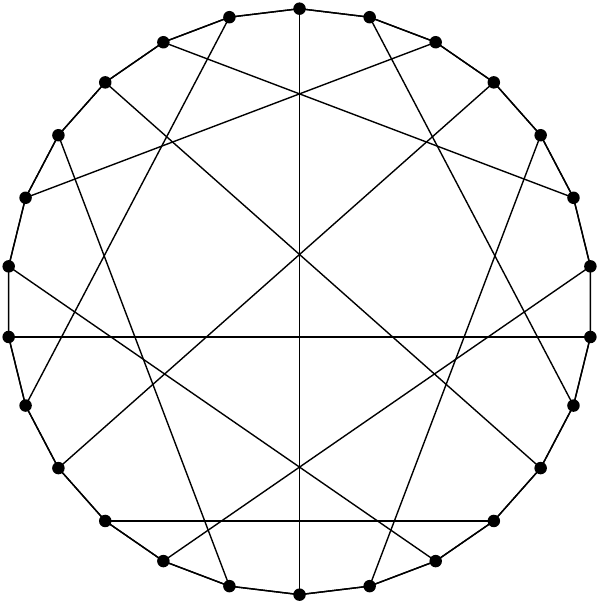}
\caption{One of the smallest vertex-transitive non-Cayley graphs~$G$ with~$\kappa(G)=1$.}
\label{fig:g26}
\end{wrapfigure}
Computer experiments show that the smallest vertex-transitive non-Cayley graphs~$G$ with~$\kappa(G)=1$ have 26 vertices, and one of them is shown in Figure~\ref{fig:g26} (its ID in the House of Graphs database is~36346).

The path~$P$ in~\eqref{eq:CP} is a Hamilton path in the quotient graph~$G/f$ obtained by collapsing each orbit of~$f$ into a single vertex.
The idea of constructing a Hamilton cycle in~$G$ by constructing a Hamilton cycle in the much smaller graph~$G/f$ that is then `lifted' to the full graph is well known in the literature, and has been used to solve some special cases of Lov\'asz' problem affirmatively; see e.g.~\cite{MR1020643,MR2548541,MR2548567,MR4328721,MR4262479}.
It is particularly useful for computer searches, as it reduces the search space dramatically.

\subsubsection{$t$-track and balanced Gray codes}
\label{sec:track}

We say that a sequence~$C$ of strings of length~$n$ consists of $t$ \emph{tracks} if in the $|C|\times n$ matrix corresponding to $C$ there are $t$ columns such that every other column is a cyclically shifted copy of one of these columns.
For example, the Gray code shown in Figure~\ref{fig:q8}~(c) has two tracks, each consisting of four cyclically shifted columns of bits.
This property is relevant for applications, as it saves hardware when implementing Gray-coded rotary encoders.
Instead of using $n$ tracks and $n$ reading heads aligned at the same angle (each reading one track), one can use only $t$ tracks, and place some of the $n$ reading heads at appropriately rotated positions.

Hiltgen, Paterson, and Brandestini~\cite{DBLP:journals/tit/HiltgenPB96} showed that the length of any 1-track cycle in~$Q_n$ must be a multiple of~$2n$.
In particular, such a cycle cannot be a Hamilton cycle unless $n$ is a power of~2.
For the case $n=2^r$, $r\ge 3$, Etzion and Paterson~\cite{MR1445874} showed that there is 1-track cycle of length~$2^n-2n$, and Schwartz and Etzion~\cite{MR1725126} subsequently showed that the length~$2^n-2n$ is best possible.
Taken together, these results show that there is no 1-track Hamilton cycle in~$Q_n$ for any $n\ge 3$.

We complement this negative result by constructing a 2-track Hamilton cycle in~$Q_n$, for every $n$ that is a sum of two powers of~2 (Theorem~\ref{thm:2track}); see Figure~\ref{fig:q8}~(c).
More generally, we obtain $t$-track Hamilton cycles in~$Q_n$ for every $n$ that is a sum of $t\ge 2$ powers of~2 (Theorem~\ref{thm:ttrack}).
In particular, $Q_n$ admits a Hamilton cycle with at most logarithmically many tracks for all~$n$.

From our construction in the Johnson graph~$J_{n,k}$ when $n$ and $k$ are coprime, we obtain 1-track Hamilton cycles that are also \emph{balanced} (Theorem~\ref{thm:johnson-coprime}), i.e., each bit is flipped equally often (cf.~\cite{MR1410880,MR4046775}).

We also construct a 1-track Hamilton cycle in~$\Pi_n^+$, for every odd~$n$, where $\Pi_n^+$ is obtained from~$\Pi_n$ by adding edges that correspond to transpositions of the first and last entry of a permutation (i.e., cyclically adjacent transpositions).
This cycle has the additional property that every transposition appears equally often.
In other words, we obtain a balanced 1-track Gray code for permutations of odd length that uses cyclically adjacent transpositions (Theorem~\ref{thm:1track}).

\subsection{Outline of this paper}

In Section~\ref{sec:prelim} we provide definitions and auxiliary results that will be used later in the paper.
In Sections~\ref{sec:cubes}--\ref{sec:cayley} we prove our results about the Hamilton compression of hypercubes, Johnson graphs and relatives, permutahedra, and abelian Cayley graphs, in that order.
The Hamilton compression~$\kappa(G)$ is a newly introduced graph parameter, so many natural follow-up questions arise.
We collect some of those open problems in Section~\ref{sec:open}.
The proofs of some technical lemmas that interfere with the main exposition are deferred to the appendix.

\section{Preliminaries}
\label{sec:prelim}

We collect definitions, preliminary observations and various results from the literature that we will use in our proofs later.

\subsection{Graphs, groups and permutations}
\label{sec:prelim-graphs}

For a graph~$G$, we write $V(G)$ and $E(G)$ for the vertex set and edge set of~$G$, respectively.
For a sequence~$\Gamma$, we write~$\lvec{\Gamma}$ for the reversed sequence.

All groups considered in this paper are finite.
For standard terminology regarding groups, we refer to the textbook~\cite{MR2286236}.
We denote the group operation of a general group~$G$ multiplicatively, and we write~$e$ for its identity element.
On the other hand, we denote abelian groups additively, writing~$0$ for the identity element and for $k \in \NN, g\in G$ we define $kg:=\underbrace{g+\cdots+g}_{k \text{ times}}$.

Let $G$ be a finite group acting on a set~$X$.
For $f \in G$ let $\langle f \rangle$ denote the (cyclic) subgroup of~$G$ generated by~$f$ and let $\ord(f):=|\langle f \rangle|$ denote the order of~$f$.
The \emph{orbit of an element} $x\in X$ under~$f$ is denoted by $\langle x \rangle_f:=\{f^i(x) \mid i=0,1,\ldots\}$ where the index~$f$ may be omitted whenever it is clear from the context.
We write $O(f)$ for the set of all orbits of~$f$.

Abelian groups are direct sums of cyclic groups~$\mathbb{Z}_n$, which is captured by the following well-known structure theorem.

\begin{theorem}[{\cite[Theorem 5.2]{MR2286236}}]
\label{thm:abelian-struct}
For any finite abelian group~$G$, there are primes $p_1,\ldots, p_\ell\geq 2$ and integers $e_1,\ldots,e_\ell\geq 1$ such that $G \cong \mathbb{Z}_{p_1^{e_1}} \oplus \cdots \oplus \mathbb{Z}_{p_\ell^{e_\ell}}$.
\end{theorem}

For any ordered set~$X$ let $S_X$ denote the symmetry group on~$X$.
For the set $[n]=\{1,\ldots,n\}$ we write $S_n:=S_{[n]}$.
Permutations $x\in S_X$, $X=\{i_1<\cdots<i_n\}$, are denoted in one-line notation $x=x_{i_1}\cdots x_{i_n}$ without commas and parentheses, and in cycle notation with commas and parentheses, for example $342165=(1,3,2,4)(5,6)\in S_6$.

The \emph{parity} of a permutation~$x=x_{i_1}\cdots x_{i_n}$ is the number of its inversions, i.e., the number of pairs $(x_a,x_b)$ with $a<b$ and $x_a>x_b$.
Equivalently, the parity of~$x$ is given by the parity of the number of inversions whose product is~$x$.
As odd cycles are products of an even number of inversions, the parity is also equal to the parity of the number of even cycles in~$x$.
A permutation is \emph{even} or \emph{odd}, if its parity is even or odd, respectively.

\subsection{Hypercubes and Cartesian products}
\label{sec:prelim-cubes}

Given two graphs $G$ and $H$, the \emph{Cartesian product} $G\boxprod H$ has the vertex set $V(G)\times V(H)$ and edges between pairs~$(u,v)$ and~$(u',v')$ if and only if $u=u'$ and $vv'\in E(H)$, or $uu'\in E(G)$ and $v=v'$.

We define a `zigzag' path $P\sqcup Q$ in a Cartesian product~$G\boxprod H$ as follows.
For a path~$P=(u_1,\ldots,u_r)$ in~$G$ and a vertex~$v$ in~$H$ we define
\begin{subequations}
\label{eq:zigzag}
\begin{equation}
(P,v):=\big((u_1,v),(u_2,v),\ldots,(u_r,v)\big).
\end{equation}
Furthermore, for a path $Q=(v_1,\ldots,v_s)$ in~$H$, where $s$ is even, we define
\begin{equation}
P\sqcup Q:=(P,v_1),(\lvec P,v_2), (P,v_3),(\lvec P,v_4),\ldots,(P,v_{s-1}),(\lvec P,v_s).
\end{equation}
\end{subequations}
Clearly, the path $P\sqcup Q$ is a subgraph of~$G\boxprod H$.
Furthermore, it starts at~$(u_1,v_1)$, ends at~$(u_1,v_s)$, and it visits all vertices of~$P\boxprod Q$.

The hypercube~$Q_{n+m}$ defined in Section~\ref{sec:results} can be viewed as the Cartesian product~$Q_n\boxprod Q_m$.

It is well known that for every automorphism~$f$ of~$Q_n$ there is a unique $\pi\in S_n$ and a unique $z\in\mathbb{Z}_2^n$ such that
\begin{equation}
\label{eq:Qaut}
f(x_1\cdots x_n)=x_{\pi(1)}\cdots x_{\pi(n)}+z
\end{equation}
for every $x\in \{0,1\}^n$, where $+$ is the addition from~$\mathbb{Z}_2^n$ (bitwise XOR).
In fact, the automorphism group of the hypercube is $\Aut(Q_n)\cong S_n \ltimes \mathbb{Z}_2^n$, the hyperoctahedral group ($\ltimes$ denotes the inner semidirect product).

We will also need the following result about the automorphism group of certain Cartesian products of graphs.
This result is a special case of Theorem~6.13 from~\cite{MR2817074}.

\begin{lemma}
\label{lem:product-auto}
Let $G_1,\ldots,G_\ell$ be graphs such that $|V(G_1)|,\ldots, |V(G_\ell)|$ are all distinct primes, then $\Aut(G_1 \boxprod \cdots \boxprod G_\ell) \cong \bigtimes_{i=1}^\ell \Aut(G_i)$, where the multiplication on the right-hand side denotes the direct product.
\end{lemma}

\subsection{Johnson graphs and relatives}
\label{sec:prelim-johnson}

For integers $n>k>0$, an \emph{$(n,k)$-combination} is a bitstring of length~$n$ with Hamming weight~$k$.
Recall that the \emph{Johnson graph~$J_{n,k}$} has as vertices all $(n,k)$-combinations, and an edge between any two strings that differ in a transposition of a~0 and~1.
We defined the \emph{middle levels graph~$M_{2n+1}$} as the subgraph of~$Q_{2n+1}$ induced by all bitstrings with Hamming weight~$n$ or~$n+1$, so these are all $(2n+1,n)$-combinations and $(2n+1,n+1)$-combinations.

If $n\neq 2k$ the automorphism group of~$J_{n,k}$ is the symmetric group~$\Aut(J_{n,k})\cong S_n$; see~\cite{MR2116180,MR2801228}.
Every automorphism~$f=\pi$ of~$J_{n,k}$ permutes the entries of a vertex~$x$, i.e., $f(x_1\cdots x_n)=x_{\pi(1)}\cdots x_{\pi(n)}$.

If $n=2k$ the automorphism group of~$J_{n,k}$ is $\Aut(J_{n,k})\cong S_n\times \mathbb{Z}_2$ with $\mathbb{Z}_2\cong\{\ide,\cpl\}$, where $\ide$ is the identity map and $\cpl$ complements all bits; see~\cite{MR2116180,MR3791054}.
Every automorphism of~$J_{n,k}$ is a pair~$f=(\pi,\alpha)$ with $\pi \in S_n$ and $\alpha\in\{\ide,\cpl\}$, and $f$ acts on a vertex $x$ by $f(x_1\cdots x_n)=\alpha(x_{\pi(1)}\cdots x_{\pi(n)})$.

The automorphism group of the middle levels graph~$M_{2n+1}$ is also~$\Aut(M_{2n+1})\cong S_{2n+1}\times \mathbb{Z}_2$ with $\mathbb{Z}_2\cong\{\ide,\cpl\}$; see~\cite{MR2801228}.

To construct symmetric Hamilton cycles in Johnson graphs, we will use the automorphism~$f=(\pi,\ide)$ that cyclically shifts all bits to the left by one position (no complementation is applied).
Formally, $f$ maps $x=x_1\cdots x_n$ to $f(x)=x_2 x_3\cdots x_n x_1$.
The orbits of~$f$ are known as \emph{necklaces}.
Note that if~$n$ and~$k$ are coprime, then every necklace has the same size~$n$.

We will use the following result due to Wang and Savage~\cite{MR1413286} (see also~\cite{MR1761724}).

\begin{theorem}[\cite{MR1413286}]
\label{thm:necklace}
For all $n-1>k>1$, there is a path in~$J_{n,k}$ from~$1^k0^{n-k}$ to~$1^{k-1}010^{n-k-1}$ that visits every necklace exactly once.
\end{theorem}

Note that the end vertices of this path are adjacent, so the path can be completed to a cycle.

\subsection{Permutahedra}
\label{sec:prelim-perm}

As mentioned before, an \emph{adjacent transposition} in~$S_n$ is a permutation that flips two adjacent positions, in cycle notation it is $t_i:=(i,i+1)$ for some $1\le i <n$.
Two permutations $x,y \in S_n$ differ by an adjacent transposition if $x=t_iy$ (equivalently $y=t_ix$) for some $1\le i<n$ where the composed permutations are applied from left to right, i.e., $x_j=y_{t_i(j)}$; see Figure~\ref{fig:p4}.
The \emph{permutahedron} of order~$n$ is the graph~$\Pi_n$ with vertex set~$S_n$ and edge set~$\{\{x,t_ix\} \mid x\in S_n, 1\le i <n\}$; that is, the vertices are all permutations of~$[n]$ and the edges are between any two permutations that differ by an adjacent transposition.
Equivalently, $\Pi_n$ is the Cayley graph of~$S_n$ generated by adjacent transpositions.
Note that our definition of~$\Pi_n$ is equivalent to another definition of the permutahedron sometimes used in geometry where edges connect permutations that differ in a transposition of adjacent values (consider the inverse permutations).
Our definition of~$\Pi_n$ extends straightforwardly to the symmetric group~$S_X$ on any ordered ground set~$X$, and we write~$\Pi_X$ for this graph with vertex set~$S_X$ and edges between pairs of permutations on~$X$ that differ by an adjacent transposition.

Clearly, the graph~$\Pi_n$ is bipartite with partition classes given by the parity of permutations, into sets of equal size for all $n\ge 2$.
Tchuente established the following strong Hamiltonicity property of~$\Pi_n$.

\begin{theorem}[\cite{MR683982}]
\label{thm:lace}
For any $n\ge 4$ or $n=2$, the permutahedron~$\Pi_n$ has a Hamilton path between any two permutations of opposite parity (i.e., it is Hamilton-laceable).
For $n=3$ it has a Hamilton path between any two permutations that differ by an adjacent transposition.
\end{theorem}

It is known~\cite{MR2185979} that $\Aut(\Pi_n)\cong \mathbb{Z}_2\ltimes S_n$ with $\mathbb{Z}_2\cong\{\ide,\rev\}$ where $\ide$ is the identity permutation and $\rev(x_1\cdots x_n):=x_n\cdots x_1$ is the \emph{reversal permutation}.
So any $f \in \Aut(\Pi_n)$ can be uniquely written as a pair $f=(\alpha,\pi)$ of $\alpha\in \{\ide,\rev\}$ and $\pi \in S_n$.
However, for our purposes we let $\alpha$ act on positions and $\pi$ on values, formally
\begin{equation}
\label{eq:f-alpha-pi}
(\alpha, \pi)(x_1\cdots x_n)=\pi(x_{\alpha(1)})\cdots \pi(x_{\alpha(n)}),
\end{equation}
that is, $(\alpha,\pi)(x)=\alpha x\pi$ under composition (applied in the order from left to right) where $\alpha$ is the identity or the reversal of positions and $\pi$ is a permutation of values.
The automorphism group $\Aut(\Pi_n)$  with this action is therefore a direct product of $\mathbb{Z}_2$ and $S_n$ since
$$(\alpha,\pi)(\beta,\rho)(x)=\beta(\alpha x\pi)\rho=(\alpha\beta)x(\pi\rho)=(\alpha\beta,\pi\rho)(x),$$
using the commutativity of $\mathbb{Z}_2$.

\subsection{Multiset permutations}
\label{sec:mperm}

A \emph{composition of an integer} $n\ge 1$ is a sequence $\ba=(a_1,\ldots,a_m)$ of positive integers with $\sum_{i=1}^m a_i=n$.
The \emph{partition of the set~$[n]$ associated} to~$\ba$ is $A_1\cup \cdots \cup A_m=[n]$ with $|A_i|=a_i$ and $x<y$ if $x\in A_i$ and $y\in A_j$ with~$i<j$.
For example, $A_1=\{1,2,3\}$ and $A_2=\{4,5\}$ for $\ba=(3,2)$.
A \emph{multiset permutation with frequencies~$\ba$}, or \emph{$\ba$-permutation} for short, is a sequence $u=u_1\cdots u_n$ of values from~$[m]$ with exactly $a_i$ occurrences of the value~$i$ for all $1\le i \le m$.
The set of all $\ba$-permutations is denoted by $\binom{[n]}{\ba}$, their number is the multinomial coefficient $\binom{n}{a_1,\ldots,a_m}$.
The lexicographically smallest $\ba$-permutation is called the \emph{identity $\ba$-permutation} and is denoted by $\ide(\ba)$.
For example, for $\ba=(3,2)$ we have $\ide(\ba)=11122$.
Let $u=u_1\cdots u_n$ be an $\ba$-permutation and for each $1\le i \le m$, let $x^i=x^i_1\cdots x^i_{a_i}$ be a permutation of~$A_i$, where the sets~$A_i$ form the partition of~$[n]$ associated to~$\ba$.
The \emph{mix} of~$u$ with $x^1,\ldots,x^m$ is the permutation of~$[n]$ denoted by $u \otimes (x^1, \ldots, x^m)$ obtained from~$u$ by replacing the $j$th occurrence of the value~$i$ with~$x^i_j$.
For example, $12211 \otimes (213,54)=25413$.

Let $G(\ba)$ denote the graph with vertex set~$\binom{[n]}{\ba}$ and edges between any two $\ba$-permutations that differ by an adjacent transposition of distinct values.
We will use the following two results on Hamilton paths and cycles in~$G(\ba)$.

\begin{theorem}[\cite{MR737262,MR821383,MR936104}]
\label{thm:comb-adj}
Let $\ba=(a_1,a_2)$ be a composition of~$n$ such that both~$a_1$ and~$a_2$ are odd.
Then the graph $G(\ba)$ has a Hamilton path between $1^{a_1}2^{a_2}$ and $2^{a_2}1^{a_1}$.
\end{theorem}

Note that the vertices~$1^{a_1}2^{a_2}$ and~$2^{a_2}1^{a_1}$ mentioned in Theorem~\ref{thm:comb-adj} have degree~1 in~$G(\ba)$, so there is no Hamilton cycle in this case.
However, for $\ba$-permutations on at least three distinct values with at least two odd multiplicities~$a_i$, Stachowiak~\cite{MR1157583} proved there is a Hamilton cycle in~$G(\ba)$, apart from one exception.

\begin{theorem}[\cite{MR1157583}]
\label{thm:multi-adj}
Let $\ba=(a_1,\ldots,a_m)$, $m\ge 3$, be a composition of~$n$ such that at least two of the~$a_i$ are odd.
Then $G(\ba)$ has a Hamilton cycle, unless $m=3$ and $n$ is even and $\{a_1,a_2,a_3\}=\{n-2,1,1\}$.
\end{theorem}

\subsection{Landau's function}
\label{sec:landau}

A \emph{partition of an integer} $n\ge 1$ is a sequence $\ba=(a_1,\ldots,a_m)$ of positive integers with $\sum_{i=1}^m a_i=n$ and $a_1\ge a_2\ge\cdots\ge a_m$.
The maximal order of an element of~$S_n$ is called \emph{Landau's function}~\cite{landau_1903}, and we denote it by~$\lambda(n)$.
It is determined by
\begin{equation}
\label{eq:landau-lcm}
\lambda(n)=\max_{a_1+\cdots+a_m=n} \lcm(a_1,\ldots,a_m),
\end{equation}
where the maximum ranges over all partitions of~$n$, and $\lcm(a_1,\ldots,a_m)$ denotes the least common multiple of~$a_1,\ldots,a_m$.
The first values of $\lambda(n)$ are shown in Table~\ref{tab:landau} (see also the appendix); this is OEIS sequence~A000793~\cite{oeis}.

Concerning the asymptotic growth of $\lambda(n)$, Landau showed that
\begin{equation}
\label{eq:landau-asymp}
\lambda(n)=e^{(1+o(1))\sqrt{n\ln n}}.
\end{equation}

For our arguments we will need two variants of Landau's function that we define in the following.
For any partition~$(a_1,\ldots,a_m)$ write $e(a_1,\ldots,a_m)$ for the number of even entries of the partition.
We define
\begin{subequations}
\label{eq:landau-even}
\begin{align}
\lambda_0(n)&:=\max_{\substack{a_1+\cdots+a_m=n \\ e(a_1,\ldots,a_m)=0}} \lcm(a_1,\ldots,a_m), \label{eq:landau0} \\
\lambda_2(n)&:=\max_{\substack{a_1+\cdots+a_m=n \\ e(a_1,\ldots,a_m)\in\{2,4,6,\ldots\}}} \lcm(a_1,\ldots,a_m). \label{eq:landau2}
\end{align}
\end{subequations}
The maximizations in~\eqref{eq:landau-even} are over all integer partitions of~$n$ that have 0 even parts (i.e., only odd parts), or exactly~$2,4,6,\ldots$ even parts, respectively.
The only difference of these definitions to~\eqref{eq:landau-lcm} are the additional requirements about the parity of the~$a_i$.
The sequences $(\lambda_0(n))_{n\ge 1}$ and $(\lambda_2(n))_{n\ge 1}$ appear not to have been studied before.

The first few values of~$\lambda_0(n)$ and~$\lambda_2(n)$ are shown in Table~\ref{tab:landau}, comparing them with the corresponding values for~$\lambda(n)$.
We clearly have $\lambda_0(n)\le \lambda(n)$ for all $n\ge 1$, with equality e.g.\ for~$n=1,3,8,15$.
Similarly, we have $\lambda_2(n)\le \lambda(n)$ for all $n\ge 1$, with equality e.g.\ for~$n=21,22,45,46,51,52,55,56,74,75,81,82,87,88,91,92,99,100$.
On the other hand, $\lambda_0(n)$ can be much smaller than $\lambda(n)$.
For example, for $n=19,20$ we have $\lambda(n)/\lambda_0(n)=28/11=2.54\ldots$, for $n=30,31$ we have $\lambda(n)/\lambda_0(n)=44/13=3.38\ldots$ and for $n=53,54$ we have $\lambda(n)/\lambda_0(n)=72/17=4.23\ldots$.
One can also see that $\lambda_0(n)\le \alpha(n)\le \lambda(n)$, where $\alpha(n)$ is the maximal order of an element of the alternating group~$A_n$ (OEIS sequence~A051593).
More numerical experiments about the Landau function and its variants are reported in the appendix.

\begin{table}[h!]
\makebox[0cm]{ 
\setlength{\tabcolsep}{2pt}
\tiny
\begin{tabular}{c|cccccccccccccccccccc}
  $n$ & 1 & 2 & 3 & 4 & 5 & 6 & 7 & 8 & 9 & 10 & 11 & 12 & 13 & 14 & 15 & 16 & 17 & 18 & 19 & 20 \\ \hline
  $\lambda(n)$ & 1 & 2 & 3 & 4 & 6 & 6 & 12 & 15 & 20 & 30 & 30 & 60 & 60 & 84 & 105 & 140 & 210 & 210 & 420 & 420 \\
  $\lcm(\cdot)$ & 1 & 2 & 3 & 4 & 3,2 & 3,2,1 & 4,3 & 5,3 & 5,4 & 5,3,2 & 5,3,2,1 & 5,4,3 & 5,4,3,1 & 7,4,3 & 7,5,3 & 7,5,4 & 7,5,3,2 & 7,5,3,2,1 & 7,5,4,3 & 7,5,4,3,1 \\ \hline
  $\lambda_0(n)$ & 1 & 1 & 3 & 3 & 5 & 5 & 7 & 15 & 15 & 21 & 21 & 35 & 35 & 45 & 105 & 105 & 105 & 105 & 165 & 165 \\
  $\lcm(\cdot)$ & 1 & 1,1 & 3 & 3,1 & 5 & 5,1 & 7 & 5,3 & 5,3,1 & 7,3 & 7,3,1 & 7,5 & 7,5,1 & 9,5 & 7,5,3 & 7,5,3,1 & 7,5,3,1,1 & 7,5,3,1,1,1 & 11,5,3 & 11,5,3,1 \\
  $\frac{\lambda(n)}{\lambda_0(n)}$ & 1 & 2 & 1 & 1.33.. & 1.2 & 1.2 & 1.71.. & 1 & 1.33.. & 1.42.. & 1.42.. & 1.71.. & 1.71.. & 1.86.. & 1 & 1.33.. & 2 & 2 & 2.54.. & 2.54.. \\ \hline
  $\lambda_2(n)$ & $-\infty$ & $-\infty$ & $-\infty$ & 2 & 2 & 4 & 6 & 6 & 12 & 12 & 20 & 30 & 30 & 60 & 60 & 84 & 84 & 140 & 210 & 210 \\
  $\lcm(\cdot)$ & -- & -- & -- & 2,2 & 2,2,1 & 4,2 & 3,2,2 & 3,2,2,1 & 4,3,2 & 4,3,2,1 & 5,4,2 & 5,3,2,2 & 5,3,2,2,1 & 5,4,3,2 & 5,4,3,2,1 & 7,4,3,2 & 7,4,3,2,1 & 7,5,4,2 & 7,5,3,2,2 & 7,5,3,2,2,1 \\
  $\frac{\lambda(n)}{\lambda_2(n)}$ & -- & -- & -- & 2 & 3 & 1.5 & 2 & 2.5 & 1.66.. & 2.5 & 1.5 & 2 & 2 & 1.4 & 1.75 & 1.66.. & 2.5 & 1.5 & 2 & 2 \\
\end{tabular}
}
\caption{First 20 values of $\lambda(n)$, $\lambda_0(n)$ and $\lambda_2(n)$.}
\label{tab:landau}
\end{table}

\begin{lemma}
\label{lem:landau-prop}
The functions~$\lambda(n)$, $\lambda_0(n)$ and $\lambda_2(n)$ have the following properties:
\begin{enumerate}[label=(\roman*),leftmargin=8mm, topsep=0mm, noitemsep]
\item The maximum in~\eqref{eq:landau-lcm} is attained for a partition of~$n$ into powers of distinct primes and~1s.
\item The maximum in~\eqref{eq:landau0} is attained for a partition of~$n$ into powers of distinct odd primes and~1s.
\item The maximum in~\eqref{eq:landau2} is attained for a partition of~$n$ into powers of distinct odd primes, a positive power of~2, a~2 and~1s.
For $n\ge 12$ this partition has at least~$m\ge 4$ parts.
\item We have $\max\{\lambda_0(n),\lambda_2(n)\}\ge \lambda(n)/2$.
\item We have $\lambda(n)/\lambda_0(n)\geq 4$ for $n\geq 739$ and $\lim_{n\rightarrow \infty}\lambda(n)/\lambda_0(n)=+\infty$, and we have $\lambda(n)/\lambda_2(n)\leq 2$ for $n\geq 18$.
Consequently, we have $2\lambda_0(n)\leq \lambda_2(n)$ for $n\geq 739$ and $\lim_{n\rightarrow \infty}2\lambda_0(n)/\lambda_2(n)=0$.
\item There are arbitrarily long intervals with $\lambda(n)=\lambda_2(n)$.
\end{enumerate}
\end{lemma}

For example, $\lambda(11)=30$ is attained by the partition~$(6,5)$ but also by the partition~$(5,3,2,1)$.
Similarly, $\lambda_2(10)=12$ is attained by the partition~$(6,4)$ but also by the partition~$(4,3,2,1)$.
The proof of Lemma~\ref{lem:landau-prop} is deferred to the appendix.

\subsection{Cayley graphs}
\label{sec:prelim-cayley}

For a group~$G$ and a generating set~$S\seq G$, we define the \emph{Cayley graph} $\Gamma(G,S)$ as the graph with vertex set~$G$ and undirected edges $\{x,y\}$ for all $x,y\in G$ and $s\in S$ with $y=xs$.
As the edges are undirected, the graphs $\Gamma(G,S)$ and $\Gamma(G,S\cup S^{-1})$ are the same, where $S^{-1}:=\{s^{-1}\mid s\in S\}$, so we can assume w.l.o.g.\ that if $s\in S$, then also $s^{-1}\in S$.
If $P=(x_1,\ldots,x_n)$ is a path in the Cayley graph~$\Gamma(G,S)$, then for any $g\in G$ the sequence $g P:=(g x_1,\ldots,g x_n)$ is also a path.

It has long been conjectured that all Cayley graphs admit a Hamilton cycle, but despite considerable effort and many partial results (see the surveys~\cite{MR762322,MR1405010,MR2548568,MR2548567}), this problem is still very much open in general.
On the other hand, for Cayley graphs of abelian groups several results are known.
A classical folklore result (see e.g.~\cite[Sec.~3]{MR762322}) states that every Cayley graph of an abelian group has a Hamilton cycle.

\begin{theorem}
\label{thm:abelian}
Let $G$ be an abelian group with $|G|\geq 3$ and $S\seq G$ a generating set.
Then the Cayley graph $\Gamma(G,S)$ has a Hamilton cycle.
\end{theorem}

In fact, this result can be generalized considerably using the following stronger notions of Hamiltonicity.
A graph is called \emph{Hamilton-connected} if for any two vertices there is a Hamilton path starting and ending at these two vertices.
Similarly, a bipartite graph is called \emph{Hamilton-laceable} if for any two vertices in different partition classes there is a Hamilton path joining these two vertices.
The following theorem of Chen and Quimpo asserts that Cayley graphs of abelian groups possess the strongest possible of these two Hamiltonicity notions.

\begin{theorem}[\cite{MR641233}]
\label{thm:chen-quimpo}
Let $G$ be an abelian group and $S\seq G$ a generating set.
If the Cayley graph $\Gamma=\Gamma(G,S)$ has minimum degree at least~3, then we have the following:
\begin{enumerate}[label=(\roman*),leftmargin=8mm, topsep=0mm, noitemsep]
\item If $\Gamma$ is bipartite, then $\Gamma$ is Hamilton-laceable.
\item If $\Gamma$ is not bipartite, then $\Gamma$ is Hamilton-connected.
\end{enumerate}
\end{theorem}

We will also need the fact that Cayley graphs are highly connected.

\begin{lemma}[{\cite[Theorem 3]{MR266804}}]
\label{lem:connectivity}
If\/ $\Gamma$ is a connected $d$-regular Cayley graph, then the vertex-connectivity of\/~$\Gamma$ is at least~$2d/3$.
\end{lemma}

\section{Hypercubes}
\label{sec:cubes}

In this section we consider the family of hypercubes~$Q_n$ introduced in Section~\ref{sec:results} (recall also Section~\ref{sec:prelim-cubes}).
We first show that the binary reflected Gray code has constant compression~4.
We then establish a general linear (in $n$) upper bound for $\kappa(Q_n)$, and a matching lower bound construction, i.e., an automorphism and a Hamilton cycle in~$Q_n$ whose compression equals this upper bound.
Lastly, we apply our constructions to derive $t$-track Hamilton cycles in~$Q_n$.

\subsection{The binary reflected Gray code (BRGC)}

Recall the definition of the BRGC~$\Gamma_n$ given in Section~\ref{sec:results-cubes}.

\begin{proposition}
\label{prop:brgc}
The BRGC $\Gamma_n$ has compression $\kappa(Q_n,\Gamma_n)=4$ for $n\ge 2$.
\end{proposition}

The BRGC~$\Gamma_n$ is illustrated in Figure~\ref{fig:comp4}~(a) and Figure~\ref{fig:q8}~(a) for $n=4$ and $n=8$, respectively, and those two pictures indeed have 4-fold rotational symmetry.

\begin{proof}
For $n=2$ the graph~$Q_2$ is a 4-cycle, so the claim is trivially true.
For the rest of the proof we assume that $n\ge 3$.
Unrolling the inductive definition~$\Gamma_n=0\Gamma_{n-1},1\lvec{\Gamma_{n-1}}$ two more times gives
\begin{equation}
\label{eq:Gamma2}
\Gamma_n=00\Gamma_{n-2},01\lvec{\Gamma_{n-2}},11\Gamma_{n-2},10\lvec{\Gamma_{n-2}}
\end{equation}
and
\begin{equation}
\label{eq:Gamma3}
\begin{aligned}
\Gamma_n=&000\Gamma_{n-3},001\lvec\Gamma_{n-3}, \\
         &011\Gamma_{n-3},010\lvec\Gamma_{n-3}, \\
         &110\Gamma_{n-3},111\lvec\Gamma_{n-3}, \\
         &101\Gamma_{n-3},100\lvec\Gamma_{n-3}.
\end{aligned}
\end{equation}
From~\eqref{eq:Gamma3} we see that~$\Gamma_n$ has compression at least~4 under the automorphism $x_1 x_2 x_3 \mapsto x_2 \overline{x_1} \overline{x_3}$ (the bits $x_4,\ldots,x_n$ are not modified), which maps each line in~\eqref{eq:Gamma3} to the next line.

To show that the compression is at most~4, let $x_1,\ldots,x_N$, $N:=2^n$, be the sequence of bitstrings of~$\Gamma_n$.
From~\eqref{eq:Gamma2} we see that $x_i$ differs from $x_{i+N/2}$ by complementing the first two bits, for all $i=1,\ldots,N$, where indices are considered modulo~$N$.
It follows that if $x_i$ and~$x_{i+1}$ have the same first two bits, then $S_i:=(x_i,x_{i+1},x_{i+N/2},x_{i+1+N/2})$ is not a 4-cycle.
On the other hand, if $x_i$ and~$x_{i+1}$ differ in one of the first two bits, then $S_i$ is a 4-cycle.
From~\eqref{eq:Gamma2} we see that this happens precisely for $i=N/4,2N/4,3N/4,N$, proving that the compression is at most~4.
\end{proof}

\subsection{An upper bound}

Recall from Section~\ref{sec:prelim-cubes} that $\Aut(Q_n)\cong S_n \ltimes \mathbb{Z}_2^n$, so from~\eqref{eq:kappa-UB} we obtain $\kappa(Q_n)\le 2\lambda(n)$, where $\lambda(n)$ is Landau's function.
We now improve this upper bound drastically to a function that is linear in~$n$ (cf.~\eqref{eq:landau-asymp}).

\begin{lemma}
\label{lem:cube-ub}
Let $n\ge 3$.
If $Q_n$ has a $k$-symmetric Hamilton cycle, then $k=2^i<2n$ for some $i$.
\end{lemma}

\begin{proof}
Consider an automorphism~$f$ of~$Q_n$ and a path~$P$ such that
\begin{equation}
\label{eq:CP2}
C=P,f(P),f^2(P),\ldots,f^{k-1}(P)
\end{equation}
is a $k$-symmetric Hamilton cycle in~$Q_n$ for some~$k$ (recall~\eqref{eq:CP}).
It follows that $k$ must divide $2^n$, i.e., we have $k=2^i$ for some $i\le n$.
Let $\pi\in S_n$ and $z\in \mathbb{Z}_2^n$ be such that~\eqref{eq:Qaut} holds.
Suppose that $\pi$ is a product of cycles of lengths $a_1\ge \cdots \ge a_m$, and note that $\ord(\pi)=\lcm(a_1,\ldots,a_m)$.
As $\pi^k=\ide$, each $a_j$ divides $k=2^i$, implying that $\ord(\pi)=\lcm(a_1,\ldots,a_m)=a_1\le n$.
We conclude that $\ord(f)=k=2^i\le 2\ord(\pi)\le 2n$, in particular $i<n$ for $n\ge 3$.

To complete the proof, it remains to rule out the possibility~$k=2n$.
In this case, $\pi$ has just a single cycle of length $a_1=n=2^{i-1}>2$.
As $|P|=2^{n-i}$ is even (because of $i<n$), the first and last vertex of~$P$ have opposite parity.
Observe that $z$ must have odd parity, otherwise we would have $f^n=\ide$ (contradicting $\ord(f)=k=2n$), implying that the first vertices of~$P$ and~$f(P)$ have opposite parity.
Combining these observations shows that the last vertex of~$P$ and the first vertex of~$f(P)$ have the same parity, so they cannot be adjacent, contradicting~\eqref{eq:CP2}.
We conclude that $k=2^i<2n$, which completes the proof.
\end{proof}

\subsection{An optimal construction}

We consider the automorphism
\begin{equation}\label{eq:gdef}
g(x_1 x_2 \cdots x_n)=x_2 \cdots x_n\overline{x_1},
\end{equation}
of~$Q_n$, i.e., $g$ cyclically shifts all bits to the left by one position and then complements the last bit.
The mapping~$g$ is an auxiliary automorphism and the automorphism~$f$ that determines~$\kappa(Q_n)$ will be defined below in Lemma~\ref{lem:product}.
Clearly $g^{2n}=\ide$, so all orbits have size at most~$2n$.
We will see that if $n$ is a power of~2, i.e., $n=2^r$ for some $r\ge 1$, then all orbits have the same size~$2n$.

For every $n=2^r$, $r\ge 1$, we inductively define a set~$R_n$ of vertices of~$Q_n$ that are representatives of orbits of~$g$, i.e., from every orbit precisely one vertex is in~$R_n$.
First, we define a function that interleaves the bits of two bitstrings~$u=u_1\cdots u_n$ and~$v=v_1\cdots v_n$ of equal length by
\begin{equation}
\label{eq:intdef}
u \circ v:=u_1 v_1 u_2 v_2 \cdots u_n v_n,
\end{equation}
i.e., the result is a bitstring of length~$2n$ that alternately contains the bits of~$u$ and~$v$, starting by the first bit of~$u$.
Observe that
\begin{subequations}\label{eq:gint}
\begin{align}
g^{2i}(u \circ v) &= g^i(u) \circ g^i(v), \label{eq:ginteven}\\
g^{2i+1}(u \circ v) &= g^i(v) \circ g^{i+1}(u) \label{eq:gintodd}
\end{align}
\end{subequations}
for every $0\le i <2n$ and $u,v\in \{0,1\}^n$.
For $n=2^r$, $r\ge 1$, we define the set~$R_n$ inductively by
\begin{subequations}\label{eq:R}
\begin{align}
R_2 &:= \{00\}, \label{eq:R2} \\
R_{2n} &:= \{g^k(u) \circ v \mid u,v\in R_n \text{ and } 0\le k <n\}. \label{eq:R2n}
\end{align}
\end{subequations}
For example, we have $R_4=\{0000,0010\}$ and
\begin{table}[h!]
\centering
\begin{tabular}{|rccccl|r|}
\hline
  & \multicolumn{2}{c|}{$u=0000$} & \multicolumn{2}{c}{$u=0010$} & & \\ \hline
  & \multicolumn{1}{c|}{$v=0000$} & \multicolumn{1}{c|}{$v=0010$} & \multicolumn{1}{c|}{$v=0000$} & $v=0010$ & & \\ \hline
   $R_8=\{$ & $00000000$, & $00000100$, & $00001000$, & $00001100$,   &  & $k=0$ \\
            & $00000010$, & $00000110$, & $00100010$, & $00100110$,   &  & $k=1$ \\
            & $00001010$, & $00001110$, & $10001010$, & $10001110$,   &  & $k=2$ \\
            & $00101010$, & $00101110$, & $00101000$, & $00101100\phantom{,}$ & $\}.$ & $k=3$ \\ \hline
\end{tabular}
\label{table:R8}
\end{table}
\torsten{Make sure that this table appears in the correct place}

\begin{lemma}
\label{lem:gRn}
For every $n=2^r$, $r\ge 1$, $0\le j <2n$, and $x,y\in R_n$, we have $g^{j}(x) \ne y$ unless~$x=y$ and~$j=0$.
\end{lemma}

\begin{proof}
We proceed by induction on~$r$.
The statement holds trivially for~$r=1$.
For the induction step from~$r$ to~$r+1$ let $0\le j <4n$ and $x,y \in R_{2n}$.
By~\eqref{eq:R2n}, $x=g^k(u) \circ v$ and $y=g^{k'}(u')\circ v'$ for some $u,v,u',v' \in R_{n}$ and $0\le k,k' < n$.
We distinguish two cases based on the parity of~$j$.

If $j=2i$ where $0\le i <2n$ then from~\eqref{eq:ginteven} we have
\begin{equation*}
g^j(x)=g^{2i}(g^k(u) \circ v)=g^{k+i}(u)\circ g^i(v),
\end{equation*}
which equals $y=g^{k'}(u')\circ v'$ only if
\begin{equation*}
g^{k+i}(u)=g^{k'}(u')\text{ and }g^i(v)=v'.
\end{equation*}
By the induction hypothesis, this holds only if
\begin{equation*}
u=u',\ k+i=k' \bmod 2n,\text{ and } v=v',\ i=0,
\end{equation*}
equivalently, $x=y$ and $j=0$.

Similarly, if $j=2i+1$ where $0\le i <2n$ then from~\eqref{eq:gintodd} we have
\begin{equation*}
g^j(x)=g^{2i+1}(g^k(u) \circ v)=g^{i}(v)\circ g^{k+i+1}(u),
\end{equation*}
which equals $y=g^{k'}(u')\circ v'$ only if
\begin{equation*}
g^{i}(v)=g^{k'}(u')\text{ and }g^{k+i+1}(u)=v'.
\end{equation*}
By the induction hypothesis, this holds only if
\begin{equation*}
v=u',\ i=k' \bmod 2n,\text{ and } u=v',\ k+i+1=0 \bmod 2n.
\end{equation*}
However, it cannot be that $k+k'+1=0 \bmod 2n$ since $0\le k,k'<n$.
So in this case, $g^j(x)\ne y$.
\end{proof}

\begin{lemma}
\label{lem:Rn-repr}
For every $n=2^r$, $r\ge 1$, all orbits of~$g$ have the same size~$2n$ and $R_n$ is a set of representatives for all ${2^n}/(2n)=2^{n-r-1}$ orbits.
\end{lemma}

\begin{proof}
Lemma~\ref{lem:gRn} shows that no two elements of~$R_n$ are in the same orbit, and that every such orbit has size~$2n$.
It remains to argue that there is no other orbit, simply because $|R_n|\cdot 2n=2^n$, which is equivalent to $|R_n|=2^{n-r-1}$.
Indeed from~\eqref{eq:R} we have $|R_2|=1=2^{n-r-1}$ for $r=1$ and
\begin{align*}
|R_{2n}| &= |R_n|^2\cdot n=(2^{n-r-1})^2\cdot 2^r=2^{2n-(r+1)-1}
\end{align*}
by induction for all $r\ge 1$.
\end{proof}

\begin{lemma}
\label{lem:PRn}
For every $n=2^r$, $r\ge 2$, there is a path~$P_n$ in~$Q_n$ that visits all vertices~$R_n$ and that starts in~$0^n$ and ends in~$0^{n-2}10$.
\end{lemma}

\begin{proof}
For $r=2$ we set $P_4:=(0000,0010)$.
It is straightforward to verify from~\eqref{eq:R} that $0^n,0^{n-2}10 \in R_n$.
Note that $g^2$ is an automorphism of~$Q_n$ with orbits of size~$n$, and $0^{n-2}10$ is adjacent in~$Q_n$ to $g^2(0^n)=0^{n-2}11$.
For the induction step $r\rightarrow r+1$ we construct the path~$P_{2n}$ recursively from~$P_n$.
Specifically, by the observations from before
\begin{subequations}
\label{eq:P'P''}
\begin{align}
P'&:=P_n,g^2(P_n),\ldots,g^{n-2}(P_n), \\
P''&:=g(P_n),g^3(P_n),\ldots,g^{n-1}(P_n)
\end{align}
\end{subequations}
are vertex-disjoint paths in~$Q_n$ that start in~$0^{n}$ and~$0^{n-1}1$, respectively.
Consider the path
\begin{equation}
\label{eq:P2n}
P_{2n}:=h(P'\sqcup P_n), h(P'' \sqcup \lvec P_n)
\end{equation}
in $Q_{2n}=Q_n\boxprod Q_n$, where $\sqcup$ is defined in~\eqref{eq:zigzag} and $h(u,v):=u\circ v$, for $u,v\in\{0,1\}^n$, is the interleaving function defined in~\eqref{eq:intdef}, applied to all vertices along the paths.

As $|P_n|=|R_n|=2^{n-r-1}$ is even and the first vertices of~$P'$ and~$P''$ differ in a single bit, the transition between the two halves of~\eqref{eq:P2n} flips a single bit, as desired.
Specifically, $P'\sqcup P_n$ ends in $0^n \circ 0^{n-2}10$ and $P'' \sqcup \lvec P_n$ starts in $0^{n-1}1 \circ 0^{n-2}10$.
Observe furthermore that $P_{2n}$ starts in $0^n \circ 0^n=0^{2n}$ and ends in $0^{n-1}1\circ 0^n=0^{2n-2}10$.
By induction, we know that~$P_n$ visits every vertex of~$R_n$ exactly one.
Using this with~\eqref{eq:R2n} and~\eqref{eq:P'P''} shows that~$P_{2n}$ visits every vertex of~$R_{2n}$ exactly once.

To illustrate the construction, for $n=4$ we have $P_4=(0000,0010)$, so
\begin{align*}
P'&=\phantom{g(}P_4\phantom{)},g^2(P_4)= (0000,0010,0011,1011), \\
P''&=g(P_4),g^3(P_4)=(0001,0101,0111,0110)
\end{align*}
and
\begin{table}[h!]
\centering
\begin{tabular}{|rccccl|l|l|}
\hline
   $P_8=($ & $00000000$, & $00001000$, & $00001010$, & $10001010$,   &  & $h(P',0000)$ & $h(P'\sqcup P_4)$\\
           & $10001110$, & $00001110$, & $00001100$, & $00000100$,   &  & $h(P',0010)$ & \\
           & $00000110$, & $00100110$, & $00101110$, & $00101100$,   &  & $h(P'',0010)$ & $h(P'' \sqcup \lvec P_4)$ \\
           & $00101000$, & $00101010$, & $00100010$, & $00000010\phantom{,}$ & $).$ & $h(P'',0000)$  & \\ \hline
\end{tabular}
\end{table}

\end{proof}

Note that the end vertex~$0^{n-2}10$ of~$P_n$ is not adjacent to $g(0^n)=0^{n-1}1$ in~$Q_n$, so $P_n$ with~$g$ does not directly produce a $2n$-symmetric Hamilton cycle in~$Q_n$ (recall~\eqref{eq:CP}).
However, in the following we show that $P_n$ can produce a $2n$-symmetric Hamilton cycle in~$Q_{n+m}$ for any~$m\ge 1$.

\begin{figure}
\makebox[0cm]{ 
\begin{tabular}{ccc}
\raisebox{37mm}{(a)} &
\includegraphics[scale=0.85]{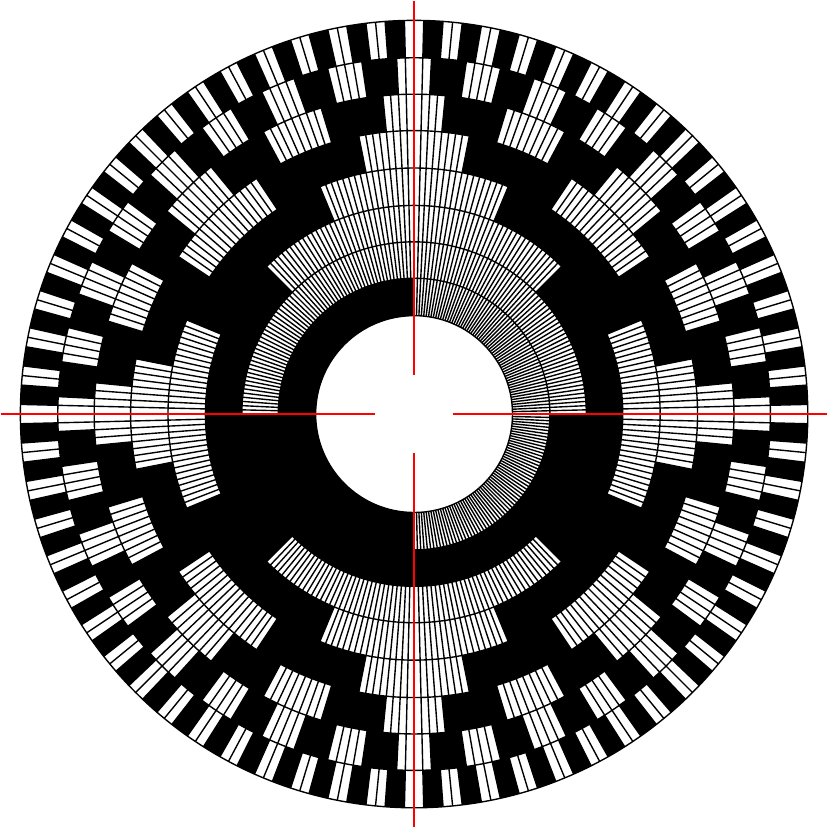} &
\includegraphics[scale=0.85]{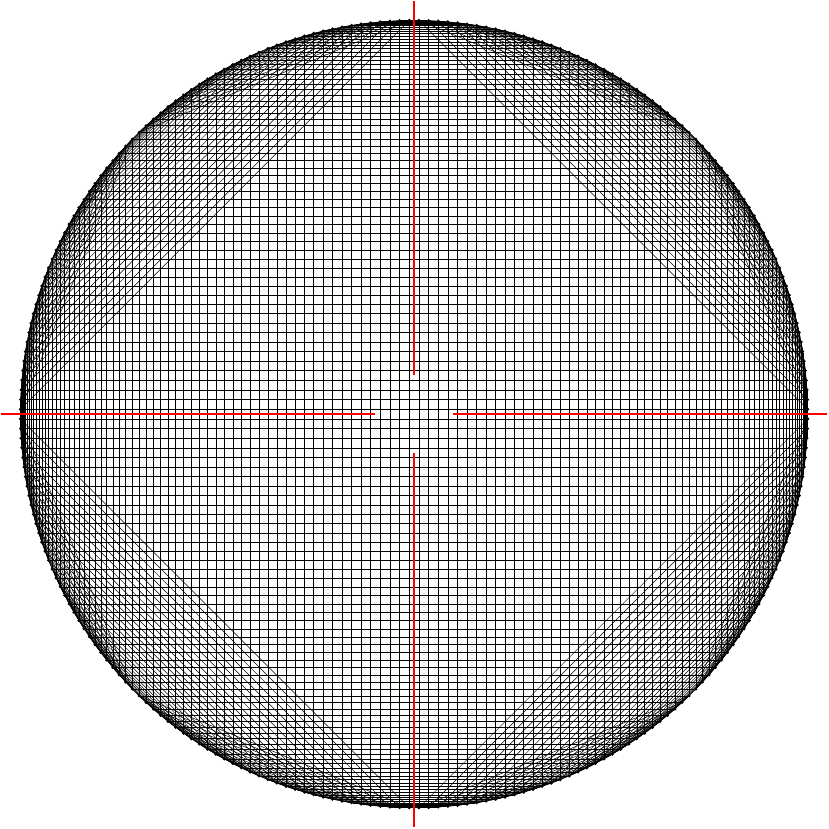} \\
\raisebox{37mm}{(b)} &
\includegraphics[scale=0.85]{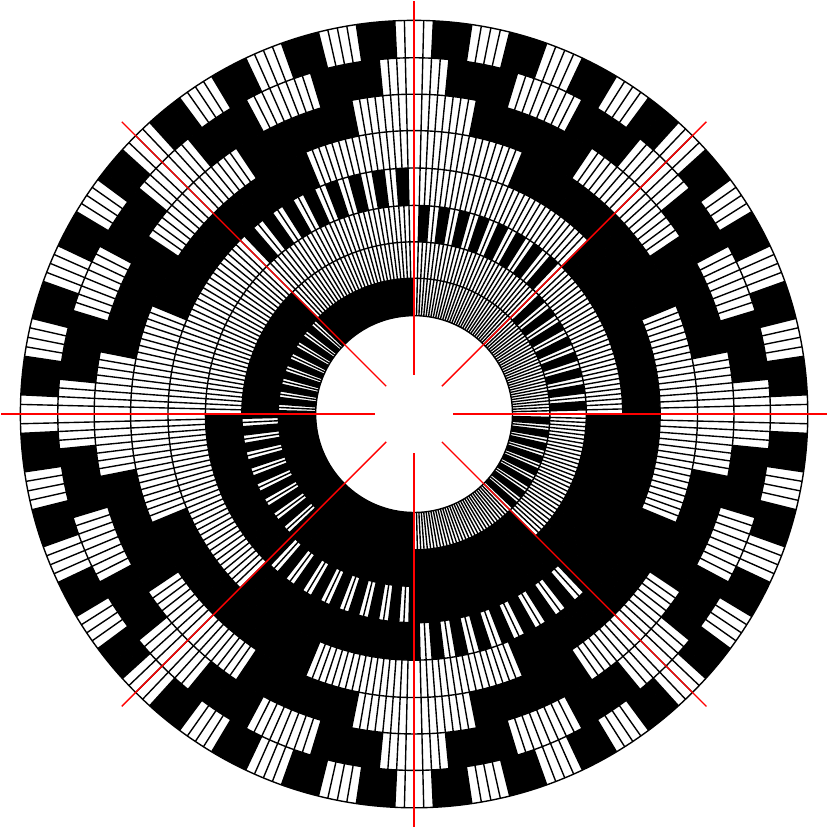} &
\includegraphics[scale=0.85]{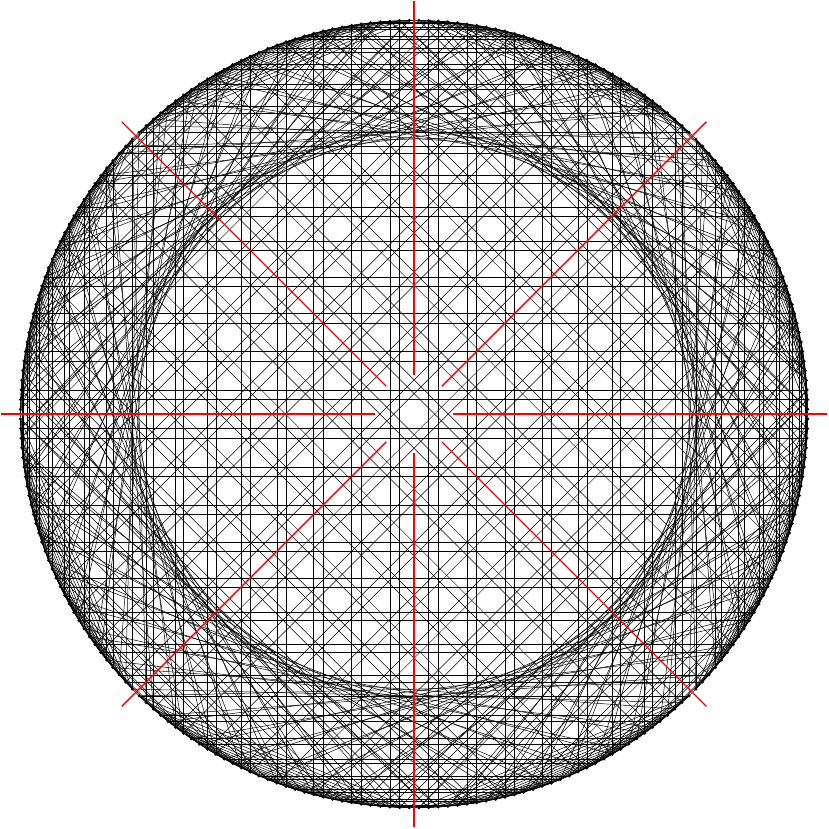} \\
\raisebox{37mm}{(c)} &
\includegraphics[scale=0.85]{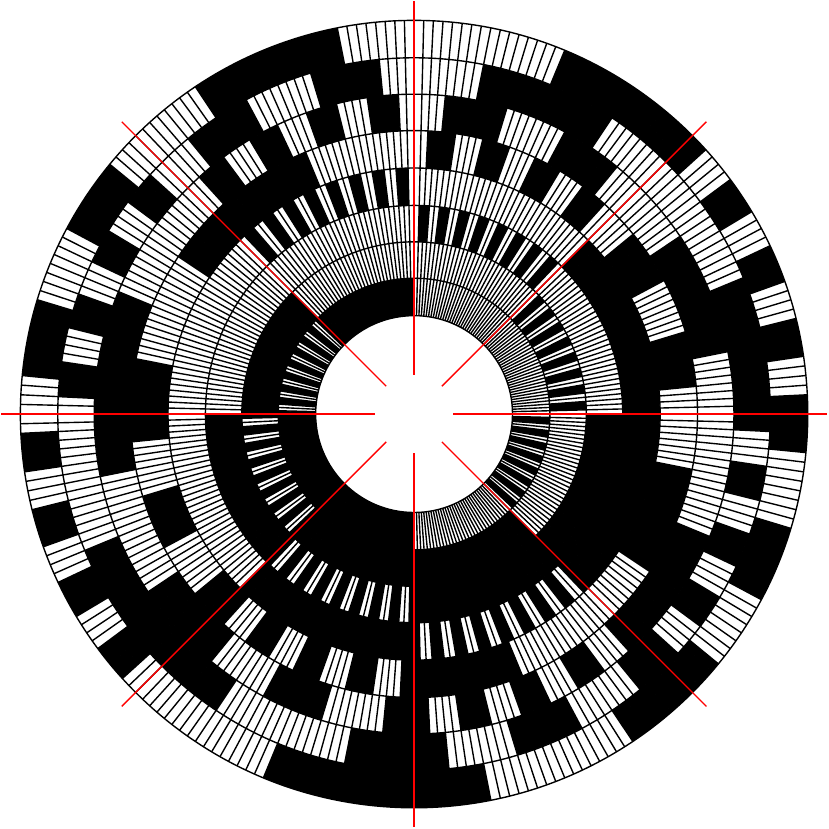} &
\includegraphics[scale=0.85]{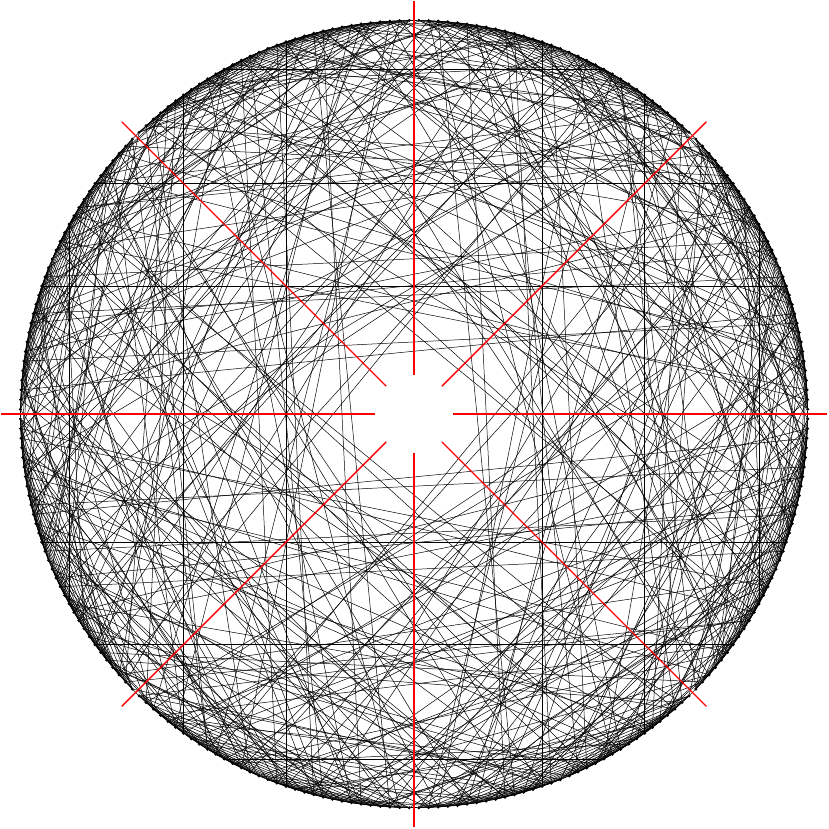} \\
\end{tabular}
}
\caption{Symmetric Hamilton cycles in~$Q_8$.
Cycles are on the left (0=white, 1=black), with the first and last bit on the inner and outer track, respectively.
The full graph~$Q_8$ is on the right, with vertices arranged in cycle order and edges drawn as straight lines.
(a)~Binary reflected Gray code~$\Gamma_8$ with compression~4;
(b)~Hamilton cycle with compression~8 from Theorem~\ref{thm:Qnm} ($h$ is complementation of the first bit and $Q=\Gamma_4$);
(c)~2-track Hamilton cycle with compression~8 from Theorem~\ref{thm:2track} ($h=g$ and $Q=\Gamma_4'$, i.e., $\Gamma_4$ with every bitstring reversed).}
\label{fig:q8}
\end{figure}

\begin{lemma}
\label{lem:product}
Let $g$ be an automorphism of a graph~$G$ with orbits of the same size $k\ge 2$ and let $P$ be a path in~$G$ on orbit representatives~$R$ starting at some vertex~$u$ that is adjacent to~$g(u)$ in~$G$.
Let $h$ be an automorphism of a graph~$H$ on an even number of vertices such that $\ord(h)$ divides~$k$ and let $Q$ be a Hamilton path of~$H$ between~$v$ and~$h(v)$ for some vertex~$v$.
Then the Cartesian product $G\boxprod H$ has a $k$-symmetric Hamilton cycle
\begin{equation}
\label{eq:CPQ}
C:=P \sqcup Q,f(P \sqcup Q),f^2(P \sqcup Q),\ldots,f^{k-1}(P \sqcup Q)
\end{equation}
where $f:=(g,h)$ is the (product) automorphism of~$G\boxprod H$.
\end{lemma}

\begin{proof}
Since $f^k=\ide$, all orbits of~$f$ have size at most~$k$.
Furthermore, for any $(u,v)$ and $(u',v')$ from the set
\begin{equation*}
A:=R \times V(H) = \{(u,v) \mid u\in R, v\in V(H)\}
\end{equation*}
and any $0\le i,i' <k$ we have that $f^i(u,v)=f^{i'}(u',v')$ only if
\begin{equation*}
g^i(u)=g^{i'}(u') \text{ and } h^i(v)=h^{i'}(v'),
\end{equation*}
which holds only if $u=u'$, $i=i'$, and $v=v'$ since $u,u'$ are orbit representatives for~$g$.
Thus no two elements of~$A$ are from the same orbit of~$f$ and since
\begin{equation*}
|A|=|R|\cdot |V(H)|=|V(G)|/k\cdot |V(H)|=|V(G\boxprod H)|/k,
\end{equation*}
the set~$A$ contains representatives of all orbits of~$f$ and they all have the same size~$k$.

As $Q$ has even length, the path $P\sqcup Q$ starts in~$(u,v)$, ends in~$(u,h(v))$, and it contains exactly the vertices of~$A$.
Furthermore, $(u,h(v))$ is adjacent to $f(u,v)=(g(u),h(v))$, so $C$ defined in~\eqref{eq:CPQ} is a $k$-symmetric Hamilton cycle in~$G\boxprod H$.
\end{proof}

\begin{theorem}
\label{thm:Qnm}
For every $n=2^r$, $r\ge 2$, and $m\ge 1$, the hypercube $Q_{n+m}$ has a $2n$-symmetric Hamilton cycle.
\end{theorem}

The 8-symmetric Hamilton cycle in~$Q_8$ obtained from Theorem~\ref{thm:Qnm} for $n=m=4$ is illustrated in Figure~\ref{fig:q8}~(b).

\begin{proof}
Let $g$ be the automorphism of~$Q_n$ given by~\eqref{eq:gdef}.
All orbits of~$g$ have the same size $k:=2n$ and by Lemmas~\ref{lem:Rn-repr} and~\ref{lem:PRn}, there is a path~$P_n$ on orbit representatives~$R_n$ that starts in~$u:=0^n$, which is adjacent to $g(u)=0^{n-1}1$.
Let $h$ be the automorphism of~$Q_m$ that flips the first bit.
Clearly, $\ord(h)=2$ divides $k$, and there is a Hamilton path in~$Q_m$ between $v:=0^m$ and $h(v)=10^{m-1}$; take for example the BRGC~$\Gamma_m$ defined in Section~\ref{sec:results-cubes}.
Applying Lemma~\ref{lem:product}, we obtain that the graph $Q_{n+m}=Q_n\boxprod Q_m$ has the $2n$-symmetric Hamilton cycle defined in~\eqref{eq:CPQ}.
\end{proof}

Combining the previous results, we obtain the following closed formula for the Hamilton compression of~$Q_n$.

\begin{theorem}
\label{thm:kappa-Qn}
We have $\kappa(Q_2)=4$ and $\kappa(Q_n)=2^{\lceil \log_2 n \rceil}$ for all $n\ge 3$.
\end{theorem}

Note that $n\le \kappa(Q_n)<2n$ for $n\ge 2$, in particular $\kappa(Q_n)=\Theta(n)$.

\begin{proof}
$Q_2$ is a 4-cycle, which has optimal compression $\kappa(Q_2)=4$.
Note that $2^{\lceil \log_2 n \rceil}$ is the largest power of~2 that is less than~$2n$.
By Lemma~\ref{lem:cube-ub}, this is a valid upper bound for~$\kappa(Q_n)$ for all $n\ge 3$.
For $n=3$ and $n=4$, this upper bound is~4, and it is attained by the BRGC~$\Gamma_n$, which has compression $\kappa(Q_n,\Gamma_n)=4$ by Proposition~\ref{prop:brgc}.
For any $n\ge 5$ we define $r:=\lceil \log_2 n\rceil-1$, $n':=2^r$, and $m:=n-2^r$, and Theorem~\ref{thm:Qnm} yields a $2n'$-symmetric Hamilton cycle in~$Q_{n'+m}=Q_n$, and since $2n'=2^{r+1}=2^{\lceil\log_2 n\rceil}$, this matches the upper bound, so it is best possible.
\end{proof}

\subsection{Application to $t$-track Gray codes}

Recall the definition of $t$-track Hamilton cycles given in Section~\ref{sec:track}.
As discussed before, there is no 1-track Hamilton cycle in~$Q_n$.
We now provide a construction of a 2-track Hamilton cycle for every~$n$ that is a sum of two powers of~2.

\begin{theorem}
\label{thm:2track}
For every $n=2^r$ and $m=2^s$, where $r\ge 2$ and $r\ge s\ge 0$, there is a $2n$-symmetric Hamilton cycle in~$Q_{n+m}$ that has 2~tracks.
\end{theorem}

The 2-track Hamilton cycle in~$Q_8$ obtained from Theorem~\ref{thm:2track} for $n=m=4$ is illustrated in Figure~\ref{fig:q8}~(c).

\begin{proof}
Since $2m$ divides~$2n$, we may apply Lemma~\ref{lem:product} with the automorphism~$g$ of~$Q_n$ given by~\eqref{eq:gdef} and the automorphism~$h$ of~$Q_m$ given by $h(x):=g(x)$.
Furthermore, as a Hamilton path~$Q$ in~$Q_m$ between~$v:=0^m$ and~$h(v)=0^{m-1}1$ we can take the listing~$\Gamma_m'$ obtained from the BRGC~$\Gamma_m$ by reversing every bitstring.
In this way, we obtain a $2n$-symmetric Hamilton cycle~$C$ of~$Q_{n+m}$.
Furthermore, as the automorphism~$f$ from Lemma~\ref{lem:product} is the product of~$g$ and~$h$, and both $g$ and~$h$ cyclically shift positions, we obtain that in the matrix corresponding to~$C$, the first $n$ columns are cyclic shifts of each other, and the last $m$ columns are cyclic shifts of each other.
\end{proof}

Theorem~\ref{thm:2track} can be generalized immediately, yielding a $t$-track Hamilton cycle for every~$n$ that is a sum of $t\ge 2$ powers of~2.
This shows in particular that every dimension~$n\ge 5$ admits a Hamilton cycle with at most $\lfloor\log_2(n+1)\rfloor$ many tracks.

\begin{theorem}
\label{thm:ttrack}
For every $n=2^r$ and $(m_1,\ldots,m_{t-1})=(2^{s_1},\ldots,2^{s_{t-1}})$, where $r,t\ge 2$ and $r\ge s_1\ge \cdots\ge s_{t-1}\ge 0$, there is a $2n$-symmetric Hamilton cycle in~$Q_{n+m_1+\cdots+m_{t-1}}$ that has $t$~tracks.
\end{theorem}

\begin{proof}
The proof is analogous to the proof of Theorem~\ref{thm:2track}, using the automorphism $h$ of~$Q_{m_1+\cdots+m_{t-1}}$ that complements the first bit and then cyclically shifts groups of bits of sizes $m_1,\ldots,m_{t-1}$, each group one position to the left, and using any Hamilton path in $Q_{m_1+\cdots+m_{t-1}}$ between $v:=0^{m_1,\ldots,m_{t-1}}$ and $h(v)=0^{m_1-1}10^{m_2,\ldots,m_{t-1}}$.
\end{proof}

\section{Johnson graphs and relatives}
\label{sec:johnson}

In this section we consider Johnson graphs and middle levels graphs introduced in Section~\ref{sec:results} (recall also Section~\ref{sec:prelim-johnson}).
We first derive some upper bounds for their Hamilton compression, and then provide corresponding lower bound constructions.
For Johnson graphs, the classical construction of a Hamilton cycle is to consider the sublist of the BRGC~$\Gamma_n$ obtained by restricting to bitstrings with fixed Hamming weight~$k$ (see~\cite{MR0349274}).
The resulting cycle in~$J_{n,k}$ is only 1-symmetric in general, so we did not analyze it further (unlike the BRGC and the SJT cycles in the previous and next section, respectively, which have compression factors~$>1$).

\subsection{An upper bound}

Recall from Section~\ref{sec:prelim-johnson} that $\Aut(J_{n,k})\cong S_n$ if $n\neq 2k$ and $\Aut(J_{n,k})\cong S_n\times \mathbb{Z}_2$ if $n=2k$, so from~\eqref{eq:kappa-UB} we obtain $\kappa(J_{n,k})\le \lambda(n)$ or $\kappa(J_{n,k})\le 2\lambda(n)$, respectively, where $\lambda(n)$ is Landau's function.
We now improve these bounds drastically to linear functions (cf.~\eqref{eq:landau-asymp}).

\begin{lemma}
\label{lem:johnson-ub}
If $n\neq 2k$, then we have $\kappa(J_{n,k})\le n$.
If $n=2k$, then we have $\kappa(J_{n,k})\le 2n$.
\end{lemma}

\begin{proof}
We first consider the case $n\neq 2k$.
Let $f=\pi$ be any automorphism of~$J_{n,k}$, and consider a fixed cycle decomposition~$C_1,\ldots,C_m$ of the permutation~$\pi$.
Consider the permutation~$\rho$ of~$[n]$ obtained by `flattening' the lists~$C_1,\ldots,C_m$.
For example, if $n=9$ and $\pi=(1,5,4)(2,9,7,8)(3,6)$ we have $\rho=154297836$.
Let $x=x_1\cdots x_n$ be the vertex of~$J_{n,k}$ defined by $x_{\rho(i)}:=1$ for $i=1,\ldots,k$ and $x_{\rho(i)}:=0$ for $i=k+1,\ldots,n$.
By definition, for all $i=1,\ldots,m$ except possibly one index $i=s$, we have that the entries of~$x$ on the indices of~$C_i$ are all the same (either all~1s or all~0s), whereas on the indices of~$C_s$ we see both~1s and~0s in~$x$.
It follows that the size of the orbit of~$x$ under~$\pi$ is at most~$|C_s|\le n$ (if there is no exceptional cycle~$C_s$, then the orbit has size~1).
This proves the first part of the lemma.

The proof of the second part is analogous, and here the additional factor of~2 comes from the possible complementation operation.
\end{proof}

As the automorphism group of the middle levels graph is also~$\Aut(J_{n,k})\cong S_n\times \mathbb{Z}_2$, the same proof idea immediately gives an analogous upper bound of twice the length of the bitstrings.

\begin{lemma}
\label{lem:middle-ub}
For all $n\ge 1$ we have $\kappa(M_{2n+1})\le 2(2n+1)$.
\end{lemma}

\subsection{Known construction for middle levels graphs}

The following is the main result of~\cite{MR4262479}.

\begin{theorem}[\cite{MR4262479}]
\label{thm:knuth}
For all $n\ge 1$, the graph~$M_{2n+1}$ has a $(2n+1)$-symmetric Hamilton cycle.
\end{theorem}

Cycles obtained from Theorem~\ref{thm:knuth} are shown in Figure~\ref{fig:comp4}~(b) and Figure~\ref{fig:mlc7}~(a).
These cycles also have the 1-track property and they are balanced (as the underlying automorphism is cyclic rotation of all bits).
From this we can determine the Hamilton compression of~$M_{2n+1}$ up to a factor of~2.

\begin{theorem}
\label{thm:kappa-middle}
For all $n\ge 1$ we have $2n+1\le \kappa(M_{2n+1})\le 2(2n+1)$.
\end{theorem}

\begin{proof}
The upper bound comes from Lemma~\ref{lem:middle-ub}, and the lower bound from Theorem~\ref{thm:knuth}.
\end{proof}

Interestingly, both bounds in Theorem~\ref{thm:kappa-middle} can sometimes be improved.
For example, in dimension~7 we can take the automorphism~$f$ defined by $x_1\cdots x_7\mapsto\overline{x_1 x_2 x_4 x_5 x_6 x_7 x_3}$, which fixes the first two bits, cyclically left-shifts the remaining five bits by one position, and then complements all bits.
A 10-symmetric Hamilton cycle under this~$f$ is shown in Figure~\ref{fig:mlc7}~(b), whereas the lower and upper bounds are~7 and~14, respectively.
In fact, computer experiments show that $\kappa(M_7)=10$.

\begin{figure}
\makebox[0cm]{ 
\begin{tabular}{ccc}
\raisebox{39mm}{(a)} &
\includegraphics{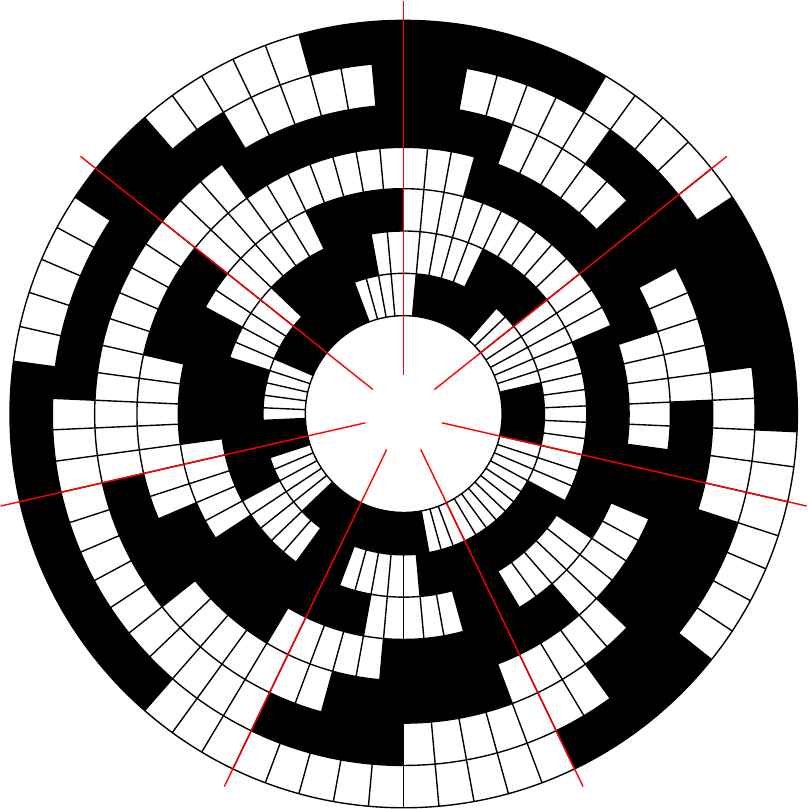} &
\includegraphics{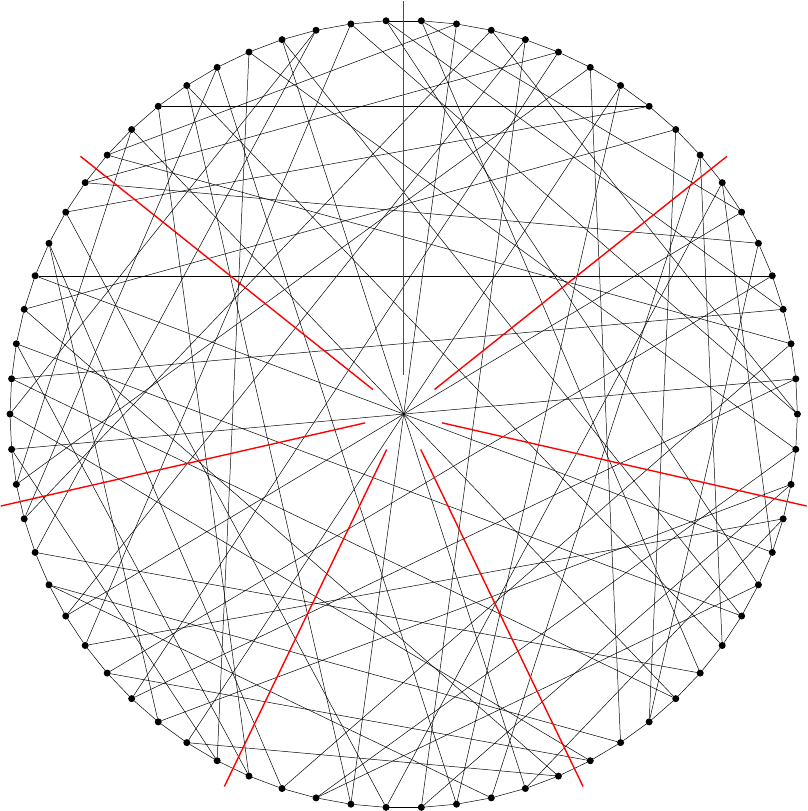} \\
\raisebox{39mm}{(b)} &
\includegraphics{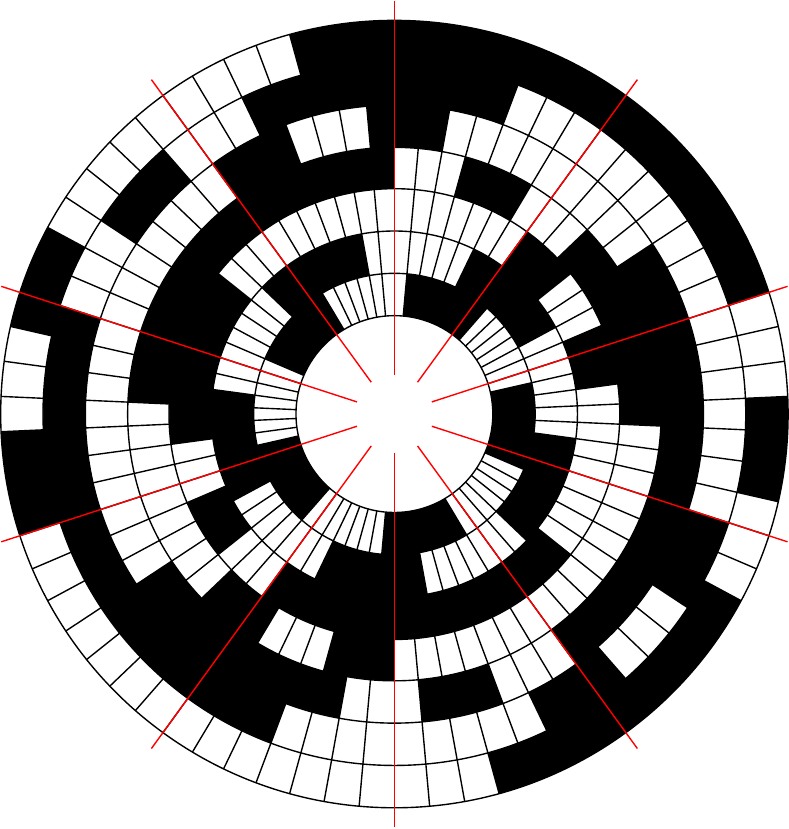} &
\includegraphics{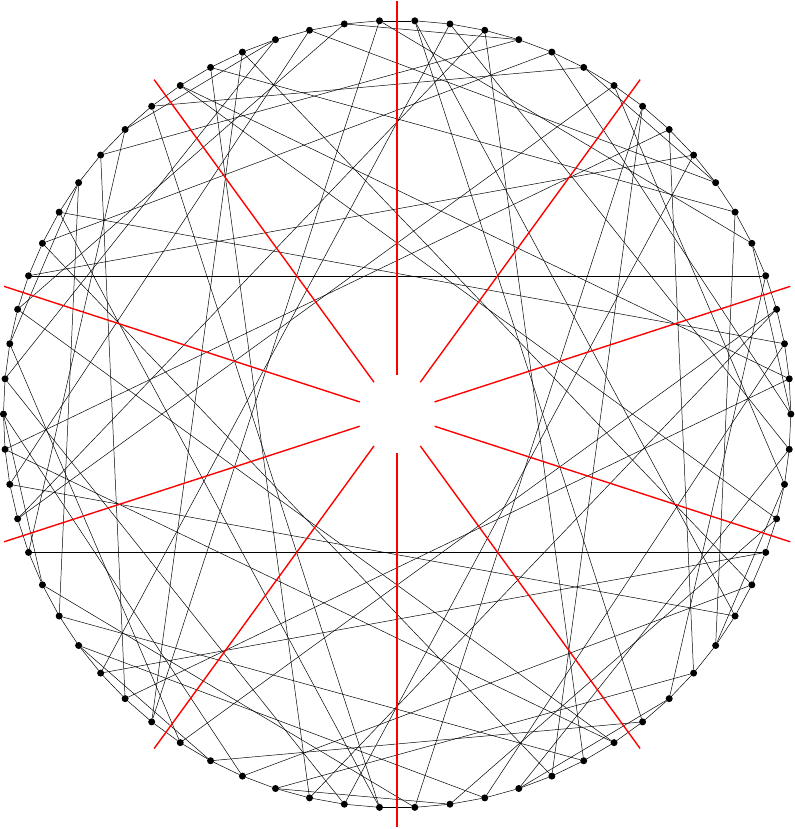} \\
\end{tabular}
}
\caption{Symmetric Hamilton cycles in the middle levels graph~$M_7$:
(a)~A solution to Knuth's problem with compression~7 for $f$ being cyclic left-shift;
(b)~Hamilton cycle with compression~10 for $f$ being left-shift of the last 5 bits and complementation of all bits.}
\label{fig:mlc7}
\end{figure}

\subsection{A near-optimal construction for Johnson graphs}

Our next result provides Hamilton cycles in~$J_{n,k}$ with optimal compression for the case when $n$ and~$k$ are coprime (this in particular means that $n\neq 2k$).

\begin{theorem}
\label{thm:johnson-coprime}
Let $n>k>0$ be such that $n$ and~$k$ are coprime.
Then $J_{n,k}$ has an $n$-symmetric Hamilton cycle that has 1~track and is balanced, i.e., each bit is flipped equally often ($\binom{n}{k}/n$ many times).
\end{theorem}

This construction is illustrated in Figure~\ref{fig:jnk}~(a) for~$J_{11,3}$.

\begin{proof}
For $k=1$ and $k=n-1$ the Johnson graph~$J_{n,k}$ is the complete graph~$K_n$, so the statement is trivial.
For the rest of the proof we therefore assume that~$n-1>k>1$.

We use the automorphism~$f$ that cyclically left-shifts all bits by one position.
The orbits of~$f$ are necklaces, and as $n$ and $k$ are coprime, every necklace has the same size~$n$.
Let~$P$ be the path in~$J_{n,k}$ guaranteed by Theorem~\ref{thm:necklace} from~$x:=1^k0^{n-k}$ to~$y:=1^{k-1}010^{n-k-1}$.
Note that $y$ is adjacent to $f(x)=1^{k-1}0^{n-k}1$.
Consequently, $C:=P,f(P),f^2(P),\ldots,f^{n-1}(P)$ is an $n$-symmetric Hamilton cycle in~$J_{n,k}$.
Furthermore, any two columns of the $|C|\times n$ matrix corresponding to~$C$ are cyclic shifts of each other, so $C$ has the 1-track property.
This immediately implies that every bit is flipped equally often.
\end{proof}

\begin{figure}
\makebox[0cm]{ 
\begin{tabular}{cc}
\includegraphics{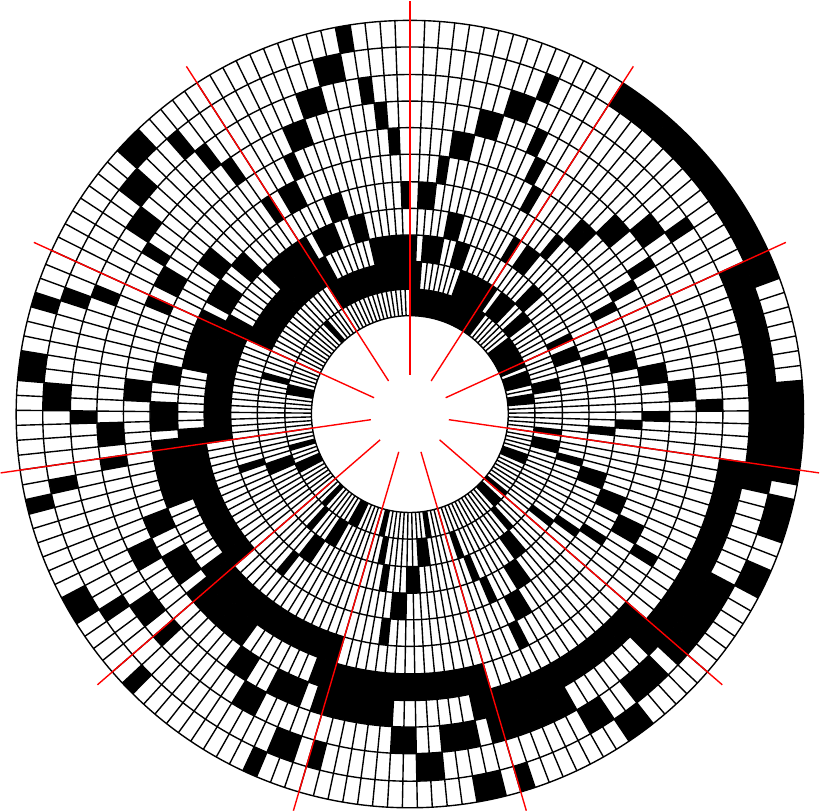} &
\includegraphics{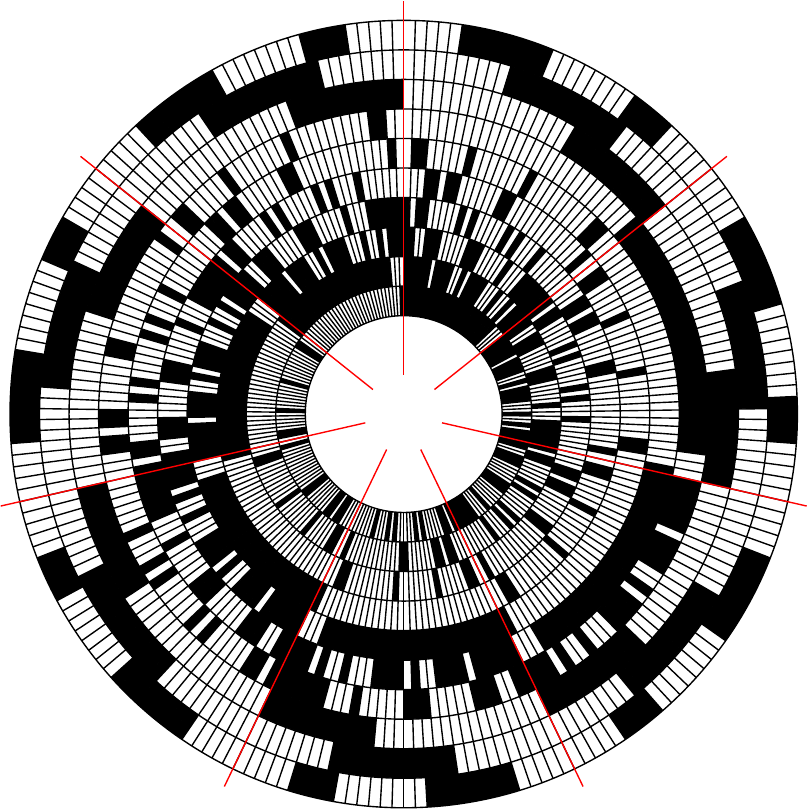} \\
(a) & (b)
\end{tabular}
}
\caption{Symmetric Hamilton cycles in Johnson graphs:
(a)~Balanced 1-track Hamilton cycle in~$J_{11,3}$ with compression~$n=11$; the automorphism left-shifts all $n$ bits;
(b)~4-track Hamilton cycle in~$J_{10,4}$ with compression~$q=7$; the automorphism left-shifts the first $q$ bits.}
\label{fig:jnk}
\end{figure}

When $n$ and $k$ are not coprime, we can slightly modify the automorphism used to prove the previous theorem.
Instead of cyclically shifting all $n$ bits, we now shift only the first $q$ bits, leaving the last $n-q$ bits unchanged.
The parameter $q$ is chosen as close to $n$ as possible, but it has to satisfy certain coprimality conditions that are needed so that all necklaces have the same size~$q$.
This construction is illustrated in Figure~\ref{fig:jnk}~(b) for~$J_{10,4}$ and $q=7$.

\begin{theorem}
\label{thm:johnson-other}
Let $q<n$ be such that $q>\max\{k,n-k\}$ and $q$ and~$\ell$ are coprime for all $\ell=k-(n-q),\ldots,k$.
Then $J_{n,k}$ has a $q$-symmetric Hamilton cycle that has $1+n-q$ tracks.
\end{theorem}

The number of tracks could be reduced to at most $1+\lfloor\log_2(n-q+1)\rfloor$ tracks with the help of Theorem~\ref{thm:ttrack}.

\begin{proof}
Let $r:=n-q$.
We use the automorphism~$f$ that cyclically left-shifts the first~$q$ bits by one position and leaves the last $r$ bits unchanged, i.e., $f(x_1\cdots x_n)=x_2x_3\cdots x_q x_1 x_{q+1}\cdots x_n$.
By the assumption $q>\max\{k,n-k\}$, every $(n,k)$-combination has both 0-bits and 1-bits among the first $q$ bits, and as $q$ and $\ell$ are coprime for all $\ell=k-r,\ldots,k$ by the assumptions in the lemma, all orbits of~$f$ have the same size~$q$ (independent of~$\ell$).

The different values of~$\ell$ specify the number of 1-bits among the first $q$ positions.
Specifically, for $\ell=k-r,\ldots,k$, let $R_\ell$ be the path on necklaces for $(q,\ell)$-combinations guaranteed by Theorem~\ref{thm:necklace}, i.e., $R_\ell$ starts at $1^\ell0^{q-\ell}$ and ends at~$1^{\ell-1}010^{q-\ell-1}$.
For the special cases $\ell=1$ (i.e., $q=n-k+1$), or $\ell=k$ and $q=k+1$ the sequence~$R_\ell$ is just the single vertex~$10^{q-1}$ or $1^k0$, respectively.
Furthermore, let $Q=(x_1,\ldots,x_N)$, $N:=2^r$, a Hamilton cycle in~$Q_r$, for example the BRGC $Q=\Gamma_r$, which satisfies $x_1=0^r$.

In the following we write $w(x_i)$ for the Hamming weight of~$x_i$.
We now define a path
\begin{equation*}
P:=R_1 x_1,R_2 x_2,\ldots,R_{N-1}x_{N-1}R_N'x_N
\end{equation*}
where $R_i:=R_\ell$ with $\ell:={k-w(x_i)}$ for $i=1,\ldots,N-1$, and $R_N'$ is obtained from~$R_\ell$, $\ell:=k-w(x_N)$, by reversing all bitstrings and cyclically shifting them by $q-\ell$ positions to the left so that $R_N'$ starts at $1^\ell 0^{q-\ell}$ and ends at $01^{\ell-1}0^{q-\ell-1}1$.
By construction, for $i=1,\ldots,N-1$ the last vertex of~$R_i x_i$ differs from the first vertex of~$R_{i+1}x_{i+1}$ by a transposition of~0 and~1.
Furthermore, by our choice of~$x_1=0^r$, the first vertex of~$P$ is $x:=1^k0^{q-k}x_1=1^k0^{q-k}0^r$ and as $w(x_N)=1$ the last vertex of~$P$ is $y:=01^{k-2}0^{q-k}1x_N$, i.e., $y$ is connected to $f(x)=1^{k-1}0^{q-k}10^r$.

We conclude that $C:=P,f(P),f^2(P),\ldots,f^{q-1}(P)$ is a $q$-symmetric Hamilton cycle in~$J_{n,k}$.
In the corresponding $|C|\times n$ matrix, the first $q$ columns are cyclic shifts of each other, so they form one track, and each of the remaining $r$ columns is its own track.
The total number of tracks is therefore $1+r=1+n-q$.
\end{proof}

The next lemma provides the good news that an integer~$q$ satisfying the conditions of Theorem~\ref{thm:johnson-other} exists for all~$n$ and~$k$.

\begin{lemma}
\label{lem:q-exists}
For all $n>k>0$ there is an integer~$q\le n$ such that $q>\max\{k,n-k\}$ and $q$ and~$\ell$ are coprime for all $\ell=k-(n-q),\ldots,k$.
\end{lemma}

One natural idea would be to take $q$ as a prime number with~$k<q\le n$, this would automatically guarantee that $q$ is coprime to all $\ell=k-(n-q),\ldots,k$.
However, as the integers contain arbitrarily long intervals of non-primes, such a choice is not always possible, so we need to argue differently.
It turns out that the proof of Lemma~\ref{lem:q-exists} requires some delicate number-theoretic reasoning, so we defer it to the appendix.
For $n>k>0$ we write $q(n,k)$ for the largest integer~$q$ satisfying the conditions of Lemma~\ref{lem:q-exists}.
The function $n-q(n,k)$ is visualized in the appendix for small values of~$n$ and~$k$.
This is the difference of the compression factor~$q$ achievable by Theorems~\ref{thm:johnson-coprime} and~\ref{thm:johnson-other} and the upper bound~$n$ provided by Lemma~\ref{lem:johnson-ub}.

We now combine the results of this section to determine the Hamilton compression of~$J_{n,k}$ exactly or almost exactly.

\begin{theorem}
\label{thm:kappa-Jnk}
The Hamilton compression of the Johnson graph~$J_{n,k}$, where $n>k>0$, has the following properties:
\begin{enumerate}[label=(\roman*),leftmargin=8mm, topsep=0mm, noitemsep]
\item If $n$ and $k$ are coprime, we have $\kappa(J_{n,k})=n$.
\item If $n$ and $k$ are not coprime and $n\neq 2k$, we have $n/2\!<\!\max\{k,n-k\}\!<\!q(n,k)\!\le\! \kappa(J_{n,k})\!\le\! n$.
\item If $n$ and $k$ are not coprime and $n=2k$, we have $n/2<\kappa(J_{n,k})\le 2n$.
\item For any $\varepsilon>0$ there is an~$n_0$ such that for all $n>n_0$ with $n\neq 2k$ we have $(1-\varepsilon)n\le \kappa(J_{n,k})\le n$.
In particular, we have $\kappa(J_{n,k})=(1-o(1))n$ for $n\neq 2k$.
\end{enumerate}
\end{theorem}

\begin{proof}
The upper bounds in~(i)--(iii) are from Lemma~\ref{lem:johnson-ub}.
The lower bound in~(i) is from Theorem~\ref{thm:johnson-coprime}.
The lower bound in~(ii)+(iii) is from Theorem~\ref{thm:johnson-other}, using Lemma~\ref{lem:q-exists}.

The last part~(iv) follows from~(ii).
Specifically, the prime number theorem guarantees that for any $\varepsilon$ there is an~$n_0$ such that for all $n>n_0$ there is a prime number in the interval~$[(1-\varepsilon)n,n]$, which shows that $q(n,k)\geq (1-\varepsilon)n$.
We conclude that $q(n,k)\geq (1-o(1))n$.
\end{proof}

For quantitative bounds on the dependence of~$n_0$ on~$\varepsilon$ in Theorem~\ref{thm:kappa-Jnk}~(iv), see e.g.~\cite{MR3745073}.

\section{Permutahedra}
\label{sec:perm}

We now consider the family of permutahedra~$\Pi_n$ introduced in Section~\ref{sec:results} (recall also Sections~\ref{sec:prelim-perm}--\ref{sec:landau}).
We first show that the Steinhaus-Johnson-Trotter cycle has constant compression (6 or 3).
We then establish a mildly exponential upper bound for~$\kappa(\Pi_n)$ that involves the Landau function~$\lambda(n)$ and its variants, and we provide a near-optimal lower bound construction.
Lastly, we apply our constructions to derive a balanced 1-track Gray code for permutations of odd length that uses cyclically adjacent transpositions.

\subsection{The Steinhaus-Johnson-Trotter (SJT) cycle}

Recall the definition of the SJT cycle~$\Lambda_n$ given in Section~\ref{sec:results-perm}.

\begin{proposition}
\label{prop:perm}
The SJT cycle~$\Lambda_n$ has compression
\begin{equation*}
\kappa(\Pi_n,\Lambda_n)=\begin{cases} 6 & n=3,4, \\ 3 & n\ge 5. \end{cases}
\end{equation*}
\end{proposition}

The SJT cycle~$\Lambda_n$ is illustrated in Figure~\ref{fig:p4}~(a) and Figure~\ref{fig:p5}~(a) for $n=4$ and $n=5$, respectively, and those two pictures indeed have 6-fold or 3-fold rotational symmetry.
While the latter might seem to have 6-fold symmetry at first glance, a careful inspection of the short chords shows that this is not the case.

\begin{proof}
The cycle~$\Lambda_3=123,132,312,321,231,213$ is a 6-cycle, and the automorphism $f=(\alpha,\pi)$ of~$\Pi_3$ with $\alpha=\rev$ and $\pi=231$ shows that it is 6-symmetric (recall~\eqref{eq:f-alpha-pi}).
The cycle~$\Lambda_4$ is shown in Figure~\ref{fig:p4}~(a), and the automorphism $f=(\alpha,\pi)$ of~$\Pi_4$ with $\alpha=\rev$ and $\pi=2314$ shows that it is 6-symmetric.

Note that $f^2=(\ide,3124)$ is another automorphism of~$\Pi_4$, which shows that $\Lambda_4$ is 3-symmetric, and we will now generalize this automorphism to show that~$\Lambda_n$, $n\ge 5$, has compression at least~3.
Note that $(n-1)!/3$ is even for $n\ge 5$, so by the definition of~$\Lambda_n$ the value~$n$ is at the rightmost position in the permutations of~$\Lambda_n$ at positions $1,N/3+1,2N/3+1$, where $N:=n!$.
Consequently, for $n\ge 5$ the automorphism~$f=(\ide,31245\cdots n)$ of~$\Pi_n$ shows that~$\Lambda_n$ has compression at least~3.

It remains to show that the compression of~$\Lambda_n$ is at most~3 for $n\ge 5$.
Let $f=(\alpha,\pi)$ be any automorphism of~$\Lambda_n$ for~$n\ge 4$, and consider the movement of the largest value~$n$ in the sequence~$\Lambda_n$.
By definition of~$\Lambda_n$, we alternately see the following two patterns: $n-1$ transpositions involving~$n$ and a smaller value, followed by a transposition involving two smaller values.
In particular, every value~$i<n$ is involved in at most~2 adjacent transpositions consecutively.
It follows that for $n\ge 4$, the permutation~$\pi$ must map $i\mapsto i$ for all $i=n,n-1,\ldots,4$.
Note that in~$\Lambda_4$ the value~4 reaches the leftmost and rightmost position precisely 3 times each (and stays at this position for one step each time).
Consequently, to complete the proof it suffices to show that~$f$ does not yield the mapping $x_i\mapsto x_{i+N/6}$ for all $i=1,\ldots,N$ when $n\ge 5$.
To see this note that~4 and~5 are adjacent in~$x_1=12345\cdots n$, but non-adjacent in $x_{1+N/6}=41325\cdots n$ (and $\pi$ preserves both values).
\end{proof}

\subsection{An upper bound}

Recall from Section~\ref{sec:prelim-perm} that $\Aut(\Pi_n) \cong \mathbb{Z}_2 \ltimes S_n$, so from~\eqref{eq:kappa-UB} we obtain $\kappa(\Pi_n)\le 2\lambda(n)$, where $\lambda(n)$ is Landau's function.
The next lemma improves this bound, based on a parity argument.
Recall the definitions of~$\lambda_0(n)$ and~$\lambda_2(n)$ from~\eqref{eq:landau-even}.

\begin{lemma}
\label{lem:perm-ub}
Let $n\ge 4$.
If $n\equiv 0,1 \bmod 4$, then we have $\kappa(\Pi_n)\le \max\{2\lambda_0(n),\lambda_2(n)\}$.
If $n\equiv 2,3 \bmod 4$, then we have $\kappa(\Pi_n)\le \lambda(n)$.
\end{lemma}

The relation between $2\lambda_0(n)$ and $\lambda_2(n)$ is somewhat unclear.
While for~$n\le 37$ we have $2\lambda_0(n)>\lambda_2(n)$, for increasing values of~$n$ it seems that more and more values satisfy $2\lambda_0(n)<\lambda_2(n)$, for example $n=38,49,50,51,52,53,54,55,66,67,68,69,70,71,72,73,74,80,85,86,87,88,\allowbreak 89,90,91,92,93,94,95,96,97,99,100$ (see the appendix).

\begin{proof}
Consider an automorphism~$f=(\alpha,\pi)$ of~$\Pi_n$ and a path~$P=(x_1,\ldots,x_\ell)$, $\ell:=n!/k$, such that $C=P,f(P),f^2(P),\ldots,f^{k-1}(P)$ is a $k$-symmetric Hamilton cycle in~$\Pi_n$, i.e., $k=\ord(f)$.

We clearly have $\ord(f)\in\{\ord(\pi),2\ord(\pi)\}$.
Furthermore, note that if $\ord(f)=2\ord(\pi)$, i.e., $\alpha=\rev$, then $\ord(\pi)$ must be odd, otherwise we would have $f^{k/2}=f^{\ord(\pi)}=\pi^{\ord(\pi)}=\ide$, a contradiction.

We first show that $\ell$ is even.
Let $2^i$ be the largest power of~2 in~$\ord(\pi)$.
Note that $2^i$ also divides the length of one of the cycles of~$\pi$, so~$2^i\le n$.
Consequently, if $n\ge 6$, then $n!/2^{i+1}$ is even, as the product~$n!$ contains the factor~$2^i$ and two additional even factors.
This implies that $\ell=n!/k=n!/\ord(f)$ is even, as $\ord(f)\in\{\ord(\pi),2\ord(\pi)\}$.
For $n=4,5$ the product~$n!$ contains the factor~$2^i$ and one additional even factor, implying that $n!/2^i$ is even and therefore $\ell$ is even unless $\ord(f)=2\ord(\pi)$.
However, in the latter case $\ord(\pi)$ must be odd, i.e., we have $i=0$, showing that $n!/2^{i+1}=n!/2$ and therefore $\ell$ is even.

As $\ell$ is even and $x_i$ has opposite parity to~$x_{i+1}$ for all $i=1,\ldots,\ell$, we obtain that $x_1$ and~$x_\ell$ have opposite parity, and $x_1$ and~$x_{\ell+1}=f(x_1)$ have the same parity (recall~\eqref{eq:CP}), showing that the mapping $f(x)=\alpha x \pi$ preserves parity (recall~\eqref{eq:f-alpha-pi}).

If $n\equiv 0,1 \bmod 4$, then $\binom{n}{2}$ is even, so $\alpha\in\{\ide,\rev\}$ is an even permutation.
Therefore $\pi$ must be an even permutation, i.e., the number of even cycles of~$\pi$ is even.
If $\pi$ consists of only odd cycles, then $k\le 2\ord(\pi)\le 2\lambda_0(n)$ by the definition~\ref{eq:landau0}.
Otherwise, $\ord(\pi)$ will be even, and therefore $k=\ord(\pi)\le\lambda_2(n)$ by our initial observations about the parity of~$\ord(\pi)$ and by the definition~\ref{eq:landau2}.
The maximum of the bounds obtained in both cases yields the claimed bound.

If $n\equiv 2,3 \bmod 4$, then $\binom{n}{2}$ is odd, so $\rev$ is an odd permutation.
If $\alpha=\ide$, then $k=\ord(\pi)\le \lambda(n)$.
If $\alpha=\rev$, then $\pi$ must be an odd permutation, i.e., the number of even cycles of~$\pi$ is odd.
In particular, there is at least one even cycle, which implies that $\ord(\pi)$ is even, so $k=\ord(\pi)\le\lambda(n)$ as well.
\end{proof}

\subsection{A near-optimal construction}

Our aim is to use an automorphism of~$\Pi_n$ whose order is close to Landau's function~$\lambda(n)$.
Let $\ba=(a_1,\ldots,a_m)$ be a composition of~$n$ and let $A_1 \cup \cdots \cup A_m$ be its associated partition of~$[n]$ as defined in Section~\ref{sec:mperm}.
By definition, we have $A_i=\{b_i+1, \ldots, b_i+a_i\}$ where $b_i:=\sum_{j=1}^{i-1}a_j$.
So $f_i:=(b_i+1,\ldots,b_i+a_i)$ is a cyclic permutation of~$A_i$.
Its $(|A_i|-1)!$ orbits can be represented by fixing a value from~$A_i$ on a particular position and permuting the remaining values in the remaining positions in all possible ways.
For example, if $f=(1,2,3)$, then by fixing the value~3 on the last position we have representatives~$123, 213$, by fixing the value~1 on the first position we have representatives~$123,132$, and in both cases these are representatives of the two orbits $\langle 123\rangle_f=\{123,231,312\}$ and $\langle 213 \rangle_f=\langle 132 \rangle_f=\{213,321,132\}$.
Let us define an automorphism $f_\ba$ of $\Pi_n$ as a product of the cyclic permutations $f_i$ applied on values (without reversal of positions), i.e.,
\begin{equation}\label{eq:fa}
  f_\ba:=(\ide,f_1 f_2 \cdots f_m).
\end{equation}
First we need to consider the orbits of~$f_\ba$.
For the following characterization recall the definitions from Sections~\ref{sec:prelim-graphs} and~\ref{sec:mperm}.
This lemma follows immediately from the Chinese remainder theorem.

\begin{lemma}
\label{lem:fa-orbits}
Let $\ba=(a_1,\ldots,a_m)$ be a composition of~$n$ and let $f_\ba$ be as in~\eqref{eq:fa}.
\begin{enumerate}[label=(\roman*),leftmargin=8mm, topsep=0mm, noitemsep]
\item
If $a_1,\ldots,a_m$ are pairwise coprime, then the orbits of~$f_\ba$ are
\begin{equation}
\label{eq:fa-orb1}
\Big\{u\otimes \big(X_1\times \cdots \times X_m\big) \mid u\in {\textstyle\binom{[n]}{\ba}} \text{ and } (X_1,\ldots,X_m)\in O(f_1)\times \cdots \times O(f_m) \Big\}.
\end{equation}
\item
If $m\ge 2$, $a_1=2$, $a_2=2^c$ for some $c\ge 1$, and $a_2,\ldots, a_m$ are pairwise coprime, then the orbits of~$f_\ba$ are
\begin{equation}
\label{eq:fa-orb2}
\begin{split}
&\Big\{u\otimes \big(B\times X_3\times \cdots \times X_m\big),u\otimes \big(B'\times X_3\times \cdots \times X_m\big) \mid u\in {\textstyle\binom{[n]}{\ba}} \text{ and } \\ &\hspace{60mm}(X_3,\ldots,X_m)\in O(f_3)\times \cdots \times O(f_m) \Big\},
\end{split}
\end{equation}
where $B:=\{12\}\times X_2^0\,\cup\, \{21\}\times X_2^1$, $B':=\{12\}\times X_2^1\,\cup\, \{21\}\times X_2^0$, and $X_2^0\cup X_2^1$ is the partition of~$X_2=:\{x,f_2(x),f_2^2(x),\ldots,f_2^{a_2-1}(x)\}$ defined by $X_2^0:=\{f_2^{2i}(x)\mid i=0,\ldots,(a_2-2)/2\}$ and $X_2^1:=\{f_2^{2i-1}(x)\mid i=1,\ldots,a_2/2\}$, respectively.
\end{enumerate}
\end{lemma}

To illustrate part~(i) of the lemma with an example, for $\ba=(2,3)$ we have
\begin{equation*}
{\textstyle\binom{[n]}{\ba}}=\{11222,12122,12212,12221,21122,21212,21221,22112,22121,22211\},
\end{equation*}
$O(f_1)=\{\{12,21\}\}$ and $O(f_2)=\{\{345,453,534\},\{543,435,354\}\}$ and hence the orbits of~$f_\ba$ are
\begin{equation*}
\Big\{u\otimes \big(\{12,21\}\times \{345,453,534\}\big) \mid u\in {\textstyle\binom{[n]}{\ba}} \Big\}
\;\cup\; \Big\{u\otimes \big(\{12,21\}\times \{543,435,354\}\big) \mid u\in {\textstyle\binom{[n]}{\ba}} \Big\}.
\end{equation*}
The first two of those 20 orbits, namely those corresponding to $u=11222$ are $\{12345,21453,12534,\allowbreak 21345,12453,21534\}$ and $\{12543,21453,12534,21543,12435,21354\}$.

Using Lemma~\ref{lem:fa-orbits}, we now build a path in~$\Pi_n$ on representatives of orbits where each permutation~$f_i$ of values~$A_i$ acts on the positions~$A_i$.
That is, the orbits are obtained from the orbits of the~$f_i$ by mixing with the identity permutation~$\ide(\ba)$.

\begin{lemma}
\label{lem:id-mix}
Let $\ba=(a_1,\ldots,a_m)$ be a composition of~$n$ and let $f_\ba$ be as in~\eqref{eq:fa}.
\begin{enumerate}[label=(\roman*),leftmargin=8mm, topsep=0mm, noitemsep]
\item
If $a_1,\ldots,a_m$ are pairwise coprime and odd, and $a_1\ge 3$, then $\Pi_n$ contains a path~$Q$ that starts at~$\ide=1\cdots n$, ends at a neighbor of~$f_\ba(\ide)$ and that visits each orbit of~$f_\ba$ from~\eqref{eq:fa-orb1} with $u=\ide(\ba)$ exactly once.
\item
If $m\ge 2$, $a_1=2$, $a_2=2^c$ for some $c\ge 1$, and $a_2,\ldots, a_m$ are pairwise coprime, then the same conclusion holds for the orbits from~\eqref{eq:fa-orb2} with $u=\ide(\ba)$.
\end{enumerate}
\end{lemma}

Note that in part~(i) of the lemma, the $a_i$ are not only required to be pairwise coprime, but also odd (unlike in part~(i) of Lemma~\ref{lem:fa-orbits}).
In part~(ii), all $a_i$, $i=3,\ldots,m$, are odd by the assumption that they are coprime to~$a_2=2^c$.

\begin{proof}
We prove both statements by induction on~$m$, the only difference being the base case.

We first consider the base case~$m=1$ for part~(i).
We have $\ba=(n)$, $f_\ba=f_1=(1,\ldots,n)$, and $\ide(\ba)=1^n$ where $n\ge 3$ is odd.
As orbit representatives we choose permutations that have the value~$n$ fixed in the last position, using that $O(f_1)=\{\langle xn \rangle_{f_1} \mid x\in S_{n-1}\}$.
By Theorem~\ref{thm:lace}, there is a Hamilton path~$P$ in~$\Pi_{n-1}$ from $x:=1\cdots (n-1)$ to $y:=2\cdots (n-1)1$ (note that $y$ is odd as $n$ is odd).
Thus $\Pi_n$ contains the path~$Q:=Pn$ from~$xn=\ide$ to~$yn=2\cdots (n-1)1n$, which is adjacent to $f_\ba(\ide)=2\cdots (n-1)n1$ by a transposition of the last two entries.
Moreover, $Q$ visits every orbit in $\bigcup_{X_1\in O(f_1)}\ide(\ba)\otimes X_1=O(f_1)$ exactly once.

We now consider the base case~$m=2$ for part~(ii).
We have $\ba=(2,2^c)$, $f_\ba=f_1f_2=(1,2)(3,\ldots,n)$ for $n:=2+2^c$, and $\ide(\ba)=1^22^{n-2}$.
As before, we choose orbit representatives with value~$n$ at the last position, using that $O(f_2)=\{\langle xn\rangle_{f_2} \mid x\in S_{\{3,\ldots,n-1\}}\}$.
If $c=1$, then $n=4$, and in this case $Q:=(1234,2134)$ is the desired path in~$\Pi_n$ since $2134$ is adjacent to $f_\ba(1234)=2143$.
Otherwise, we have $c\ge 2$ and therefore $n=2+2^c\ge 6$.
By Theorem~\ref{thm:lace}, there are Hamilton paths~$P,P'$ in $\Pi_{\{3,\ldots,n-1\}}$ from~$x:=34\cdots (n-1)$ to~$y:=435\cdots (n-1)$ and from~$y':=435\cdots(n-1)$ to~$x':=4\cdots (n-1)3$, respectively (note that $x'$ is even as $n$ is even).
Thus $\Pi_n$ contains the path $Q:=12Pn,21P'n$ from~$12xn=\ide$ to~$21x'n=214\cdots (n-1)3n$, which is adjacent to $f_\ba(\ide)=214\cdots (n-1)n3$ by a transposition of the last two entries.
Moreover, $Q$ visits every orbit in~$\{\ide(\ba)\otimes B\}\cup \{\ide(\ba)\otimes B'\}$ exactly once ($B$ and~$B'$ are as defined after~\eqref{eq:fa-orb2}).

For the induction step consider a composition~$(a_1,\ldots,a_m)$ of~$n$ as in the lemma, with $m\ge 2$ in part~(i) and $m\ge 3$ in part~(ii), and define $\ba':=(a_1,\ldots,a_{m-1})$ and $n':=n-a_m$.
By induction there is a path~$Q'=(z_1,\ldots,z_\ell)$ in~$\Pi_{n'}$ from $z_1=\ide':=1\cdots n'$ to a neighbor~$z_\ell$ of $f_{\ba'}(\ide')$ that visits each orbit of~$f_{\ba'}$ exactly once.
If $a_m=1$ then $Q:=Q'n$ is the desired path in~$\Pi_n$ from~$z_1n=\ide$ to~$z_\ell n$, which is a neighbor of $f_\ba(\ide)=f_{\ba'}(\ide')n$, and we are done.
Otherwise, we have $a_m\ge 3$ and we use two sets of representatives of orbits of~$f_m$, namely
\begin{equation*}
O(f_m)=\{\langle xn \rangle_{f_m} \mid x\in S_{\{n'+1,\ldots,n-1\}}\}=\{\langle(n'+2)x\rangle_{f_m} \mid x\in S_{\{n'+1,n'+3,\ldots,n\}}\}.
\end{equation*}
By Theorem~\ref{thm:lace}, there is a Hamilton path $P$ in~$\Pi_{\{n'+1,\ldots,n-1\}}$ from $x:=(n'+1)\cdots(n-1)$ to $y:=(n'+2)(n'+1)(n'+3)\cdots(n-1)$ and a Hamilton path~$P'$ in~$\Pi_{\{n'+1,n'+3,\ldots,n\}}$ from $y':=(n'+1)(n'+3)\cdots n$ to $x':=(n'+3)\cdots n(n'+1)$.
Using that $yn=(n'+2)y'$ we obtain that
\begin{equation}
\label{eq:QPP'}
Q:=\big(z_1Pn,z_2\lvec Pn, \ldots,z_{\ell-3}Pn,z_{\ell-2}\lvec Pn,z_{\ell-1}Pn,z_\ell(n'+2)P'\big)
\end{equation}
is a path in~$\Pi_n$ from $z_1xn=\ide$ to $z_\ell (n'+2)x'=z_\ell (n'+2)(n'+3)\cdots n(n'+1)$, which is adjacent to $f_\ba(\ide)=f_{\ba'}(\ide')(n'+2)(n'+3)\cdots n(n'+1)$, as $z_\ell$ is adjacent to $f_{\ba'}(\ide')$ in~$\Pi_{n'}$.
Note that~\eqref{eq:QPP'} requires $\ell$ to be even, which is satisfied as the number of orbits of~$f_1$ is even in part~(i) and the number of orbits of~$f_1f_2$ is even in part~(ii).
By this construction and the induction hypothesis, the path~$Q$ visits every orbit from~\eqref{eq:fa-orb1} or~\eqref{eq:fa-orb2} with $u=\ide(\ba)$ exactly once.
\end{proof}

\begin{theorem}
\label{thm:Pin}
Let $\ba=(a_1,\ldots,a_m)$ be a composition of~$n\ge 3$ and let $f_\ba$ be as in~\eqref{eq:fa}.
If
\begin{enumerate}[label=(\roman*),leftmargin=8mm, topsep=0mm, noitemsep]
\item
$a_1,\ldots,a_m$ are pairwise coprime and odd, and $a_1\ge 3$, or
\item
$m\ge 4$, $a_1=2$, $a_2=2^c$ for some $c\ge 1$, and $a_2,\ldots,a_m$ are pairwise coprime,
\end{enumerate}
then $\Pi_n$ has a $k$-symmetric Hamilton cycle for $f_\ba$ for $k:=\lcm(\ba)$.
\end{theorem}

The 5-symmetric Hamilton cycle in~$\Pi_5$ obtained from Theorem~\ref{thm:Pin}~(i) for $\ba=(5)$ is illustrated in Figure~\ref{fig:p5}~(b).

\begin{figure}
\makebox[0cm]{ 
\begin{tabular}{ccc}
\raisebox{37mm}{(a)} &
\includegraphics[scale=0.9]{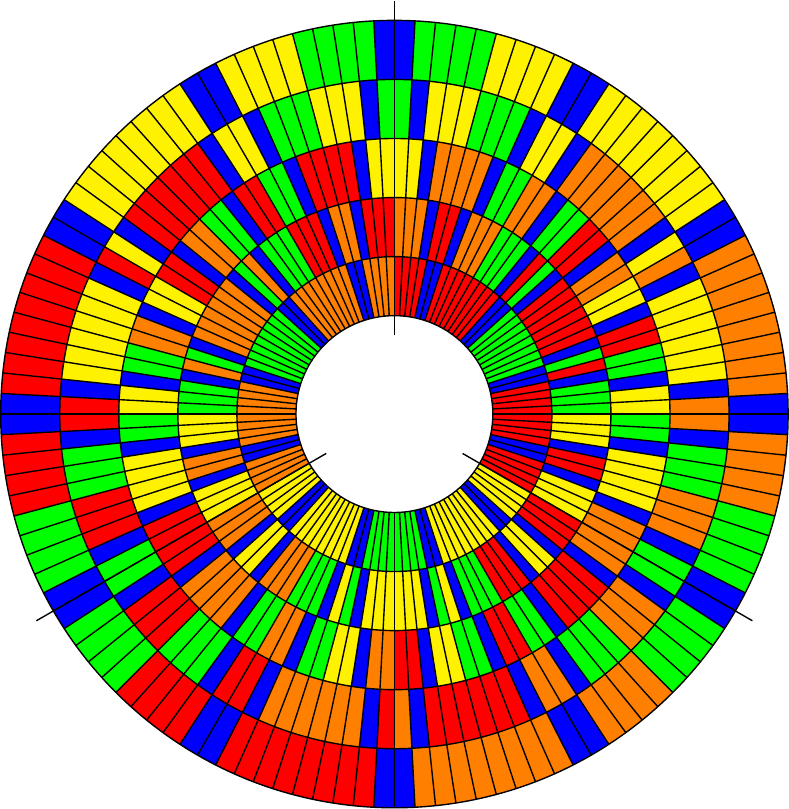} &
\includegraphics[scale=0.9]{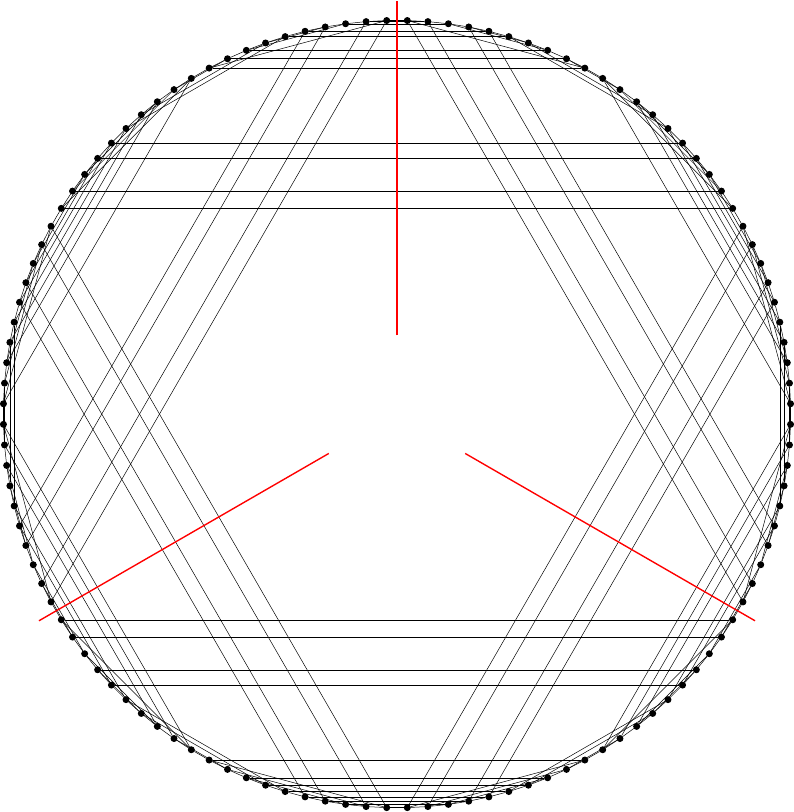} \\
\raisebox{37mm}{(b)} &
\includegraphics[scale=0.9]{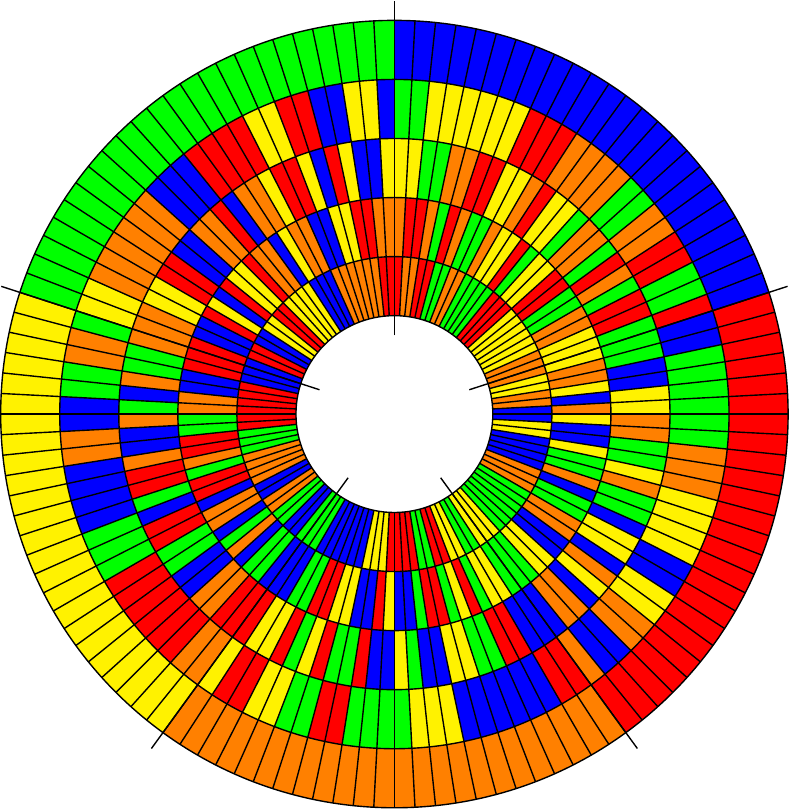} &
\includegraphics[scale=0.9]{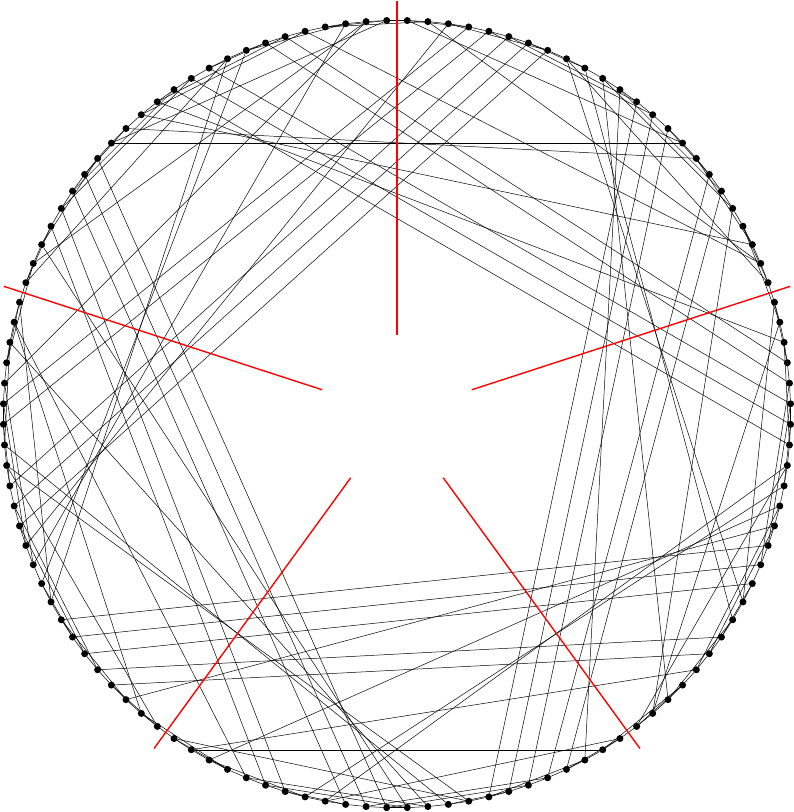} \\
\raisebox{37mm}{(c)} &
\includegraphics[scale=0.9]{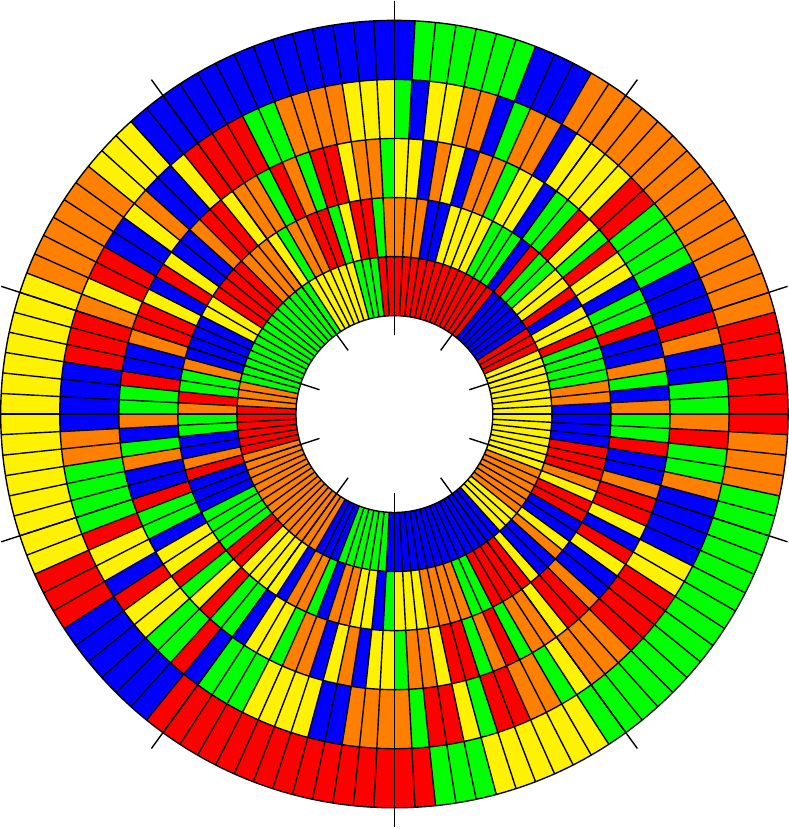} &
\includegraphics[scale=0.9]{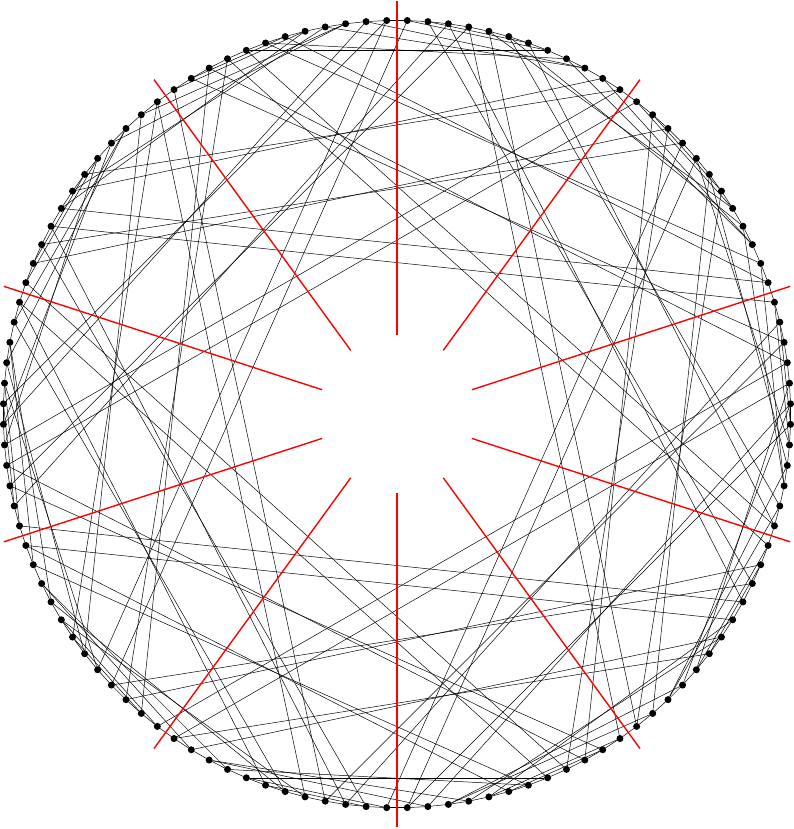} \\
\end{tabular}
}
\caption{Symmetric Hamilton cycles in~$\Pi_5$ (1=red, 2=orange, 3=yellow, 4=green, 5=blue):
(a)~Steinhaus-Johnson-Trotter cycle~$\Lambda_5$ with compression~3;
(b)~Cycle with compression~5 from Theorem~\ref{thm:Pin} for $f=(\ide,23451)$;
(c)~Cycle with compression~10 for $f=(\rev,23451)$.}
\label{fig:p5}
\end{figure}

\begin{proof}
Note that $k=\ord(f_\ba)$.
From Lemma~\ref{lem:id-mix} we obtain a path $Q=(x_1,\ldots,x_\ell)$ in~$\Pi_n$ that starts at $x_1=\ide$, ends at a neighbor~$x_\ell$ of~$f_\ba(\ide)$ and that visits each orbit of~$f_\ba$ from~\eqref{eq:fa-orb1} or~\eqref{eq:fa-orb2} with $u=\ide(\ba)$ exactly once.
Note that $\ell$ is even as some $f_1$ has an even number of orbits in part~(i) and $f_1f_2$ has an even number of orbits in part~(ii).
For $i=1,\ldots,\ell$, let $y_i$ be the $m$-tuple of permutations of the sets~$A_j$ whose concatenation gives~$x_i$, i.e., $x_i=u\otimes y_i$.

In part~(i) and $m=1$, we have $k=n$ and we are done as $Q$ already visits all orbits, so
\begin{equation*}
Q,f_\ba(Q),\ldots,f_\ba^{k-1}(Q)
\end{equation*}
is a $k$-symmetric Hamilton cycle in~$\Pi_n$.

In part~(i) and $m=2$, we apply Theorem~\ref{thm:comb-adj} to obtain a Hamilton path~$R$ in~$G(\ba)$ from the identity $\ba$-permutation $u:=\ide(\ba)=1^{a_1}2^{a_2}$ to $v:=2^{a_2}1^{a_1}$.

For any two consecutive vertices~$x_i=u \otimes y_i$ and $x_{i+1}=u \otimes y_{i+1}$ on the path~$Q$, we have that $v \otimes y_i$ and $v \otimes y_{i+1}$ are also adjacent.
It follows that
\begin{equation*}
P:=\big(R \otimes y_1, \lvec R \otimes y_2, \ldots, R \otimes y_{\ell-1}, \lvec R \otimes y_\ell\big)
\end{equation*}
is a path in~$\Pi_n$ from $x_1=u \otimes y_1=\ide$ to $x_\ell=u \otimes y_\ell$, which is a neighbor of~$f_\ba(\ide)$.
Moreover, $P$ visits all orbits of~$f_\ba$ exactly once.
Therefore,
\begin{equation*}
P,f_\ba(P),\ldots,f_\ba^{k-1}(P)
\end{equation*}
is a $k$-symmetric Hamilton cycle in~$\Pi_n$.

In part~(i) and $m\ge 3$ and in part~(ii) with $c\ge 2$ we apply Theorem~\ref{thm:multi-adj} to obtain a Hamilton~$C$ in~$G(\ba)$ (in part~(i), all $a_i$ are odd, which rules out the exceptional case $(n-2,1,1)$ with even~$n$, and in part~(ii), we assumed $m\ge 4$, which guarantees at least two odd~$a_i$).
We consider the $\ba$-permutation~$u:=\ide(\ba)$, and let $v,v'\in \binom{[n]}{\ba}$ be the neighbors of $u$ on~$C$.
Furthermore, let~$R,R'$ denote the Hamilton paths of~$G(\ba)$ from~$u$ to~$v$ or~$v'$, respectively, obtained from~$C$ by removing the corresponding edge.
Clearly, $v$ and~$v'$ differ from~$u$ by adjacent transpositions of a different pair of values, and as all $a_i$, $i=2,\ldots,m$ are either~1 or $\ge 3$, we have that for any pair of positions with same values in~$u$, at least one of the transpositions to reach~$v$ or~$v'$ from~$u$ does not involve any of these two positions.
It follows that for any two consecutive vertices~$x_i=u \otimes y_i$ and $x_{i+1}= u \otimes y_{i+1}$ on the path~$Q$, we have that $v \otimes y_i$ and $v \otimes y_{i+1}$ are adjacent, or $v' \otimes y_i$ and $v' \otimes y_{i+1}$ are adjacent.
In the former case, we define $R_i:=R$, whereas in the latter case we define $R_i:=R'$.
With these definitions
\begin{equation*}
P:=\big(R_1 \otimes y_1, \lvec R_1 \otimes y_2, \ldots, R_{\ell-1} \otimes y_{\ell-1}, \lvec R_{\ell-1} \otimes y_\ell\big)
\end{equation*}
is a path in~$\Pi_n$ from~$x_1=u\otimes y_1=\ide$ to~$x_\ell=u \otimes y_\ell$, which is a neighbor of $f_\ba(\ide)$.
Moreover, $P$ visits all orbits of~$f_\ba$ exactly once.
Therefore,
\begin{equation*}
P,f_\ba(P),\ldots,f_\ba^{k-1}(P)
\end{equation*}
is a $k$-symmetric Hamilton cycle in~$\Pi_n$.

The only case that has to be treated separately is part~(ii) with $c=1$, i.e., $a_1=a_2=2$.
In this case we use $u:=1^2 3^{a_3}4^{a_4}\cdots m^{a_m} 2^2$ in the argument above, and we obtain a path~$P$ in~$\Pi_n$ from~$u\otimes y_1$ to~$u\otimes y_\ell$, which is a neighbor of $f_\ba(u\otimes y_1)$, yielding a $k$-symmetric Hamilton cycle in~$\Pi_n$.
This completes the proof.
\end{proof}

Combining our previous results, we obtain the following near-optimal bounds for the Hamilton compression of~$\Pi_n$, which are tight in infinitely many cases.

\begin{theorem}
\label{thm:kappa-Pin}
The Hamilton compression of the permutahedron~$\Pi_n$ has the following properties:
\begin{enumerate}[label=(\roman*),leftmargin=8mm, topsep=0mm, noitemsep]
\item We have $\kappa(\Pi_3)=6$, $\kappa(\Pi_4)=6$ and $\kappa(\Pi_5)=10$.
\item For $n\ge 3$ we have $\kappa(\Pi_n)\ge \max\{\lambda_0(n),\lambda_2(n)\}\ge \lambda(n)/2$.
\item For $n\ge 4$, if $n\equiv 0,1 \bmod 4$, then we have $\kappa(\Pi_n)\le \max\{2\lambda_0(n),\lambda_2(n)\}$, and if $n\equiv 2,3 \bmod 4$, then we have $\kappa(\Pi_n)\le \lambda(n)$.
\item Let $n\ge 4$.
If $n\equiv 0,1 \bmod 4$ and $2\lambda_0(n)\le \lambda_2(n)$, then we have $\kappa(\Pi_n)=\lambda_2(n)$, and the second condition holds for all~$n\geq 739$.
If $n\equiv 2,3 \bmod 4$ and $\lambda(n)\in\{\lambda_0(n),\lambda_2(n)\}$, then we have $\kappa(\Pi_n)=\lambda(n)$, and the condition $\lambda(n)=\lambda_2(n)$ holds for arbitrarily large intervals.
\item The lower and upper bounds in~(ii) and~(iii) differ at most by a factor of~2.
In particular, we have $\kappa(\Pi_n)=\Theta(\lambda(n))=e^{(1+o(1))\sqrt{n \ln n}}$.
\end{enumerate}
\end{theorem}

From part~(iv) of the theorem we obtain exact results for the following values~$n\le 100$:
We have $\kappa(\Pi_n)=\lambda_0(n)$ for $n=15$ and $\kappa(\Pi_n)=\lambda_2(n)$ for $n=22,46,49,51,52,53,55,68,69,72,73,74,75,\allowbreak 80,82,85,87,88,89,91,92,93,96,97,99,100$.
See the appendix for more exact results and the corresponding values of~$\kappa(\Pi_n)$.
A 10-symmetric Hamilton cycle in~$\Pi_5$ obtained by computer search is shown in Figure~\ref{fig:p5}~(c).

\begin{proof}
The exact values for $n=3,4,5$ stated in~(i) were obtained by computer search.

We now prove~(ii).
Let $n\ge 3$.
By Lemma~\ref{lem:landau-prop}~(ii) there is a partition~$\ba=(a_1,\ldots,a_m)$ of~$n$ into powers of distinct odd primes and~1s such that~$\lcm(\ba)=\lambda_0(n)$, and we clearly have $a_1\ge 3$.
In particular, $a_1,\ldots,a_m$ are pairwise coprime and odd, so by Theorem~\ref{thm:Pin}~(i) there is a $k$-symmetric Hamilton cycle in~$\Pi_n$ where $k:=\lcm(\ba)=\lambda_0(n)$.
This shows that $\kappa(\Pi_n)\ge \lambda_0(n)$ for $n\ge 3$.

Let $n\ge 12$.
By Lemma~\ref{lem:landau-prop}~(iii) there is a partition~$\ba=(a_1,\ldots,a_m)$ of~$n$ into powers of distinct odd primes, $2^c$ for some $c\ge 1$, 2 and 1s such that $\lcm(\ba)=\lambda_2(n)$, and moreover $m\ge 4$.
In particular, all the $a_i$ except~2 and $2^c$ are pairwise coprime, so by Theorem~\ref{thm:Pin}~(ii) there is a $k$-symmetric Hamilton cycle in~$\Pi_n$ where $k:=\lcm(\ba)=\lambda_2(n)$.
This shows that $\kappa(\Pi_n)\ge \lambda_2(n)$ for $n\ge 12$.

For $n=3,\ldots,11$ we can verify from Table~\ref{tab:landau} that $\lambda_0(n)\ge \lambda_2(n)$.

Combining the three observations from before shows that $\kappa(\Pi_n)\ge \max\{\lambda_0(n),\lambda_2(n)\}$ for $n\ge 3$, and the lower bound $\lambda(n)/2$ for the maximum follows from Lemma~\ref{lem:landau-prop}~(iv).

The upper bounds stated in~(iii) are from Lemma~\ref{lem:perm-ub}.

The statements in~(iv) are an immediate consequence of~(ii) and~(iii), also using Lemma~\ref{lem:landau-prop}~(v)+(vi).

The statements in~(v) follow by observing that $\max\{a,b\}$ and $\max\{2a,b\}$ differ by at most a factor of~2, and by using~\eqref{eq:landau-asymp}.
\end{proof}

\subsection{Application to balanced 1-track Gray codes}

We write $\Pi_n^+$ for the graph obtained from~$\Pi_n$ by adding edges that correspond to transpositions of the first and last entry of a permutation, i.e., we allow cyclically adjacent transpositions.

\begin{theorem}
\label{thm:1track}
For every odd~$n\geq 3$ there is an $n$-symmetric Hamilton cycle in~$\Pi_n^+$ that has 1~track and is balanced, i.e., each of the $n$ transpositions is used equally often ($(n-1)!$ many times).
\end{theorem}

\begin{figure}[b!]
\makebox[0cm]{ 
\begin{tabular}{cc}
\includegraphics{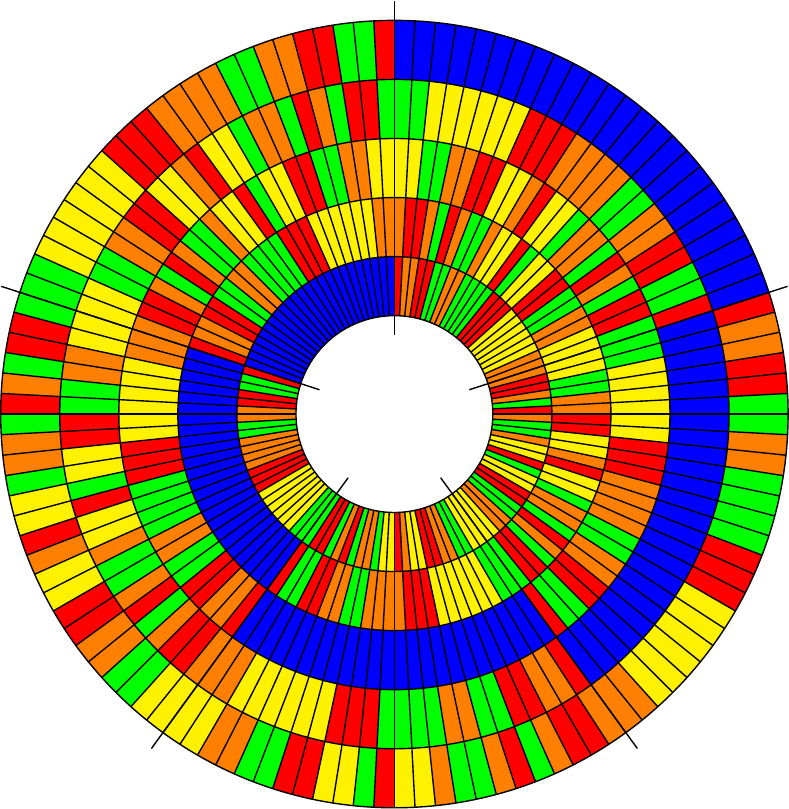} &
\includegraphics{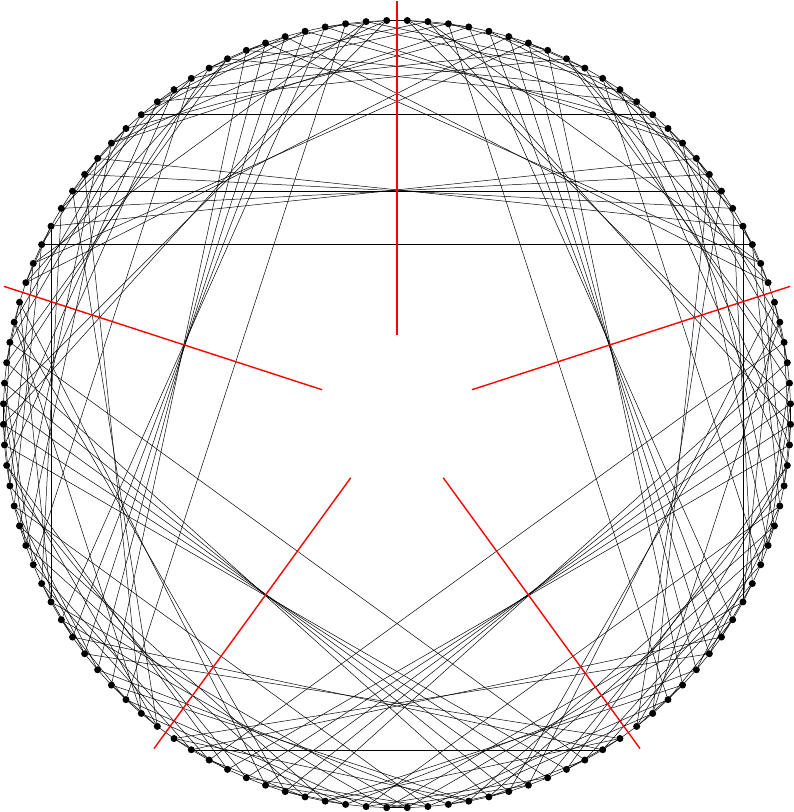} \\
\end{tabular}
}
\caption{Balanced 1-track Hamilton cycle in~$\Pi_5^+$ with compression~5 from Theorem~\ref{thm:1track} (cyclically adjacent transpositions).}
\label{fig:p5p}
\end{figure}

\begin{proof}
We consider the automorphism~$f$ of~$\Pi_n^+$ that cyclically shifts all entries one position to the left, i.e., $f(x_1\cdots x_n)=x_2\cdots x_nx_1$.
We choose permutations that have the symbol~$n$ fixed at the last position as representatives.
By Theorem~\ref{thm:lace} there is a Hamilton path~$Q$ in~$\Pi_{n-1}$ from $x:=1\cdots (n-1)$ to~$y:=2\cdots (n-1)1$ (note that $y$ is odd as $n$ is odd).
Thus $\Pi_n^+$ contains the path $P:=Qn$ from $xn=\ide$ to $yn=2\cdots(n-1)1n$, which is adjacent to $f(\ide)=2\cdots (n-1)n1$ by a transposition of the last two entries.
Consequently, $C:=P,f(P),f^2(P),\ldots,f^{n-1}(P)$ is an $n$-symmetric Hamilton cycle of~$\Pi_n^+$.
Furthermore, in the matrix corresponding to~$C$, any two columns are cyclic shifts of each other, so $C$ has the 1-track property.
Lastly, note that if~$P$ applies a transposition~$(i,i+1)$, then this becomes a transposition~$(i-j,i+1-j)$ (modulo~$n$) in~$f^j(P)$, which implies the balancedness property.
\end{proof}

\section{Abelian Cayley graphs}
\label{sec:cayley}

In this section we consider the Hamilton compression of Cayley graphs of abelian groups, introduced in Section~\ref{sec:prelim-cayley} (recall also Section~\ref{sec:prelim-graphs}).
After deriving a stronger version of the well-known factor group lemma, we construct an infinite family of abelian Cayley graphs that have Hamilton compression~1.
We complement this by showing that all other abelian Cayley graphs have Hamilton compression at least~2.

\subsection{The factor group lemma}

An important tool for proving Hamiltonicity in Cayley graphs is the so-called \emph{factor group lemma}.

\begin{lemma}[Factor group lemma, {\cite[Sec.~2.2]{MR762322}}]
\label{lem:factor}
Let $G$ be a group, $S\seq G$ a generating set, and $N$ a normal subgroup of $G$.
If $s_1,\ldots, s_{|G/N|}$ is a sequence of elements in~$S$ such that
\begin{enumerate}[label=(\roman*),leftmargin=8mm, topsep=0mm, noitemsep]
\item $Ns_1,Ns_1s_2,\ldots,Ns_1\cdots s_{|G/N|}$ is a Hamilton cycle of $\Gamma(G/N,\{Ns\mid s \in S\})$,
\item $N = \langle s_1\cdots s_{|G/N|} \rangle$,
\end{enumerate}
then the Cayley graph $\Gamma(G,S)$ has a Hamilton cycle.
\end{lemma}

It turns out that while proving the factor group lemma one can also guarantee some symmetry in the Hamilton cycle obtained.
Thus, we state and prove the following compression version of the factor group lemma.

\begin{lemma}
\label{lem:factor-comp}
Let $G$ be a group, $S\seq G$ a generating set, $g \in G$ and~$s\in S$.
If there exists a path in the Cayley graph $\Gamma=\Gamma(G,S)$ from~$e$ to~$g s$ that intersects every right coset of $\langle g \rangle$ exactly once, then $\kappa(\Gamma)\ge \ord(g)$.
\end{lemma}

Note that Lemma~\ref{lem:factor-comp} implies Lemma~\ref{lem:factor} by setting $g:=s_1\cdots s_{|G/N|}$, yielding a lower bound for the Hamilton compression of~$\kappa(\Gamma(G,S))\geq |N|$.

\begin{proof}
We begin by noting that $f:G \to G$ defined by $f(x)=g x$ is an automorphism of~$\Gamma$.
We consider the subgroup~$\langle g \rangle$ generated by~$g$, and we define $k:=\ord(g)$.
By the orbit-stabilizer theorem, the number of right cosets of $\langle g\rangle$ is $\ell:=|G|/k$.
Let $P=(x_1,\ldots,x_\ell)$ be a path in~$\Gamma$ from~$e$ to~$g s$ that intersects every right coset of $\langle g\rangle$ exactly once.
We claim that the paths $P,g P,\ldots,g^{k-1}P$ are pairwise vertex-disjoint.
Suppose for the sake of contradiction that there exists $i,j \in \{0,\ldots,k-1\}$, $i\neq j$, such that $g^i P \cap g^j P\neq \emptyset$.
Then there are $a,b\in [\ell]$ such that $g^i x_a=g^j x_b$, which means that $x_a$ and $x_b$ lie in the same coset of~$\langle g\rangle$.
As both are on~$P$, we obtain that $a=b$, and therefore $i=j$, a contradiction, so the claim is proved.
Consequently, we have $\big|\bigcup_{i=0}^{k-1} g^i P \big|=\ell k=|G|$.
Furthermore, since $\{g^is,g^i\}$ is an edge of~$\Gamma$ for all $i \in \{0,\ldots,k-1\}$, we obtain that
\begin{equation*}
C:=P,f(P),\ldots, f^{k-1}(P)=P,g P,\ldots,g^{k-1}P
\end{equation*}
is a $k$-symmetric Hamilton cycle in~$\Gamma$; see Figure~\ref{fig:factor-group}.
It follows that $\kappa(\Gamma)\ge k=\ord(g)$.
\end{proof}

\begin{figure}
\includegraphics[page=1]{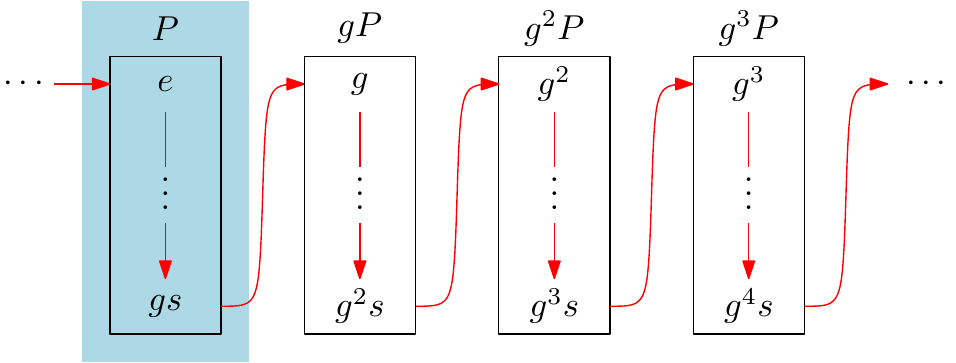}
\caption{Illustration of the proof of Lemma~\ref{lem:factor-comp}.}
\label{fig:factor-group}
\end{figure}

We can also phrase Lemma~\ref{lem:factor-comp} in the language of graph covers.
Let $G$ be a group, $S\seq G$ a generating set, and $N$ a normal subgroup of~$G$.
We define the quotient Cayley graph $\Gamma^N := \Gamma (G/N, \{Ns \mid s\in S\})$.
The \emph{voltage} of a cycle $C = Ns_1,\dots, Ns_1\cdots s_k$ in $\Gamma^N$ is given by the product $s_1\cdots s_k$.
A direct application of Lemma~\ref{lem:factor-comp} gives the following lemma.

\begin{lemma}\label{lem:voltage}
Let $G$ be a group, $S \seq G$ a generating set, $\Gamma = \Gamma(G,S)$ its Cayley graph, and $N$ a normal subgroup of~$G$.
If there is a Hamilton cycle in~$\Gamma^N$ with voltage~$v$ such that $\langle v \rangle = N$, then $\kappa(\Gamma) \geq |N|$.
\end{lemma}

\begin{remark}
\label{rem:inclusion}
For a group~$G$ and a generating set~$S$, automorphisms of~$\Gamma(G,S)$ in the \emph{left regular representation} of the group are the automorphisms of the form $x \mapsto gx$ for some $g\in G$.
In particular, note that the automorphisms in Lemma~\ref{lem:voltage} are in the left regular representation.
Such mappings are automorphisms of~$\Gamma(G,S)$ \emph{independently} of the generating set~$S$.
The same observation holds for mappings~$f:G \to G$ that satisfy
\begin{equation}
\label{eq:fyfx}
f(y)^{-1}f(x)=x^{-1}y
\end{equation}
for all~$x,y \in G$, as they are also automorphisms of~$\Gamma(G,S)$ independently of~$S$.
All automorphisms considered in this section are either in the left regular representation or they satisfy~\eqref{eq:fyfx}, so we may assume without loss of generality that the generating set~$S$ is inclusion-minimal.
\end{remark}

\subsection{Odd order}

In this section we consider Cayley graphs of abelian groups~$G$ for which the order $n:=|G|$ is odd.
We will distinguish two regimes: when $n$ is square-free and when $n$ has a square divisor.
Note that when $n$ is square-free, by Theorem~\ref{thm:abelian-struct}, $G$ is a direct sum of cyclic groups of prime order, i.e., $G=\bigoplus_{i=1}^\ell \mathbb{Z}_{p_i}$ for distinct odd primes $p_1,\ldots, p_\ell$.
In this case, we say that a generating set~$S\seq G$ is \emph{canonical} if $S = \{s^1,\dots, s^\ell\}$ and for every $i,j \in [\ell]$ we have that $s_i^j = 0$ if and only if $i\neq j$.
The main result of this section is that, in the odd order case, the Cayley graphs are incompressible if and only if $n$ is composite and square-free, and $S$ is a canonical generating set.

\begin{theorem}
\label{thm:odd-order}
Let $G$ be an abelian group of odd order $n$, $S\seq G$ a generating set and $\Gamma = \Gamma(G,S)$ its Cayley graph.
\begin{enumerate}[label=(\roman*),leftmargin=8mm, topsep=0mm, noitemsep]
\item If $n$ is composite and square-free, and $S$ is a canonical generating set, then $\kappa(\Gamma) = 1$.
\item Otherwise, there is a prime~$p$ in the decomposition of~$n$ such that $\kappa(\Gamma)\geq p$.
\end{enumerate}
\end{theorem}

A particularly useful lemma in the odd order case is the following.

\begin{lemma}[{\cite[consequence of Theorem 2.1]{MR1057481}}]
\label{lem:cycle_decomp}
Let $G$ be an abelian group of odd order, $S\seq G$ a generating set, and $\Gamma = \Gamma(G,S)$ its Cayley graph.
If $C$ is a cycle in~$\Gamma$, then there are Hamilton cycles $H_1,\dots, H_r$ in~$\Gamma$ such that $C = H_1 \symdif \cdots \symdif H_r$, i.e., $C$ is the symmetric difference of $H_1,\ldots,H_r$.
\end{lemma}

We are now ready to prove Theorem~\ref{thm:odd-order}.

\begin{proof}[Proof of Theorem~\ref{thm:odd-order}]
We begin by proving~(i).
Let $n = p_1 \cdots p_\ell$ with $p_i$ prime for $i \in [\ell]$.
By Theorem~\ref{thm:abelian-struct}, we have that $G = \bigoplus_{i=1}^\ell \ZZ_{p_i}$.
Since $S$ is a canonical generating set, we have that $S=\{s_1,\dots, s_\ell\}$ with $\ord(s_i)=p_i$ for every $i\in [\ell]$.
As a consequence, $\Gamma = C_{p_1} \boxprod \dots \boxprod C_{p_\ell}$; i.e., $\Gamma$ is a Cartesian product of \emph{prime} graphs, i.e., graphs that are not products of two non-trivial graphs.
Thus, by~{\cite[Theorem 6.13]{MR2817074}} the automorphism group of~$\Gamma$ is the product of the automorphism groups of each prime graph, i.e., $\Aut(\Gamma)$ is the direct product of $D_{2p_i}$ the dihedral groups of order $2p_i$ for $i \in [\ell]$.
Suppose there is a Hamilton cycle~$C$ with $\kappa(C)=m$ for some divisor~$m$ of~$n$.
Then, $C$ is $p$-symmetric for every prime~$p$ dividing~$m$; in particular, there is an $i\in[\ell]$ such that $C$ is $p_i$-symmetric, and there is an automorphism $f \in \Aut(\Gamma)$ of order~$p_i$ such that $C$ is $f$-invariant.
Since $\langle s_i \rangle$ is a normal subgroup of $D_{2p_i}$, it is also a normal subgroup of $\Aut(\Gamma)$, which implies that $f(x) = ks_i+x$ for some $k \in [p_i-1]$.
Hence, there is a path $P = x_1, \dots, x_{n/p_i}$ hitting every coset of $\langle s_i \rangle$ exactly once, and
\[C = P,f(P), \dots, f^{p_i-1}(P).\]
Thus, there is an edge labelled~$s_i$ or $-s_i$ between~$x_{n/p_{i}}$ and~$ks_i+x_{1}$; i.e., $x_{n/p_{i}}-(ks_i+x_{1})$ equals either $s_i$ or $-s_i$.
However, this would imply that $x_{n/p_{i}}+ \langle s_i \rangle= x_1+\langle s_i\rangle$, a contradiction.

We now prove part~(ii) of the theorem.
We consider three subcases.
\begin{itemize}
\item \textbf{Case 1:} $n$ is prime.
In this case $\Gamma = C_p$ and therefore $\kappa (\Gamma) = p$.

\item \textbf{Case 2:} $n$ is composite and square-free, $S$ is not canonical.
Let $n = p_1\dots p_\ell$ with $p_i$ prime for all~$i \in [\ell]$.
Since $S$ is not canonical, there exist an $s \in S$ of composite order $p_im$ for some $i \in [\ell]$ and $m \in \NN$.
Let $h = ms$, and consider the graph $\Gamma^{\langle h \rangle}$.
By Lemma~\ref{lem:connectivity}, $\Gamma$ is 2-connected, which implies that there is a cycle containing two vertices that lie in the same coset of $\langle h\rangle$, and consequently we obtain a cycle $C$ in $\Gamma^{\langle h \rangle}$ with nonzero voltage.
Additionally, since $G/\langle h\rangle$ is an abelian group of odd order, by Lemma~\ref{lem:cycle_decomp} we have that there are $r$ Hamilton cycles $H_1, \dots, H_r$ in $\Gamma^{\langle  h \rangle}$ such that $C = H_1 \symdif \cdots \symdif H_r$.
Furthermore, since $C$ has nonzero voltage, at least one of $H_1,\ldots,H_r$ must have a nonzero voltage.
By Lemma~\ref{lem:voltage} we conclude that $\kappa(\Gamma)\geq |\langle h \rangle |= p_i$.

\item \textbf{Case 3:} $n$ is divisible by $p^2$ for some prime $p$.
We may assume without loss of generality that $S$ is inclusion-minimal (recall Remark~\ref{rem:inclusion}).
We claim that there exists an element of~$G\setminus S$ of order~$p$.
Note that a Sylow $p$-subgroup $P$ of~$G$ contains either~$\ZZ_{p^2}$ or~$\ZZ_{p}^2$ as a subgroup~$H$.
In the first case, we take a generator~$a$ of $H$, and we let $h:=pa$.
By the minimality of~$S$, we have that $h \notin S$.
In the second case, we note that $H=\langle a \rangle \oplus \langle b \rangle$ for some $a,b \in G$.
By the minimality of~$S$, we have that $\{a,b,a+b\}\not\seq S$ and let $h \in  \{a,b,a+b\}\setminus S$.
Thus, in both cases, we have that $G/\langle h \rangle$ is an abelian group of odd order, and we can proceed as in case~2 before.\qedhere
\end{itemize}
\end{proof}

\subsection{Even order}

In the previous section, we saw that if $|G|$ is odd, then the Cayley graph $\Gamma(G,S)$ has trivial Hamilton compression if and only if $|G|$ is composite and square-free, and $S$ is a canonical generating set.
The main objective of this section is to show that if $|G|$ is even, then the Cayley graph for any generating set has non-trivial compression.

\begin{theorem}
\label{thm:comp2}
Let $G$ be an abelian group and~$S\seq G$ a generating set.
If $|G|\geq 4$ is even, then the Cayley graph~$\Gamma=\Gamma(G,S)$ has compression $\kappa(\Gamma)\ge 2$.
\end{theorem}

\begin{proof}
If $G = \langle s \rangle$ for some $s\in S$, then we have $\kappa (\Gamma) = \kappa(C_{|G|})= |G| \geq 4$ and the theorem holds.
Thus, from now on we assume that $G\neq \langle s \rangle$ for every $s\in S$.
We begin by noting that $S$ must have an element~$s\in S$ of even order.
Furthermore, we may assume without loss of generality that $S$ is inclusion-minimal (recall Remark~\ref{rem:inclusion}).
We consider two cases.
\begin{itemize}
\item \textbf{Case 1:} $\ord(s)=2$.
Let $S' := S \setminus \{s,-s\}$.
In this case $G \cong \langle s \rangle \oplus \langle S' \rangle \cong \ZZ_2 \oplus \langle S' \rangle $, and in the rest of the proof we will write elements of~$G$ as pairs w.r.t.\ this direct sum representation.
Theorem~\ref{thm:abelian} yields a Hamilton path~$P$ in~$\Gamma(\langle S' \rangle, S')$ from~$0$ to $v$ for some $v\in\langle S'\rangle$.
It follows that $\hat P:=(0,P)$ is a path in~$\Gamma$ from $(0,0)$ to $(0,v)$.
Let us define a mapping $f:G \to G$ by $f(x,y)=(s+x,v-y)$, where $x\in \langle s\rangle$ and $y\in\langle S'\rangle$.
It is easy to check that $f$ is an automorphism of~$\Gamma$.
Furthermore, for every $(x,y),(x',y') \in G$ we have
\begin{align*}
  -f(x',y')+f(x,y)=(x-x',y'-y)=(x'-x,y'-y)=-(x,y)+(x',y')
\end{align*}
which implies that~\eqref{eq:fyfx} holds.

We claim that $C:=\hat P,f(\hat P)$ is a 2-symmetric Hamilton cycle.
Note that $\hat P$ and $f(\hat P)$ are disjoint paths, as $V(\hat P)=\{0\}\oplus \langle S'\rangle$ and $V(f(\hat P))=\{g\}\oplus \langle S'\rangle$.
The path~$\hat P$ starts at~$(0,0)$ and ends at~$(0,v)$, and the path~$f(\hat P)$ starts at~$(g,v)$ and ends~$(g,0)$, so $C$ is indeed a cycle in~$\Gamma$.
Its length is $2|\hat P|=2|\langle S'\rangle|=|G|$, so it is a Hamilton cycle.
The automorphism~$f$ shows that $C$ is 2-symmetric, proving that $\kappa(\Gamma)\ge 2$.

\item \textbf{Case 2:} $\ord(s)\geq 4$.
Let $h := s+s$ and consider the quotient graph $\Gamma^{\langle h\rangle}$.
Since $\Gamma^{\langle h \rangle}$ is the Cayley graph of an abelian group, by Theorem~\ref{thm:abelian}, it has a Hamilton cycle $C = x_1 + \langle h \rangle, x_2 + \langle h \rangle,\dots, x_{n/\ord(h)}+\langle h \rangle$.
We say that an edge is labelled $+s$ if it is of the form $(x,x+s)$ for some $x \in G$.
Similarly, we say that an edge is labelled $-s$ if it is of the form $(x,x-s)$ for some $x\in G$.
Let $K$ be the number of edges in~$C$ labelled $+s$ or $-s$.
Since $S$ is minimal, we have that $K>0$ and, since $C$ is a cycle, we have that $K$ is even.
Furthermore, since $s+\langle h\rangle =-s+ \langle h\rangle$, replacing an edge labelled $+s$ with an edge labelled $-s$ (and vice versa) gives another Hamilton cycle in $\Gamma^{\langle h \rangle}$.
In particular, for every $k,\ell \in \NN$ such that $k+\ell = K$, there is a Hamilton cycle $C^{k,\ell}$ of $\Gamma^{\langle h \rangle}$ that uses $k$ edges labelled $+s$ and $\ell$ edges labelled $-s$.
Since the voltage of $C^{k,\ell}$ is $(k-\ell)s$, there is a Hamilton cycle of $\Gamma^{\langle h \rangle}$ with voltage~$s+s$.
Thus, by Lemma~\ref{lem:voltage} we conclude that $\kappa(\Gamma) \geq |\langle h \rangle| = \frac{\ord(s)}{2}\geq 2$.\qedhere
\end{itemize}
\end{proof}

\section{Open questions}
\label{sec:open}

We conclude this paper with a number of interesting open questions.

\begin{enumerate}[label=(Q\arabic*), leftmargin=8mm, topsep=1mm, noitemsep]
\item
Can the Gray codes constructed in this paper be computed efficiently?
While our proofs translate straightforwardly into algorithms whose running time is polynomial in the size of the graph, a more ambitious goal would be algorithms whose running time per generated vertex is polynomial in the length of the vertex labels (bitstrings, combinations, permutations, etc.).
For the cycles in the hypercube with optimal Hamilton compression it should be possible to derive such an algorithm from our construction.
For Johnson graphs this should also be possible, as the path guaranteed by Theorem~\ref{thm:necklace} is efficiently computable.
For permutahedra, this task seems most complicated, as it would require efficiently computing the structures guaranteed by Theorems~\ref{thm:lace} and~\ref{thm:multi-adj}, for which no algorithms are known (unlike for Theorem~\ref{thm:comb-adj}).

\item
What is the Hamilton compression of the middle levels graph?
The best known bounds are $2n+1\le \kappa(M_{2n+1})\le 2(2n+1)$ (recall Theorem~\ref{thm:kappa-middle}).

\item
Odd graphs are another interesting class of vertex-transitive graphs with unknown Hamilton compression.
For any integer~$k\ge 1$, the odd graph~$O_k$ has as vertices all $(2k+1,k)$-combinations, and an edge between any two combinations that have no 1s in common.
Note that the odd graph~$O_k$ is the special Kneser graph~$K(2k+1,k)$.
Odd graphs~$O_k$, $k\ge 3$, were shown to have a Hamilton cycle in~\cite{MR4273468}, so $\kappa(O_k)\ge 1$.
Similarly to the middle levels graph, we can use cyclic shifts as the automorphism.
It is easy to see that $\kappa(O_k)\le 2k+1$, and since all necklaces have the same size~$2k+1$, there is hope to build a $(2k+1)$-symmetric Hamilton cycle.
We constructed such a solution for $k=4$, and we indeed conjecture that $\kappa(O_k)=2k+1$ for all~$k\ge 4$.

\item
What is the Hamilton compression of the associahedron, which has as automorphism group the dihedral group of a regular $n$-gon?
For $n=5,6,7,8$ we determined the values $5,2,7,2$ by computer, and we suspect that the primality of~$n$ plays a role.

\item
Instead of asking about the largest number $k=\kappa(G)$ such that $\Aut(G,C)$ (automorphisms of~$G$ that preserve~$C$) contains the cyclic subgroup of order~$k$ for some Hamilton cycle~$C$ in~$G$, we may ask for the dihedral subgroup of the largest order, which would allow not only for rotations of the drawings but also reflections.

\item
Is there a 1-track Hamilton cycle in~$\Pi_n$ (recall Theorem~\ref{thm:1track})?
In other words, can all $n!$ permutations be listed by adjacent transpositions so that every column is a cyclic shift of every other column?

\item
Is there a balanced Hamilton cycle in~$\Pi_n$?
In other words, can all $n!$ permutations be listed using each of the $n-1$ adjacent transpositions equally often?
Alternatively, can all $n!$ permutations be listed using each of the $\binom{n}{2}$ transpositions equally often (see~\cite{MR4046775})?
For $n=5$, we found orderings satisfying the constraints of both questions.

\item
Is there a balanced Gray code for listing permutations by star transpositions, i.e., transpositions~$(1,i)$ for $i=2,\ldots,n$?
One idea to build such a code is to use the automorphism~$f$ that cyclically left-shifts the last $n-1$ positions, leaving the first position unchanged, and to search any path~$P$ from~$\ide=123\cdots n$ to a neighbor of~$f(\ide)=134\cdots n2$ that visits every orbit exactly one.
By construction, the cycle $C:=P,f(P),\ldots,f^{n-2}(P)$ would be balanced.
We found such a solution for $n=6$ with computer help, and we believe it exists for all even~$n\ge 6$ (for odd~$n$ there are parity problems, and the Gray code has to be built differently).

\item
An automorphism of a graph in which all orbits have the same size is called \emph{semiregular}.
Maru\v{s}i\v{c}~\cite{MR621894} asked if every vertex-transitive digraph has a semiregular automorphism.
This question, independently raised by Jordan~\cite{MR979101}, is now known as `polycirculant conjecture'.
It was shown to hold in some special cases, but remains open in general.

\end{enumerate}

\torsten{Add those new sequences involving $\lambda(n)$ to the OEIS.}

\section*{Acknowledgements}

We thank Fedor Petrov for an idea of how to prove Lemma~\ref{lem:landau-prop}~(v).
We also thank Michal Kouck\'y for an idea that simplified the proof of Lemma~\ref{lem:cube-ub}.
Furthermore, we sincerely thank both anonymous reviewers for their helpful comments.
In particular, one reviewer suggested the current approach for Section~\ref{sec:cayley}, which solved one of the open problems in the conference version of this paper.

\bibliographystyle{alpha}
\bibliography{refs}

\begin{thebibliography}{FKMS20}

\bibitem[ABCC06]{tsp_book}
D.~L. Applegate, R.~E. Bixby, V.~Chv\'{a}tal, and W.~J. Cook.
\newblock {\em The Traveling Salesman Problem: A Computational Study}.
\newblock Princeton University Press, 2006.

\bibitem[Als89]{MR1020643}
B.~Alspach.
\newblock Lifting {H}amilton cycles of quotient graphs.
\newblock {\em Discrete Math.}, 78(1-2):25--36, 1989.

\bibitem[ALW90]{MR1057481}
B.~Alspach, Stephen~C. Locke, and D.~Witte.
\newblock The {H}amilton spaces of {C}ayley graphs on abelian groups.
\newblock {\em Discrete Math.}, 82(2):113--126, 1990.

\bibitem[BS96]{MR1410880}
G.~S. Bhat and C.~D. Savage.
\newblock Balanced {G}ray codes.
\newblock {\em Electron. J. Combin.}, 3(1):Paper 25, 11~pp., 1996.

\bibitem[BW84]{MR737262}
M.~Buck and D.~Wiedemann.
\newblock Gray codes with restricted density.
\newblock {\em Discrete Math.}, 48(2-3):163--171, 1984.

\bibitem[CG96]{MR1405010}
S.~J. Curran and J.~A. Gallian.
\newblock Hamiltonian cycles and paths in {C}ayley graphs and digraphs---a
  survey.
\newblock {\em Discrete Math.}, 156(1-3):1--18, 1996.

\bibitem[CQ81]{MR641233}
C.~C. Chen and N.~F. Quimpo.
\newblock On strongly {H}amiltonian abelian group graphs.
\newblock In {\em Combinatorial mathematics, {VIII} ({G}eelong, 1980)}, volume
  884 of {\em Lecture Notes in Math.}, pages 23--34. Springer, Berlin-New York,
  1981.

\bibitem[DF04]{MR2286236}
D.~S. Dummit and R.~M. Foote.
\newblock {\em Abstract algebra}.
\newblock John Wiley \& Sons, Inc., Hoboken, NJ, third edition, 2004.

\bibitem[DKM21]{MR4328721}
S.~Du, K.~Kutnar, and D.~Maru\v{s}i\v{c}.
\newblock Resolving the {H}amiltonian problem for vertex-transitive graphs of
  order a product of two primes.
\newblock {\em Combinatorica}, 41(4):507--543, 2021.

\bibitem[DNZ08]{MR2523311}
M.~Del\'{e}glise, J.-L. Nicolas, and P.~Zimmermann.
\newblock Landau's function for one million billions.
\newblock {\em J. Th\'{e}or. Nombres Bordeaux}, 20(3):625--671, 2008.

\bibitem[Dus18]{MR3745073}
P.~Dusart.
\newblock Explicit estimates of some functions over primes.
\newblock {\em Ramanujan J.}, 45(1):227--251, 2018.

\bibitem[EHR84]{MR821383}
P.~Eades, M.~Hickey, and R.~C. Read.
\newblock Some {H}amilton paths and a minimal change algorithm.
\newblock {\em J. Assoc. Comput. Mach.}, 31(1):19--29, 1984.

\bibitem[EP96]{MR1445874}
T.~Etzion and K.~G. Paterson.
\newblock Near optimal single-track {G}ray codes.
\newblock {\em IEEE Trans. Inform. Theory}, 42(3):779--789, 1996.

\bibitem[Fen06]{MR2185979}
Y.-Q. Feng.
\newblock Automorphism groups of {C}ayley graphs on symmetric groups with
  generating transposition sets.
\newblock {\em J. Combin. Theory Ser. B}, 96(1):67--72, 2006.

\bibitem[FKMS20]{MR4046775}
S.~Felsner, L.~Kleist, T.~M\"{u}tze, and L.~Sering.
\newblock Rainbow cycles in flip graphs.
\newblock {\em SIAM J. Discrete Math.}, 34(1):1--39, 2020.

\bibitem[Fru77]{MR463029}
R.~Frucht.
\newblock A canonical representation of trivalent {H}amiltonian graphs.
\newblock {\em J. Graph Theory}, 1(1):45--60, 1977.

\bibitem[Gan18]{MR3791054}
A.~Ganesan.
\newblock On the automorphism group of a {J}ohnson graph.
\newblock {\em Ars Combin.}, 136:391--396, 2018.

\bibitem[GJ79]{MR519066}
M.~R. Garey and D.~S. Johnson.
\newblock {\em Computers and intractability}.
\newblock A Series of Books in the Mathematical Sciences. W. H. Freeman and
  Co., San Francisco, Calif., 1979.
\newblock A guide to the theory of NP-completeness.

\bibitem[Gou91]{MR1106528}
R.~J. Gould.
\newblock Updating the {H}amiltonian problem---a survey.
\newblock {\em J. Graph Theory}, 15(2):121--157, 1991.

\bibitem[Gou03]{MR1974368}
R.~J. Gould.
\newblock Advances on the {H}amiltonian problem---a survey.
\newblock {\em Graphs Combin.}, 19(1):7--52, 2003.

\bibitem[Gou14]{MR3143857}
R.~J. Gould.
\newblock Recent advances on the {H}amiltonian problem: {S}urvey {III}.
\newblock {\em Graphs Combin.}, 30(1):1--46, 2014.

\bibitem[Gra53]{gray:patent}
F.~Gray.
\newblock Pulse code communication, 1953.
\newblock March 17, 1953 (filed Nov. 1947). U.S. Patent 2,632,058.

\bibitem[HIK11]{MR2817074}
R.~Hammack, W.~Imrich, and S.~Klav{\v z}ar.
\newblock {\em Handbook of product graphs}.
\newblock Discrete Mathematics and its Applications (Boca Raton). CRC Press,
  Boca Raton, FL, second edition, 2011.
\newblock With a foreword by Peter Winkler.

\bibitem[HPB96]{DBLP:journals/tit/HiltgenPB96}
A.~P. Hiltgen, K.~G. Paterson, and M.~Brandestini.
\newblock Single-track {G}ray codes.
\newblock {\em {IEEE} Trans. Inf. Theory}, 42(5):1555--1561, 1996.

\bibitem[Jon05]{MR2116180}
G.~A. Jones.
\newblock Automorphisms and regular embeddings of merged {J}ohnson graphs.
\newblock {\em European J. Combin.}, 26(3-4):417--435, 2005.

\bibitem[Jor88]{MR979101}
D.~Jordan.
\newblock Eine {S}ymmetrieeigenschaft von {G}raphen.
\newblock In {\em Graphentheorie und ihre {A}nwendungen ({S}tadt {W}ehlen,
  1988)}, volume~9 of {\em Dresdner Reihe Forsch.}, pages 17--20. P\"{a}d.
  Hochsch. Dresden, Dresden, 1988.

\bibitem[KM09]{MR2548567}
K.~Kutnar and D.~Maru{\v s}i{\v c}.
\newblock Hamilton cycles and paths in vertex-transitive graphs---current
  directions.
\newblock {\em Discrete Math.}, 309(17):5491--5500, 2009.

\bibitem[KMR23]{kutnar-marusic-raza:23}
K.~Kutnar, D.~Maru{\v{s}}i{\v{c}}, and A.~S. Razafimahatratra.
\newblock Infinite families of vertex-transitive graphs with prescribed
  {H}amilton compression.
\newblock {\it arXiv:2305.09465}, 2023.

\bibitem[Knu11]{MR3444818}
D.~E. Knuth.
\newblock {\em The Art of Computer Programming. {V}ol. 4{A}. {C}ombinatorial
  Algorithms. {P}art 1}.
\newblock Addison-Wesley, Upper Saddle River, NJ, 2011.

\bibitem[KO12]{MR2889513}
D.~K\"{u}hn and D.~Osthus.
\newblock A survey on {H}amilton cycles in directed graphs.
\newblock {\em European J. Combin.}, 33(5):750--766, 2012.

\bibitem[Lan03]{landau_1903}
E.~Landau.
\newblock {\"U}ber die {M}aximalordnung der {P}ermutationen gegebenen {G}rades.
\newblock In {\em Arch. Math. Phys. Ser. 3}, volume~5. Teubner, 1903.

\bibitem[Lov70]{MR0263646}
L.~Lov{\'a}sz.
\newblock Problem 11.
\newblock In {\em Combinatorial Structures and Their Applications (Proc.
  Calgary Internat. Conf., Calgary, AB, 1969)}. Gordon and Breach, New York,
  1970.

\bibitem[Mar81]{MR621894}
D.~Maru\v{s}i\v{c}.
\newblock On vertex symmetric digraphs.
\newblock {\em Discrete Math.}, 36(1):69--81, 1981.

\bibitem[MMM21]{MR4262479}
A.~Merino, O.~Mi\v{c}ka, and T.~M\"{u}tze.
\newblock On a combinatorial generation problem of {K}nuth.
\newblock In {\em Proceedings of the 2021 {ACM}-{SIAM} {S}ymposium on
  {D}iscrete {A}lgorithms ({SODA})}, pages 735--743. [Society for Industrial
  and Applied Mathematics (SIAM)], Philadelphia, PA, 2021.

\bibitem[MNW21]{MR4273468}
T.~M\"{u}tze, J.~Nummenpalo, and B.~Walczak.
\newblock Sparse {K}neser graphs are {H}amiltonian.
\newblock {\em J. Lond. Math. Soc. (2)}, 103(4):1253--1275, 2021.

\bibitem[M{\"u}t22]{muetze:22}
T.~M{\"u}tze.
\newblock Combinatorial {G}ray codes---an updated survey.
\newblock {\it arXiv:2202.01280}, 2022.

\bibitem[Nic69a]{MR0253514}
J.-L. Nicolas.
\newblock Calcul de l'ordre maximum d'un \'{e}l\'{e}ment du groupe
  sym\'{e}trique {$S_{n}$}.
\newblock {\em Rev. Francaise Informat. Recherche Op\'{e}rationnelle},
  3(S\'{e}r. R-2):43--50, 1969.

\bibitem[Nic69b]{MR254130}
J.-L. Nicolas.
\newblock Ordre maximal d'un \'{e}l\'{e}ment du groupe {$S_{n}$} des
  permutations et ``highly composite numbers''.
\newblock {\em Bull. Soc. Math. France}, 97:129--191, 1969.

\bibitem[{oei}22]{oeis}
{OEIS} {F}oundation {I}nc. {T}he {O}n-{L}ine {E}ncyclopedia of {I}nteger
  {S}equences, 2022.
\newblock \url{http://oeis.org}.

\bibitem[PR09]{MR2548568}
I.~Pak and R.~Radoi{\v c}i{\'c}.
\newblock Hamiltonian paths in {C}ayley graphs.
\newblock {\em Discrete Math.}, 309(17):5501--5508, 2009.

\bibitem[RD11]{MR2801228}
M.~Ramras and E.~Donovan.
\newblock The automorphism group of a {J}ohnson graph.
\newblock {\em SIAM J. Discrete Math.}, 25(1):267--270, 2011.

\bibitem[Rus88]{MR936104}
F.~Ruskey.
\newblock Adjacent interchange generation of combinations.
\newblock {\em J. Algorithms}, 9(2):162--180, 1988.

\bibitem[Sav97]{MR1491049}
C.~D. Savage.
\newblock A survey of combinatorial {G}ray codes.
\newblock {\em SIAM Rev.}, 39(4):605--629, 1997.

\bibitem[SE99]{MR1725126}
M.~Schwartz and T.~Etzion.
\newblock The structure of single-track {G}ray codes.
\newblock {\em IEEE Trans. Inform. Theory}, 45(7):2383--2396, 1999.

\bibitem[SSS09]{MR2548541}
I.~Shields, B.~Shields, and C.~D. Savage.
\newblock An update on the middle levels problem.
\newblock {\em Discrete Math.}, 309(17):5271--5277, 2009.

\bibitem[Sta92]{MR1157583}
G.~Stachowiak.
\newblock Hamilton paths in graphs of linear extensions for unions of posets.
\newblock {\em SIAM J. Discrete Math.}, 5(2):199--206, 1992.

\bibitem[Tch82]{MR683982}
M.~Tchuente.
\newblock Generation of permutations by graphical exchanges.
\newblock {\em Ars Combin.}, 14:115--122, 1982.

\bibitem[TL73]{MR0349274}
D.~Tang and C.~Liu.
\newblock Distance-{$2$} cyclic chaining of constant-weight codes.
\newblock {\em IEEE Trans. Comput.}, C-22:176--180, 1973.

\bibitem[Ued00]{MR1761724}
T.~Ueda.
\newblock Gray codes for necklaces.
\newblock {\em Discrete Math.}, 219(1-3):235--248, 2000.

\bibitem[Wat70]{MR266804}
M.~E. Watkins.
\newblock Connectivity of transitive graphs.
\newblock {\em J. Combinatorial Theory}, 8:23--29, 1970.

\bibitem[WG84]{MR762322}
D.~Witte and J.~A. Gallian.
\newblock A survey: {H}amiltonian cycles in {C}ayley graphs.
\newblock {\em Discrete Math.}, 51(3):293--304, 1984.

\bibitem[WS96]{MR1413286}
T.~M.~Y. Wang and Carla~D. Savage.
\newblock A {G}ray code for necklaces of fixed density.
\newblock {\em SIAM J. Discrete Math.}, 9(4):654--673, 1996.

\end{thebibliography}

\appendix

\section{Proof of Lemma~\ref{lem:landau-prop}}

\begin{proof}[Proof of Lemma~\ref{lem:landau-prop}]
For any finite sequence~$\ba$ of integers, we write $S(\ba)$ for the sum of its entries.
Let $x$ be an entry of~$\ba$ that has the prime factorization $x=\Pi_{i=1}^\ell p_i^{e_i}$.
Defining $s:=\sum_{i=1}^\ell p_i^{e_i}$, we clearly have $x\ge s$ and therefore $d:=x-s\ge 0$.
Let $\ba'$ be the sequence obtained from~$\ba$ by replacing the entry~$x$ by the subsequence~$(p_1^{e_1},\ldots,p_\ell^{e_\ell},1^d)$.
Observe that
\begin{equation}
\label{eq:lcmS}
\lcm(\ba)=\lcm(\ba') \text{ and } S(\ba)=S(\ba').
\end{equation}
Suppose that $\ba$ contains two prime powers~$x=p^a$ and~$x'=p^b$ with $1\le a\le b$, and let~$\ba'$ denote the sequence obtained from~$\ba$ by replacing~$x$ by the subsequence~$(p,1^d)$ where $d:={p^a-p}$.
Then we also have~\eqref{eq:lcmS}.
Similarly, let~$\ba'$ be the sequence obtained from~$\ba$ by replacing~$x$ by the subsequence~$1^x$.
Then we also have~\eqref{eq:lcmS}.

Applying these rules exhaustively shows that for any partition~$\ba$, there is a partition~$\ba'$ satisfying~\eqref{eq:lcmS} and the additional conditions stated in~(i), so the claim follows from~\eqref{eq:landau-lcm}.
Moreover, for any partition~$\ba$ into odd parts (i.e., 0 even parts), there is a partition~$\ba'$ satisfying~\eqref{eq:lcmS} and the additional conditions stated in~(ii), so the claim follows from~\eqref{eq:landau0}.
Lastly, for any partition~$\ba$ with $2,4,6,\ldots$ even parts, there is a partition~$\ba'$ satisfying~\eqref{eq:lcmS} and the additional conditions stated in~(iii), so the first part of the claim follows from~\eqref{eq:landau2}.

It remains to prove that for $n\ge 12$ the partition as in~(iii) has at least 4~parts.
For $n=12,\ldots,35$ this can be verified by direct computation of the relevant partitions~$\ba$ for which $\lcm(\ba)=\lambda_2(n)$.
For $n\ge 36$ we argue as follows:
Note that any partition~$\ba$ of~$n$ into 1,2 or 3 parts has $\lcm(\ba)\le f(n):=\max\{n,n^2/4,n^3/27\}=\Theta(n^3)$.
We now identify partitions~$\ba$ with $2,4,6,\ldots$ even parts and at least~4 parts in total that satisfy $\lcm(\ba)>f(n)$, and this construction is split into two cases.
First note that any partition~$\ba$ of $n-8$ can be turned into a partition~$\ba'$ of~$n$ with at $4,6,8,\ldots$ even parts and at least~5 parts in total by inserting either $2,2,2,2$ or $2,2,2,1,1$.
It follows that $\lambda_2(n)\ge \lambda(n-8)$, and $\lambda(n-8)>f(n)$ holds for $n=36,\ldots,675$, which can be shown by using the values of~$\lambda(n)$ tabulated on the OEIS (sequence~A000793).
For $n\ge 676$ we argue as follows:
There is a partition~$\ba=(p_1,p_2,p_3,p_4,2,2,1^{n-4-p_1-p_2-p_3-p_4})$ of~$n$ with two even parts and at least~6 parts in total with $\lcm(\ba)=2 p_1 p_2 p_3 p_4$ obtained as follows:
We define $a_0:=\lfloor (n-4)/2\rfloor$ and for $i=1,\ldots,4$ we define $a_i:=\lfloor a_{i-1}/2\rfloor$ and let $p_i$ be a prime number in the interval~$\left]a_i,a_{i-1}\right]$, which exists by Bertrand's postulate.
Clearly, we have $p_1+p_2+p_3+p_4\le (n-4)(1/2+1/4+1/8+1/16)\le n-4$ and $p_1 p_2 p_3 p_4\ge a_1 a_2 a_3 a_4=:g(n)=\Theta(n^4)$, and we have $g(n)>f(n)$ for all $n\ge 676$ (we also have $p_4\ge a_4>2$ for those~$n$).

We now prove~(iv).
Let $\ba$ be a partition of~$n$ into powers of distinct primes and~1s such that $\lcm(\ba)=\lambda(n)$.
If all entries of~$\ba$ are odd, then we have $\lambda_0(n)=\lambda(n)$.
If $\ba$ contains~2, then replacing~$2$ by $1,1$ yields a partition~$\ba'$ into powers of distinct odd primes and~1s that satisfies $\lcm(\ba')=\lcm(\ba)/2$, which shows that $\lambda_0(n)\ge \lambda(n)/2$.
If $\ba$ contains~$2^c$ for some $c\ge 2$, then replacing~$2^c$ by $2^{c-1},2^{c-1}$ yields a partition~$\ba'$ into powers of distinct odd primes, two powers of~2 and~1s that satisfies $\lcm(\ba')=\lcm(\ba)/2$, which shows that $\lambda_2(n)\ge \lambda(n)/2$.
Combining these three observations proves~(iv).

We now prove~(v).
Consider a partition~$\ba$ of~$n$ into powers of distinct odd primes and~1s such that $\lcm(\ba)=\lambda_0(n)$.
If $\ba$ contains a power of~3, then let $a\geq 1$ be the exponent of~3 in~$\ba$, otherwise let~$a:=0$.
Furthermore, let $p^e$ be the largest prime power in~$\ba$.
We consider two ways of modifying~$\ba$.
If $a\geq 2$, we let $c$ be such that $3^a/3\le 2^c\le 2\cdot 3^a/3$, and then replacing $3^a$ in~$\ba$ by $3^{a-1}$, $2^c$ and~1s yields a partition~$\ba'$ that satisfies $\lcm(\ba')=\lcm(\ba)/3\cdot 2^c\geq 3^{a-2}\lcm(\ba)$, which proves that
\begin{equation}
\label{eq:lambda1}
\lambda(n)\geq 3^{a-2}\cdot \lambda_0(n).
\end{equation}
Note that this inequality holds trivially also for $a=0$ and $a=1$.
Now let $b$ and $c$ be such that $p^e/6\le 3^b\le p^e/2$ and $p^e/4\le 2^c\le p^e/2$.
Replacing $p^e$ in~$\ba$ by $3^b$, $2^c$ and~1s yields a partition~$\ba'$ that satisfies $\lcm(\ba')=\lcm(\ba)/p^e\cdot 3^{b-a}\cdot 2^c\geq p^e 3^{-a}/24\cdot \lcm(\ba)$, which proves that
\begin{equation}
\label{eq:lambda2}
\lambda(n)\geq p^e 3^{-a}/24\cdot\lambda_0(n).
\end{equation}
Multiplying the inequalities~\eqref{eq:lambda1} and~\ref{eq:lambda2} and taking the square root yields
\begin{equation*}
\lambda(n)\geq \frac{1}{6\sqrt{6}} p^{e/2}\cdot \lambda_0(n).
\end{equation*}
As $p^e\rightarrow\infty$ as $n\rightarrow\infty$, we obtain $\lim_{n\rightarrow \infty}\lambda(n)/\lambda_0(n)=+\infty$.
In order to show that $\lambda(n)/\lambda_0(n)\geq 4$ for $n\geq 739$, we consider the inequality $\frac{1}{6\sqrt{6}} p^{e/2}\geq 4$, which can be rearranged to $p^e\geq 3456=:q$.
We consider the set $\{p_2,p_3,\ldots,p_{482}\}=\{3,5,\ldots,3449\}$ of all odd primes with value at most~$q$, and for $i=2,\ldots,482$ we let $e_i\geq 1$ be the largest exponent so that $p_i^{e_i}\leq q$.
For example, we have $e_2=7$, as $3^7\leq q$ but $3^8>q$, and $e_3=5$, as $5^5\leq q$ but $5^6>q$.
We also consider the first prime~$p_{483}=3457$ that is larger than~$q$.
Clearly, if $n\geq p_{483}+\sum_{i=2}^{482} p_i^{e_i}=788670=:s$, then for any partition~$\ba$ of~$n$ into powers of distinct odd primes and~1s such that $\lcm(\ba)=\lambda_0(n)$ we must have $p^e\geq q$.
Indeed, if we restrict the partition to using only distinct powers of primes that are at most~$q$, then the partition would contain more than $p_{483}$ many~1s, and those could be replaced by the prime~$p_{483}$, yielding a better partition, which is a contradiction.
We now show that $\lambda(n)/\lambda_0(n)\geq 4$ for $n=739,\ldots,s-1$, and this is done in two steps.
Let $\ba$ be a partition of~$n$ into powers of distinct primes and~1s such that $\lcm(\ba)=\lambda(n)$.
If $\ba$ contains~$2^c$ for some $c\ge 1$, then we have
\begin{equation}
\label{eq:lambda0-lambda}
\lambda_0(n-2^c)=\lambda(n)/2^c.
\end{equation}
Using one of the algorithms from~\cite{MR0253514,MR2523311}, we can compute~$\lambda(n)$ for all $n<s$.
From~\eqref{eq:lambda0-lambda} we can then compute exact values for~$\lambda_0(n)$ for about half of all values of~$n<s$, and for the remaining ones we can compute upper bounds via~$\lambda_0(n)\leq \lambda_0(n+1)$.
Doing this shows that $\lambda(n)/\lambda_0(n)\geq 4$ for all $n=4507,\ldots,s-1$.
For the remaining values $n=739,\ldots,4506$ we can use a simple dynamic program to compute~$\lambda_0(n)$ exactly, verifying that $\lambda(n)/\lambda_0(n)\geq 4$.

Consider a partition~$\ba$ of~$n$ into powers of distinct primes and~1s such that $\lcm(\ba)=\lambda(n)$.
By a similar replacement argument as before, we can argue that for all large enough~$n$ the partition~$\ba$ must contain~$2^c$ for some~$c\geq 2$.
Replacing $2^c$, by $2^{c-1}$, $2$ and 1s, yields a partition~$\ba'$ with two even parts that satisfies $\lcm(\ba')=\lcm(\ba)/2$.
We obtain that $\lambda_2(n)\geq \lambda(n)/2$, i.e., $\lambda(n)/\lambda_2(n)\leq 2$ for all large enough~$n$.
This argument can be made effective for $n\geq 78$ (again by verifying a finite number of small cases with computer help).
For the remaining cases $n=18,\ldots,77$ the inequality $\lambda(n)/\lambda_2(n)\leq 2$ can be verified by direct computation.
This proves the first part of~(v).
The second part of~(v) is an immediate consequence of the first part.

It remains to prove~(vi).
Nicolas~\cite{MR254130} showed that there are arbitrarily long intervals where~$\lambda(n)$ is constant.
Let $[n,n+\ell]$ be such an interval of length~$\ell+1\ge 3$, i.e., we have $\lambda(n)=\lambda(n+i)$ for all $i=0,\ldots,\ell$.
All numbers~$\lambda(n+i)$ have the same prime factorization, so the corresponding partition~$\ba_i$ of~$n+i$ into distinct prime powers and~1s with $\lcm(\ba_i)=\lambda(n+i)$ and $S(\ba)=n+i$ has at least~$i$ many~1s.
Moreover, this partition must contain a positive power of~2, otherwise we could replace $1,1$ in $\ba_2$ by~2, yielding a partition~$\ba_2'$ with $\lcm(\ba_2')>\lcm(\ba_2)$, which is impossible.
Consequently, replacing $1,1$ in~$\ba_i$ by~2 for $i=2,\ldots,\ell$ yields a partition~$\ba_i'$ with two even parts satisfying $\lcm(\ba_i')=\lcm(\ba_i)$ and therefore $\lambda_2(n+i)=\lambda(n+i)$ holds for all $i=2,\ldots,\ell$.

This completes the proof of the lemma.
\end{proof}

\section{Proof of Lemma~\ref{lem:q-exists}}

To prove Lemma~\ref{lem:q-exists}, we need the following notation.
For integers $a\le b$, we write $[a,b]:=\{a,a+1,\ldots,b\}$.
Furthermore, we write $p_i$ for the $i$th prime number, i.e., $p_1=2$, $p_2=3$, $p_3=5$ etc.
We use the following number-theoretic properties about the integers sieved by the first $r$ primes.

\begin{lemma}
\label{lem:sieve}
For any $r\ge 1$ consider the set $S_r:=\{p_1,\ldots,p_r\}$ of the first $r$ prime numbers.
The numbers from~$\mathbb{Z}$ that are divisible by at least one of the primes from~$S_r$ form maximal intervals whose lengths are in~$\{p_i-2\mid i=2,\ldots,r+1\}$, and the intervals of maximum length $p_{r+1}-2$ are exactly $I_c:=[2,(p_{r+1}-1)]+c\cdot p_1\cdots p_r$ and $I_c':=[-(p_{r+1}-1),-2]+c\cdot p_1\cdots p_r$ for any $c\in\mathbb{Z}$.

Furthermore, for any prime number~$p_{r+1}$ we have for all $j=0,\ldots,p_{r+1}-p_r-2$ that every prime factor of~$j+2$ is also a factor of~$p_{r+1}-1-j$.
\end{lemma}

Lemma~\ref{lem:sieve} can be proved by straightforward induction on~$r$; we omit the details.

\begin{proof}[Proof of Lemma~\ref{lem:q-exists}]
The proof uses the following two basic number-theoretic facts about integers~$a,b\in\mathbb{Z}$ and $r\geq 1$:
\begin{enumerate}[label=(\roman*),leftmargin=8mm, topsep=0mm, noitemsep]
\item $a$ and $b$ are coprime if and only if $a$ and $a-b$ are coprime.
\item if $|b|<p_{r+1}$, then $a$ and $b$ are coprime if and only if $a+p_1\cdots p_r$ and $b$ are coprime.
\end{enumerate}

From~(i) we see that $q$ and~$\ell$ are coprime for all $\ell=k-(n-q),\ldots,k$ if and only if $q$ and $\ell'$ are coprime for all $\ell'=k'-(n-q),\ldots,k'$ where $k':=n-k$.
Consequently, it suffices to prove the lemma for $n\ge 2k$.

We argue by induction on~$k$.
If $n$ and $k$ are coprime, then the integer $q:=n$ satisfies the conditions of the lemma.
In particular, the lemma holds whenever $k=1$, which settles the base case for the induction.
For the induction step we consider some fixed value of~$k\ge 2$, and we prove that the lemma holds for this fixed~$k$ and all~$n\ge 2k$.

Let $r$ be such that $p_r\le k<p_{r+1}$, and define $S_r:=\{p_1,\ldots,p_r\}$ and $d:=p_{r+1}-k$.
We consider all $k$ numbers $q_i=n-i$, where $i=0,\ldots,k-1$, which satisfy the requirement $q_i=n-i>\max\{k,n-k\}=n-k$, and we need to show that for some $i\in\{0,\ldots,k-1\}$, the number~$q_i$ is coprime to all numbers in the set~$L_i:=[k-(n-q_i),k]=[k-i,k]$ (note that $|L_i|=i+1$).

If $q_i$ is not coprime some number in~$L_i$, then as $\max L_i=k<p_{r+1}$, the number~$q_i$ is divisible by a prime from~$S_r$.
We have to rule out that this happens to~$q_i$ for all $i=0,\ldots,k-1$ simultaneously.
For the sake of contradiction suppose that it does happen, then we have found an interval~$[q_{k-1},q_0]$ of $k$ integers that are all divisible by at least one of the primes from~$S_r$.
By the first part of Lemma~\ref{lem:sieve}, the possible intervals of such numbers have lengths in~$\{p_i-2\mid i=2,\ldots,r+1\}$, and all those lengths except~$p_{r+1}-2$ are strictly less than~$k$ (since $k\geq p_r$).
Furthermore, if $d=1$, i.e., $k=p_{r+1}-1$, then we also have $p_{r+1}-2<k$, i.e., we arrived at a contradiction.

It remains to investigate the situation where the interval~$[q_{k-1},q_0]$ is contained in an interval of length~$p_{r+1}-2$ and $d\geq 2$.
By the first part of Lemma~\ref{lem:sieve} this must be an interval~$I_c$ or~$I_c'$ as defined in the lemma for some $c\in\mathbb{Z}$.
If the interval~$[q_{k-1},q_0]$ is contained in~$I_c$, then $q_0=\max I_c-j$ for some $j\in\{0,\ldots,d-2\}$.
For the following investigations about coprimality, we can set $c:=0$ by~(ii).
Using the definition of~$I_0$, we obtain
\begin{equation}
\label{eq:qi}
\begin{alignedat}{2}
q_0 &= p_{r+1}-1-j=n, \\
q_1 &= p_{r+1}-1-(j+1), \\
&\vdots \\
q_{d-2-j} &= p_{r+1}-1-(j+(d-2-j)) = p_{r+1}-1-(d-2)=k+1.
\end{alignedat}
\end{equation}
The corresponding sets~$L_i$ are
\begin{equation}
\label{eq:Li}
\begin{alignedat}{2}
L_0&=[k,k]&&=[p_{r+1}-d,p_{r+1}-d], \\
L_1&=[k-1,k]&&=[p_{r+1}-d-1,p_{r+1}-d], \\
&\vdots \\
L_{d-2-j}&=[k-(d-2-j),k]&&=[p_{r+1}-d-(d-2-j),p_{r+1}-d].
\end{alignedat}
\end{equation}
We aim to show that there is a $q_i$ coprime to all elements of~$L_i$ for some $i\in\{0,\ldots,d-2-j\}$.
To prove this, we use~(i) and consider the sets $L_i':=q_i-L_i$, which are
\begin{align*}
L_0'&=[d-j-1,d-j-1], \\
L_1'&=[d-j-2,d-j-1], \\
&\vdots \\
L_{d-2-j}'&=[1,d-j-1].
\end{align*}
The fact that one of the~$q_i$, $i\in\{0,\ldots,d-2-j\}$, is coprime to all elements of~$L_i'$ now follows by induction for $k':=d-j-1$, observing that
\begin{equation*}
k'=d-j-1\le d-1=p_{r+1}-k-1\le p_{r+1}-p_r-1\le 2p_r-p_r-1=p_r-1<k,
\end{equation*}
where the estimate $p_{r+1}\le 2p_r$ follows from Bertrand's postulate.

It remains to consider the case that~$[q_{k-1},q_0]$ is contained in the interval~$I_c'$.
In this case we have $q_0=\max I_c'-j$ for some $j\in\{0,\ldots,d-2\}$, and again we can set $c:=0$ by~(ii).
The sets~$L_i$ are as in~\eqref{eq:Li}, and using the definition of~$I_0'$, the numbers~$q_i=:q_i'$ are now
\begin{equation}
\label{eq:qi2}
\begin{alignedat}{2}
q_0'=q_0&= -2-j=-(j+2), \\
q_1'=q_1&= -3-j=-(j+3), \\
&\vdots \\
q_{d-2-j}'=q_{d-2-j}&=(-2-(d-2-j))-j =-d, \\
\end{alignedat}
\end{equation}
Comparing~\eqref{eq:qi} and~\eqref{eq:qi2}, we see that we can apply the second part of Lemma~\ref{lem:sieve} (as $d-2\leq p_{r+1}-p_r-2$), which shows that every prime factor of~$q_i'$ is also a factor of~$q_i$.
We have already demonstrated that one of the~$q_i$ in~\eqref{eq:qi} is coprime to all numbers in~$L_i$, so this implies that~$q_i'$ is also coprime to all numbers in~$L_i$.
This completes the proof of the lemma.
\end{proof}

\begin{figure}[b!]
\includegraphics{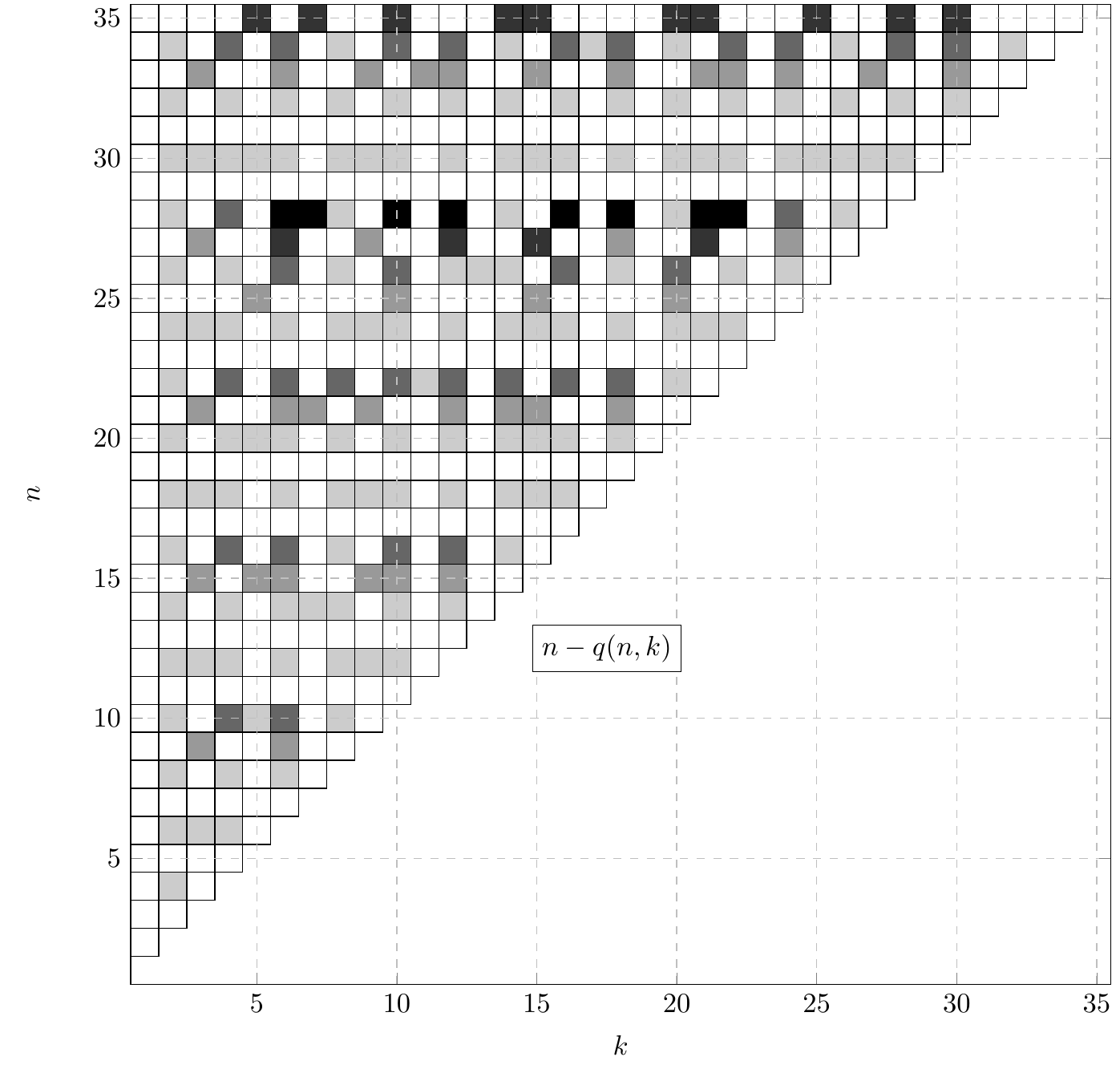}
\caption{The function~$n-q(n,k)$, which has values $0,\ldots,5$ in the depicted range, shaded by corresponding gray values (0=white, 5=black).}
\label{fig:qnk}
\end{figure}

Figure~\ref{fig:qnk} shows the function $n-q(n,k)$, with $q(n,k)$ as defined after Lemma~\ref{lem:q-exists}.
The white squares correspond to $n-q(n,k)=0$, i.e., these are the cases when $n$ and $k$ are coprime, yielding optimal compression.

\section{\texorpdfstring{Values of $\lambda(n)$, $\lambda_0(n)$ and $\lambda_2(n)$ for $n=1,\ldots,140$}{Values of lambda(n), lambda0(n) and lambda2(n) for n=1,...,140}}
\label{sec:landau02}

\definecolor{lightyellow}{rgb}{0.98,0.96,0.5}
\definecolor{lightorange}{rgb}{0.99,0.79,0.61}
\definecolor{lightgreen}{rgb}{0.61,0.99,0.44}

\pgfdeclarelayer{background}    
\pgfsetlayers{background,main}

\begin{table}
\setlength{\tabcolsep}{4pt}
\def\arraystretch{1.0}
\tiny
\makebox[0cm]{ 
\begin{tabular}{ccc}
\begin{tabular}{rrrrlll}
$n$ & $\lambda(n)$ & $\lambda_0(n)$ & $\lambda_2(n)$ & $\frac{\lambda(n)}{\lambda_0(n)}$ & $\frac{\lambda(n)}{\lambda_2(n)}$ & $\frac{2\lambda_0(n)}{\lambda_2(n)}$ \\ \hline
1 & 1 & 1 & $-\infty$ & 1 \cellcolor{lightorange} & -- & -- \\
2 & 2 & 1 & $-\infty$ & 2 & -- & -- \\
3 & 3 & 3 & $-\infty$ & 1 \cellcolor{lightorange} & -- & -- \\
4 & 4 & 3 & 2 & 1.33.. & 2 & 3 \\
5 & 6 & 5 & 2 & 1.2 & 3 & 5 \\
6 & 6 & 5 & 4 & 1.2 & 1.5 & 2.5 \\
7 & 12 & 7 & 6 & 1.71.. & 2 & 2.33.. \\
8 & 15 & 15 & 6 & 1 \cellcolor{lightorange} & 2.5 & 5 \\
9 & 20 & 15 & 12 & 1.33.. & 1.66.. & 2.5 \\
10 & 30 & 21 & 12 & 1.42.. & 2.5 & 3.5 \\
11 & 30 & 21 & 20 & 1.42.. & 1.5 & 2.1 \\
12 & 60 & 35 & 30 & 1.71.. & 2 & 2.33.. \\
13 & 60 & 35 & 30 & 1.71.. & 2 & 2.33.. \\
14 & 84 & 45 & 60 & 1.86.. & 1.4 & 1.5 \\
15 & 105 & \cellcolor{lightgreen} 105 & 60 & 1 \cellcolor{lightorange} & 1.75 & 3.5 \\
16 & 140 & 105 & 84 & 1.33.. & 1.66.. & 2.5 \\
17 & 210 & 105 & 84 & 2 & 2.5 & 2.5 \\
18 & 210 & 105 & 140 & 2 & 1.5 & 1.5 \\
19 & 420 & 165 & 210 & 2.54.. & 2 & 1.57.. \\
20 & 420 & 165 & 210 & 2.54.. & 2 & 1.57.. \\
21 & 420 & 315 & 420 & 1.33.. & 1 \cellcolor{lightorange}& 1.5 \\
22 & 420 & 315 & \cellcolor{lightgreen} 420 & 1.33.. & 1 \cellcolor{lightorange}& 1.5 \\
23 & 840 & 385 & 420 & 2.18.. & 2 & 1.83.. \\
24 & 840 & 385 & 420 & 2.18.. & 2 & 1.83.. \\
25 & 1,260 & 495 & 840 & 2.54.. & 1.5 & 1.17.. \\
26 & 1,260 & 1,155 & 840 & 1.09.. & 1.5 & 2.75 \\
27 & 1,540 & 1,155 & 1,260 & 1.33.. & 1.22.. & 1.83.. \\
28 & 2,310 & 1,365 & 1,260 & 1.69.. & 1.83.. & 2.16.. \\
29 & 2,520 & 1,365 & 1,540 & 1.84.. & 1.63.. & 1.77.. \\
30 & 4,620 & 1,365 & 2,310 & 3.38.. & 2 & 1.18.. \\
31 & 4,620 & 1,365 & 2,520 & 3.38.. & 1.83.. & 1.08.. \\
32 & 5,460 & 3,465 & 4,620 & 1.57.. & 1.18.. & 1.5 \\
33 & 5,460 & 3,465 & 4,620 & 1.57.. & 1.18.. & 1.5 \\
34 & 9,240 & 4,095 & 5,460 & 2.25.. & 1.69.. & 1.5 \\
35 & 9,240 & 4,095 & 5,460 & 2.25.. & 1.69.. & 1.5 \\
36 & 13,860 & 5,005 & 9,240 & 2.76.. & 1.5 & 1.08.. \\
37 & 13,860 & 5,005 & 9,240 & 2.76.. & 1.5 & 1.08.. \\
38 & 16,380 & 6,435 & 13,860 & 2.54.. & 1.18.. & 0.92.. \cellcolor{lightyellow} \\
39 & 16,380 & 15,015 & 13,860 & 1.09.. & 1.18.. & 2.16.. \\
40 & 27,720 & 15,015 & 16,380 & 1.84.. & 1.69.. & 1.83.. \\
41 & 30,030 & 15,015 & 16,380 & 2 & 1.83.. & 1.83.. \\
42 & 32,760 & 15,015 & 27,720 & 2.18.. & 1.18.. & 1.08.. \\
43 & 60,060 & 19,635 & 30,030 & 3.05.. & 2 & 1.30.. \\
44 & 60,060 & 19,635 & 32,760 & 3.05.. & 1.83.. & 1.19.. \\
45 & 60,060 & 45,045 & 60,060 & 1.33.. & 1 \cellcolor{lightorange}& 1.5 \\
46 & 60,060 & 45,045 & \cellcolor{lightgreen} 60,060 & 1.33.. & 1 \cellcolor{lightorange}& 1.5 \\
47 & 120,120 & 45,045 & 60,060 & 2.66.. & 2 & 1.5 \\
48 & 120,120 & 45,045 & 60,060 & 2.66.. & 2 & 1.5 \\
49 & 180,180 & 58,905 & \cellcolor{lightgreen} 120,120 & 3.05.. & 1.5 & 0.98.. \cellcolor{lightyellow} \\
50 & 180,180 & 58,905 & 120,120 & 3.05.. & 1.5 & 0.98.. \cellcolor{lightyellow} \\
51 & 180,180 & 69,615 & \cellcolor{lightgreen} 180,180 & 2.58.. & 1 \cellcolor{lightorange}& 0.77.. \cellcolor{lightyellow} \\
52 & 180,180 & 69,615 & \cellcolor{lightgreen} 180,180 & 2.58.. & 1 \cellcolor{lightorange}& 0.77.. \cellcolor{lightyellow} \\
53 & 360,360 & 85,085 & \cellcolor{lightgreen} 180,180 & 4.23.. & 2 & 0.94.. \cellcolor{lightyellow} \\
54 & 360,360 & 85,085 & 180,180 & 4.23.. & 2 & 0.94.. \cellcolor{lightyellow} \\
55 & 360,360 & 109,395 & \cellcolor{lightgreen} 360,360 & 3.29.. & 1 \cellcolor{lightorange}& 0.60.. \cellcolor{lightyellow} \\
56 & 360,360 & 255,255 & 360,360 & 1.41.. & 1 \cellcolor{lightorange}& 1.41.. \\
57 & 471,240 & 255,255 & 360,360 & 1.84.. & 1.30.. & 1.41.. \\
58 & 510,510 & 285,285 & 360,360 & 1.78.. & 1.41.. & 1.58.. \\
59 & 556,920 & 285,285 & 471,240 & 1.95.. & 1.18.. & 1.21.. \\
60 & 1,021,020 & 285,285 & 510,510 & 3.57.. & 2 & 1.11.. \\
61 & 1,021,020 & 285,285 & 556,920 & 3.57.. & 1.83.. & 1.02.. \\
62 & 1,141,140 & 765,765 & 1,021,020 & 1.49.. & 1.11.. & 1.5 \\
63 & 1,141,140 & 765,765 & 1,021,020 & 1.49.. & 1.11.. & 1.5 \\
64 & 2,042,040 & 855,855 & 1,141,140 & 2.38.. & 1.78.. & 1.5 \\
65 & 2,042,040 & 855,855 & 1,141,140 & 2.38.. & 1.78.. & 1.5 \\
66 & 3,063,060 & 855,855 & 2,042,040 & 3.57.. & 1.5 & 0.83.. \cellcolor{lightyellow} \\
67 & 3,063,060 & 855,855 & 2,042,040 & 3.57.. & 1.5 & 0.83.. \cellcolor{lightyellow} \\
68 & 3,423,420 & 1,119,195 & \cellcolor{lightgreen} 3,063,060 & 3.05.. & 1.11.. & 0.73.. \cellcolor{lightyellow} \\
69 & 3,423,420 & 1,119,195 & \cellcolor{lightgreen} 3,063,060 & 3.05.. & 1.11.. & 0.73.. \cellcolor{lightyellow} \\
70 & 6,126,120 & 1,322,685 & 3,423,420 & 4.63.. & 1.78.. & 0.77.. \cellcolor{lightyellow} \\
\end{tabular}
& &
\begin{tabular}{crrrlll}
$n$ & $\lambda(n)$ & $\lambda_0(n)$ & $\lambda_2(n)$ & $\frac{\lambda(n)}{\lambda_0(n)}$ & $\frac{\lambda(n)}{\lambda_2(n)}$ & $\frac{2\lambda_0(n)}{\lambda_2(n)}$ \\ \hline
71 & 6,126,120 & 1,322,685 & 3,423,420 & 4.63.. & 1.78.. & 0.77.. \cellcolor{lightyellow} \\
72 & 6,846,840 & 1,616,615 & \cellcolor{lightgreen} 6,126,120 & 4.23.. & 1.11.. & 0.52.. \cellcolor{lightyellow} \\
73 & 6,846,840 & 1,616,615 & \cellcolor{lightgreen} 6,126,120 & 4.23.. & 1.11.. & 0.52.. \cellcolor{lightyellow} \\
74 & 6,846,840 & 2,078,505 & \cellcolor{lightgreen} 6,846,840 & 3.29.. & 1 \cellcolor{lightorange}& 0.60.. \cellcolor{lightyellow} \\
75 & 6,846,840 & 4,849,845 & \cellcolor{lightgreen} 6,846,840 & 1.41.. & 1 \cellcolor{lightorange}& 1.41.. \\
76 & 8,953,560 & 4,849,845 & 6,846,840 & 1.84.. & 1.30.. & 1.41.. \\
77 & 9,699,690 & 4,849,845 & 6,846,840 & 2 & 1.41.. & 1.41.. \\
78 & 12,252,240 & 4,849,845 & 8,953,560 & 2.52.. & 1.36.. & 1.08.. \\
79 & 19,399,380 & 5,870,865 & 9,699,690 & 3.30.. & 2 & 1.21.. \\
80 & 19,399,380 & 5,870,865 & \cellcolor{lightgreen} 12,252,240 & 3.30.. & 1.58.. & 0.95.. \cellcolor{lightyellow} \\
81 & 19,399,380 & 14,549,535 & 19,399,380 & 1.33.. & 1 \cellcolor{lightorange}& 1.5 \\
82 & 19,399,380 & 14,549,535 & \cellcolor{lightgreen} 19,399,380 & 1.33.. & 1 \cellcolor{lightorange}& 1.5 \\
83 & 38,798,760 & 14,549,535 & 19,399,380 & 2.66.. & 2 & 1.5 \\
84 & 38,798,760 & 14,549,535 & 19,399,380 & 2.66.. & 2 & 1.5 \\
85 & 58,198,140 & 17,612,595 & \cellcolor{lightgreen} 38,798,760 & 3.30.. & 1.5 & 0.90.. \cellcolor{lightyellow} \\
86 & 58,198,140 & 17,612,595 & 38,798,760 & 3.30.. & 1.5 & 0.90.. \cellcolor{lightyellow} \\
87 & 58,198,140 & 19,684,665 & \cellcolor{lightgreen} 58,198,140 & 2.95.. & 1 \cellcolor{lightorange}& 0.67.. \cellcolor{lightyellow} \\
88 & 58,198,140 & 19,684,665 & \cellcolor{lightgreen} 58,198,140 & 2.95.. & 1 \cellcolor{lightorange}& 0.67.. \cellcolor{lightyellow} \\
89 & 116,396,280 & 19,684,665 & \cellcolor{lightgreen} 58,198,140 & 5.91.. & 2 & 0.67.. \cellcolor{lightyellow} \\
90 & 116,396,280 & 19,684,665 & 58,198,140 & 5.91.. & 2 & 0.67.. \cellcolor{lightyellow} \\
91 & 116,396,280 & 25,741,485 & \cellcolor{lightgreen} 116,396,280 & 4.52.. & 1 \cellcolor{lightorange}& 0.44.. \cellcolor{lightyellow} \\
92 & 116,396,280 & 25,741,485 & \cellcolor{lightgreen} 116,396,280 & 4.52.. & 1 \cellcolor{lightorange}& 0.44.. \cellcolor{lightyellow} \\
93 & 140,900,760 & 30,421,755 & \cellcolor{lightgreen} 116,396,280 & 4.63.. & 1.21.. & 0.52.. \cellcolor{lightyellow} \\
94 & 140,900,760 & 30,421,755 & 116,396,280 & 4.63.. & 1.21.. & 0.52.. \cellcolor{lightyellow} \\
95 & 157,477,320 & 37,182,145 & 140,900,760 & 4.23.. & 1.11.. & 0.52.. \cellcolor{lightyellow} \\
96 & 157,477,320 & 37,182,145 & \cellcolor{lightgreen} 140,900,760 & 4.23.. & 1.11.. & 0.52.. \cellcolor{lightyellow} \\
97 & 232,792,560 & 47,805,615 & \cellcolor{lightgreen} 157,477,320 & 4.86.. & 1.47.. & 0.60.. \cellcolor{lightyellow} \\
98 & 232,792,560 & 111,546,435 & 157,477,320 & 2.08.. & 1.47.. & 1.41.. \\
99 & 232,792,560 & 111,546,435 & \cellcolor{lightgreen} 232,792,560 & 2.08.. & 1 \cellcolor{lightorange}& 0.95.. \cellcolor{lightyellow} \\
100 & 232,792,560 & 111,546,435 & \cellcolor{lightgreen} 232,792,560 & 2.08.. & 1 \cellcolor{lightorange}& 0.95.. \cellcolor{lightyellow} \\
101 & 281,801,520 & 111,546,435 & \cellcolor{lightgreen} 232,792,560 & 2.52.. & 1.21.. & 0.95.. \cellcolor{lightyellow} \\
102 & 446,185,740 & 111,546,435 & 232,792,560 & 4 & 1.91.. & 0.95.. \cellcolor{lightyellow} \\
103 & 446,185,740 & 111,546,435 & 281,801,520 & 4 & 1.58.. & 0.79.. \cellcolor{lightyellow} \\
104 & 446,185,740 & 334,639,305 & 446,185,740 & 1.33.. & 1 \cellcolor{lightorange}& 1.5 \\
105 & 446,185,740 & 334,639,305 & 446,185,740 & 1.33.. & 1 \cellcolor{lightorange}& 1.5 \\
106 & 892,371,480 & 334,639,305 & 446,185,740 & 2.66.. & 2 & 1.5 \\
107 & 892,371,480 & 334,639,305 & 446,185,740 & 2.66.. & 2 & 1.5 \\
108 & 1,338,557,220 & 334,639,305 & \cellcolor{lightgreen} 892,371,480 & 4 & 1.5 & 0.75 \cellcolor{lightyellow} \\
109 & 1,338,557,220 & 334,639,305 & \cellcolor{lightgreen} 892,371,480 & 4 & 1.5 & 0.75 \cellcolor{lightyellow} \\
110 & 1,338,557,220 & 421,936,515 & \cellcolor{lightgreen} 1,338,557,220 & 3.17.. & 1 \cellcolor{lightorange}& 0.63.. \cellcolor{lightyellow} \\
111 & 1,338,557,220 & 421,936,515 & \cellcolor{lightgreen} 1,338,557,220 & 3.17.. & 1 \cellcolor{lightorange}& 0.63.. \cellcolor{lightyellow} \\
112 & 2,677,114,440 & 451,035,585 & \cellcolor{lightgreen} 1,338,557,220 & 5.93.. & 2 & 0.67.. \cellcolor{lightyellow} \\
113 & 2,677,114,440 & 451,035,585 & \cellcolor{lightgreen} 1,338,557,220 & 5.93.. & 2 & 0.67.. \cellcolor{lightyellow} \\
114 & 2,677,114,440 & 510,765,255 & \cellcolor{lightgreen} 2,677,114,440 & 5.24.. & 1 \cellcolor{lightorange}& 0.38.. \cellcolor{lightyellow} \\
115 & 2,677,114,440  & 510,765,255 & \cellcolor{lightgreen} 2,677,114,440 & 5.24.. & 1 \cellcolor{lightorange}& 0.38.. \cellcolor{lightyellow} \\
116 & 2,677,114,440 & 570,855,285 & \cellcolor{lightgreen} 2,677,114,440 & 4.68.. & 1 \cellcolor{lightorange}& 0.42.. \cellcolor{lightyellow} \\
117 & 2,677,114,440 & 570,855,285 & \cellcolor{lightgreen} 2,677,114,440 & 4.68.. & 1 \cellcolor{lightorange}& 0.42.. \cellcolor{lightyellow} \\
118 & 3,375,492,120 & 610,224,615 & 2,677,114,440 & 5.53.. & 1.26.. & 0.45.. \cellcolor{lightyellow} \\
119 & 3,375,492,120 & 610,224,615 & 2,677,114,440 & 5.53.. & 1.26.. & 0.45.. \cellcolor{lightyellow} \\
120 & 5,354,228,880 & 746,503,065 & \cellcolor{lightgreen} 3,375,492,120 & 7.17.. & 1.58.. & 0.44.. \cellcolor{lightyellow} \\
121 & 5,354,228,880 & 746,503,065 & \cellcolor{lightgreen} 3,375,492,120 & 7.17.. & 1.58.. & 0.44.. \cellcolor{lightyellow} \\
122 & 5,354,228,880 & 1,003,917,915 & \cellcolor{lightgreen} 5,354,228,880 & 5.33.. & 1 \cellcolor{lightorange}& 0.37.. \cellcolor{lightyellow} \\
123 & 5,354,228,880 & 1,003,917,915 & \cellcolor{lightgreen} 5,354,228,880 & 5.33.. & 1 \cellcolor{lightorange}& 0.37.. \cellcolor{lightyellow} \\
124 & 5,354,228,880 & 1,673,196,525 & \cellcolor{lightgreen} 5,354,228,880 & 3.2 & 1 \cellcolor{lightorange}& 0.62.. \cellcolor{lightyellow} \\
125 & 5,354,228,880 & 1,673,196,525 & \cellcolor{lightgreen} 5,354,228,880 & 3.2 & 1 \cellcolor{lightorange}& 0.62.. \cellcolor{lightyellow} \\
126 & 6,750,984,240 & 1,673,196,525 & 5,354,228,880 & 4.03.. & 1.26.. & 0.62.. \cellcolor{lightyellow} \\
127 & 6,750,984,240 & 3,234,846,615 & 5,354,228,880 & 2.08.. & 1.26.. & 1.20.. \\
128 & 7,216,569,360 & 3,234,846,615 & \cellcolor{lightgreen} 6,750,984,240 & 2.23.. & 1.06.. & 0.95.. \cellcolor{lightyellow} \\
129 & 7,216,569,360 & 3,457,939,485 & 6,750,984,240 & 2.08.. & 1.06.. & 1.02.. \\
130 & 8,172,244,080 & 3,457,939,485 & 7,216,569,360 & 2.36.. & 1.13.. & 0.95.. \cellcolor{lightyellow} \\
131 & 12,939,386,460 & 3,457,939,485 & 7,216,569,360 & 3.74.. & 1.79.. & 0.95.. \cellcolor{lightyellow} \\
132 & 13,385,572,200 & 3,457,939,485 & \cellcolor{lightgreen} 8,172,244,080 & 3.87.. & 1.63.. & 0.84.. \cellcolor{lightyellow} \\
133 & 13,831,757,940 & 9,704,539,845 & 12,939,386,460 & 1.42.. & 1.06.. & 1.5 \\
134 & 13,831,757,940 & 9,704,539,845 & 13,385,572,200 & 1.42.. & 1.03.. & 1.45 \\
135 & 25,878,772,920 & 10,373,818,455 & 13,831,757,940 & 2.49.. & 1.87.. & 1.5 \\
136 & 25,878,772,920 & 10,373,818,455 & 13,831,757,940 & 2.49.. & 1.87.. & 1.5 \\
137 & 38,818,159,380 & 10,373,818,455 & \cellcolor{lightgreen} 25,878,772,920 & 3.74.. & 1.5 & 0.80.. \cellcolor{lightyellow} \\
138 & 38,818,159,380 & 10,373,818,455 & 25,878,772,920 & 3.74.. & 1.5 & 0.80.. \cellcolor{lightyellow} \\
139 & 41,495,273,820 & 10,373,818,455 & 38,818,159,380 & 4 & 1.06.. & 0.53.. \cellcolor{lightyellow} \\
140 & 41,495,273,820 & 10,373,818,455 & \cellcolor{lightgreen} 38,818,159,380 & 4 & 1.06.. & 0.53.. \cellcolor{lightyellow} \\
\end{tabular}
\end{tabular}
}
\caption{Values of $\lambda(n)$, $\lambda_0(n)$ and $\lambda_2(n)$ for $n=1,\ldots,140$, and their ratios.
Columns~3+4: Exact values $\kappa(\Pi_n)=\lambda_0(n)$ or $\kappa(\Pi_n)=\lambda_2(n)$, respectively, obtained from Theorem~\ref{thm:kappa-Pin}~(iv) are highlighted.
Columns~5+6: Ratios $\lambda(n)/\lambda_0(n)=1$ or $\lambda(n)/\lambda_2(n)=1$ are highlighted.
Column~7: Ratios $2\lambda_0(n)/\lambda_2(n)\le 1$ are highlighted.
}
\end{table}

\begin{tikzpicture}
\begin{axis}[
    xlabel={$n$},
    ylabel={$\lambda(n)/\lambda_0(n)$},
    xmin=0, xmax=140,
    ymin=0.8, ymax=7.5,
    xtick={0,10,20,30,40,50,60,70,80,90,100,110,120,130,140},
    ytick={1,2,3,4,5,6,7},
    ymajorgrids=true,
    xmajorgrids=true,
    grid style=dashed,
    width=14cm,
    height=6cm
]
\addplot[
    only marks,
    color=blue,
    mark=*,
    mark size=0.9pt
    ]
    coordinates {
    (1,1.)(2,2.)(3,1.)(4,1.33333)(5,1.2)(6,1.2)(7,1.71429)(8,1.)(9,1.33333)(10,1.42857)(11,1.42857)(12,1.71429)(13,1.71429)(14,1.86667)(15,1.)(16,1.33333)(17,2.)(18,2.)(19,2.54545)(20,2.54545)(21,1.33333)(22,1.33333)(23,2.18182)(24,2.18182)(25,2.54545)(26,1.09091)(27,1.33333)(28,1.69231)(29,1.84615)(30,3.38462)(31,3.38462)(32,1.57576)(33,1.57576)(34,2.25641)(35,2.25641)(36,2.76923)(37,2.76923)(38,2.54545)(39,1.09091)(40,1.84615)(41,2.)(42,2.18182)(43,3.05882)(44,3.05882)(45,1.33333)(46,1.33333)(47,2.66667)(48,2.66667)(49,3.05882)(50,3.05882)(51,2.58824)(52,2.58824)(53,4.23529)(54,4.23529)(55,3.29412)(56,1.41176)(57,1.84615)(58,1.78947)(59,1.95215)(60,3.57895)(61,3.57895)(62,1.4902)(63,1.4902)(64,2.38596)(65,2.38596)(66,3.57895)(67,3.57895)(68,3.05882)(69,3.05882)(70,4.63158)(71,4.63158)(72,4.23529)(73,4.23529)(74,3.29412)(75,1.41176)(76,1.84615)(77,2.)(78,2.52632)(79,3.30435)(80,3.30435)(81,1.33333)(82,1.33333)(83,2.66667)(84,2.66667)(85,3.30435)(86,3.30435)(87,2.95652)(88,2.95652)(89,5.91304)(90,5.91304)(91,4.52174)(92,4.52174)(93,4.63158)(94,4.63158)(95,4.23529)(96,4.23529)(97,4.86957)(98,2.08696)(99,2.08696)(100,2.08696)(101,2.52632)(102,4.)(103,4.)(104,1.33333)(105,1.33333)(106,2.66667)(107,2.66667)(108,4.)(109,4.)(110,3.17241)(111,3.17241)(112,5.93548)(113,5.93548)(114,5.24138)(115,5.24138)(116,4.68966)(117,4.68966)(118,5.53156)(119,5.53156)(120,7.17241)(121,7.17241)(122,5.33333)(123,5.33333)(124,3.2)(125,3.2)(126,4.03478)(127,2.08696)(128,2.23088)(129,2.08696)(130,2.36333)(131,3.74194)(132,3.87097)(133,1.42529)(134,1.42529)(135,2.49462)(136,2.49462)(137,3.74194)(138,3.74194)(139,4.)(140,4.)
    };
\begin{pgfonlayer}{background}
\fill[lightorange] (0,16) rectangle (140,24);
\end{pgfonlayer}
\end{axis}
\end{tikzpicture}

\begin{tikzpicture}
\begin{axis}[
    xlabel={$n$},
    ylabel={$\lambda(n)/\lambda_2(n)$},
    xmin=0, xmax=140,
    ymin=0.8, ymax=3.2,
    xtick={0,10,20,30,40,50,60,70,80,90,100,110,120,130,140},
    ytick={1,1.5,2,2.5,3},
    ymajorgrids=true,
    xmajorgrids=true,
    grid style=dashed,
    width=14cm,
    height=6cm
]
\addplot[
    only marks,
    color=blue,
    mark=*,
    mark size=0.9pt
    ]
    coordinates {
    (4,2.)(5,3.)(6,1.5)(7,2.)(8,2.5)(9,1.66667)(10,2.5)(11,1.5)(12,2.)(13,2.)(14,1.4)(15,1.75)(16,1.66667)(17,2.5)(18,1.5)(19,2.)(20,2.)(21,1.)(22,1.)(23,2.)(24,2.)(25,1.5)(26,1.5)(27,1.22222)(28,1.83333)(29,1.63636)(30,2.)(31,1.83333)(32,1.18182)(33,1.18182)(34,1.69231)(35,1.69231)(36,1.5)(37,1.5)(38,1.18182)(39,1.18182)(40,1.69231)(41,1.83333)(42,1.18182)(43,2.)(44,1.83333)(45,1.)(46,1.)(47,2.)(48,2.)(49,1.5)(50,1.5)(51,1.)(52,1.)(53,2.)(54,2.)(55,1.)(56,1.)(57,1.30769)(58,1.41667)(59,1.18182)(60,2.)(61,1.83333)(62,1.11765)(63,1.11765)(64,1.78947)(65,1.78947)(66,1.5)(67,1.5)(68,1.11765)(69,1.11765)(70,1.78947)(71,1.78947)(72,1.11765)(73,1.11765)(74,1.)(75,1.)(76,1.30769)(77,1.41667)(78,1.36842)(79,2.)(80,1.58333)(81,1.)(82,1.)(83,2.)(84,2.)(85,1.5)(86,1.5)(87,1.)(88,1.)(89,2.)(90,2.)(91,1.)(92,1.)(93,1.21053)(94,1.21053)(95,1.11765)(96,1.11765)(97,1.47826)(98,1.47826)(99,1.)(100,1.)(101,1.21053)(102,1.91667)(103,1.58333)(104,1.)(105,1.)(106,2.)(107,2.)(108,1.5)(109,1.5)(110,1.)(111,1.)(112,2.)(113,2.)(114,1.)(115,1.)(116,1.)(117,1.)(118,1.26087)(119,1.26087)(120,1.58621)(121,1.58621)(122,1.)(123,1.)(124,1.)(125,1.)(126,1.26087)(127,1.26087)(128,1.06897)(129,1.06897)(130,1.13243)(131,1.79301)(132,1.63793)(133,1.06897)(134,1.03333)(135,1.87097)(136,1.87097)(137,1.5)(138,1.5)(139,1.06897)(140,1.06897)
    };
\begin{pgfonlayer}{background}
\fill[lightorange] (0,18.5) rectangle (140,21.5);
\end{pgfonlayer}
\end{axis}
\end{tikzpicture}

\begin{tikzpicture}
\begin{axis}[
    xlabel={$n$},
    ylabel={$2\lambda_0(n)/\lambda_2(n)$},
    xmin=0, xmax=140,
    ymin=0, ymax=3.2,
    xtick={0,10,20,30,40,50,60,70,80,90,100,110,120,130,140},
    ytick={0,0.5,1,1.5,2,2.5,3},
    ymajorgrids=true,
    xmajorgrids=true,
    grid style=dashed,
    width=14cm,
    height=6cm
]
\addplot[
    only marks,
    color=blue,
    mark=*,
    mark size=0.9pt
    ]
    coordinates {
    (4,3.)(5,5.)(6,2.5)(7,2.33333)(8,5.)(9,2.5)(10,3.5)(11,2.1)(12,2.33333)(13,2.33333)(14,1.5)(15,3.5)(16,2.5)(17,2.5)(18,1.5)(19,1.57143)(20,1.57143)(21,1.5)(22,1.5)(23,1.83333)(24,1.83333)(25,1.17857)(26,2.75)(27,1.83333)(28,2.16667)(29,1.77273)(30,1.18182)(31,1.08333)(32,1.5)(33,1.5)(34,1.5)(35,1.5)(36,1.08333)(37,1.08333)(38,0.928571)(39,2.16667)(40,1.83333)(41,1.83333)(42,1.08333)(43,1.30769)(44,1.19872)(45,1.5)(46,1.5)(47,1.5)(48,1.5)(49,0.980769)(50,0.980769)(51,0.772727)(52,0.772727)(53,0.944444)(54,0.944444)(55,0.607143)(56,1.41667)(57,1.41667)(58,1.58333)(59,1.21078)(60,1.11765)(61,1.02451)(62,1.5)(63,1.5)(64,1.5)(65,1.5)(66,0.838235)(67,0.838235)(68,0.730769)(69,0.730769)(70,0.772727)(71,0.772727)(72,0.527778)(73,0.527778)(74,0.607143)(75,1.41667)(76,1.41667)(77,1.41667)(78,1.08333)(79,1.21053)(80,0.958333)(81,1.5)(82,1.5)(83,1.5)(84,1.5)(85,0.907895)(86,0.907895)(87,0.676471)(88,0.676471)(89,0.676471)(90,0.676471)(91,0.442308)(92,0.442308)(93,0.522727)(94,0.522727)(95,0.527778)(96,0.527778)(97,0.607143)(98,1.41667)(99,0.958333)(100,0.958333)(101,0.958333)(102,0.958333)(103,0.791667)(104,1.5)(105,1.5)(106,1.5)(107,1.5)(108,0.75)(109,0.75)(110,0.630435)(111,0.630435)(112,0.673913)(113,0.673913)(114,0.381579)(115,0.381579)(116,0.426471)(117,0.426471)(118,0.455882)(119,0.455882)(120,0.442308)(121,0.442308)(122,0.375)(123,0.375)(124,0.625)(125,0.625)(126,0.625)(127,1.20833)(128,0.958333)(129,1.02443)(130,0.958333)(131,0.958333)(132,0.846264)(133,1.5)(134,1.45)(135,1.5)(136,1.5)(137,0.801724)(138,0.801724)(139,0.534483)(140,0.534483)
    };
\begin{pgfonlayer}{background}
\fill[lightyellow] (0,0) rectangle (140,100);
\end{pgfonlayer}
\end{axis}
\end{tikzpicture}

\end{document}